\documentclass[11pt]{article}
\usepackage[round]{natbib}
\usepackage{amssymb}
\usepackage{amsmath}
\usepackage{fullpage}
\usepackage{amsthm}
\usepackage{graphicx}
\usepackage{xcolor}
\usepackage{algorithm}
\usepackage{algpseudocode}
\algnewcommand\algorithmicinput{\textbf{INPUT:}}
\algnewcommand\INPUT{\item[\algorithmicinput]}
\algnewcommand\algorithmicoutput{\textbf{OUTPUT:}}
\algnewcommand\OUTPUT{\item[\algorithmicoutput]}
\usepackage{hyperref}[]
\hypersetup{
    colorlinks=true,
    linkcolor=blue,
    filecolor=magenta,      
    urlcolor=cyan,
    citecolor=blue,
}
\usepackage{cleveref}
\usepackage{authblk}
\usepackage{subcaption}
\usepackage{float}
\usepackage{setspace}
\DeclareMathOperator*{\argmin}{arg\,min}
\DeclareMathOperator*{\argmax}{arg\,max}

\newcommand{\p}{\mathbb P}

\newcommand{\I}{{\mathcal I }} 
\newcommand{\T}{{\mathcal T }}

  \newcommand{\E}{\mathbb E}
  \newcommand{\he}{{\widehat \eta }}
  \newcommand{\ws}{{ \widetilde {\mathcal  S }}}
  \newcommand{\wm}{{ \widetilde {\mathcal M }}}

 \newcommand{\s}{ {\mathfrak{s} }}

\newcommand{\ot}{ { (0, t] }}
\newcommand{\tn}{ { (t,  n ] }}

\newtheorem{assumption}{Assumption}
\newtheorem{lemma}{Lemma}
\newtheorem{proposition}{Proposition}
\newtheorem{corollary}{Corollary}

\newtheorem{theorem}{Theorem}
\newtheorem{condition}{Condition}

\date{\vspace{-5ex}}

\title{Optimal Change-point Testing for High-dimensional Linear Models with Temporal Dependence}

\author {Zifeng Zhao$^{1}$}
\author {Xiaokai Luo$^{1}$}
\author  {Zongge Liu$^{2}$}
\author {Daren Wang$^{1}$}

\affil{  University of Notre Dame$^{1}$ and Apple$^{2}$}
 
\begin{document}
\setlength{\abovedisplayskip}{4pt}
\setlength{\belowdisplayskip}{4pt}
\setlength{\abovedisplayshortskip}{2.5pt}
\setlength{\belowdisplayshortskip}{2.5pt}	

\maketitle
	
\begin{abstract}
In this paper, we study change-point testing for high-dimensional linear models, an important problem that has not been well explored in the literature. Specifically, we propose a quadratic-form  cumulative sum (CUSUM)  statistic to test the stability of regression coefficients in high-dimensional linear models. The test controls type-I error at any desired level and is robust to temporally dependent observations. We establish its asymptotic distribution under   the null hypothesis, and demonstrate that it is asymptotically powerful against multiple change-point alternatives and achieves the optimal detection boundary for a wide class of high-dimensional models.
We further develop an adaptive procedure to estimate the tuning parameters of the test, making our method practical in applications. Additionally, we extend our approach to localize change-points  in the regression time series and establish sharp error bounds for our  change-point estimator. Extensive numerical experiments and a real data application in macroeconomics are conducted to demonstrate the promising performance and practical utility of the proposed test. 

\end{abstract}


\clearpage
\section{Introduction}
Over the last two decades, with the rise of ``big data'', high-dimensional linear models have become a key  component of modern statistical modeling, with tremendous applications across various  fields, including biology, neuroscience, climatology, finance, economics, and cybersecurity, and much more. A vast body of statistical literature explores the methodological, computational, and theoretical aspects of high-dimensional linear models. We refer the readers to \cite{Buehlmann2011} for an excellent book-length review.

On the other hand, one of the most commonly encountered issues for ``big data" is heterogeneity. For sequentially collected data, heterogeneity often manifests itself through non-stationarity, where the data-generating mechanism experiences structural breaks or change-points over time. To address such issues, numerous statistical testing procedures have been proposed to examine the existence of structural breaks.

Change-point testing is a classical problem in statistics and has been extensively studied in the low-dimensional setting, we refer the readers to \cite{Aue2009}, \cite{shao2010testing}, \cite{Matteson2014}, \cite{Kirch2015} and \cite{Zhang2018} (among many others) for some recent work and \cite{perron2006} and \cite{aue:13} for comprehensive reviews. More recently, there is a surge of interest in change-point testing under the high-dimensional setting, see for example \cite{Horvath2012}, \cite{Cho2015}, \cite{Jirak2015}, \cite{Wang2018c}, \cite{Enikeeva2019}, and \cite{Wang2020}. However, these works mainly focus on testing the stability of mean vector or covariance matrices of high-dimensional times series, and we are not aware of any valid change-point testing procedure for high-dimensional linear models.

The existing change-point literature for high-dimensional linear models focuses on change-point estimation, where the main objective is to locate the unknown structural breaks. \cite{Lee2016,Lee2018} and \cite{Kaul2019} study change-point estimation with the knowledge of a single change-point, and \cite{Leonardi2016} and \cite{Wang2021a} propose procedures for localization of multiple change-points. However, all these works are essentially built upon penalized model selection procedures that cannot be easily adapted for valid hypothesis testing of structural instability. Moreover, all these works require temporal independence, which may not be realistic for real data applications, especially for financial and economics studies.

The change-point testing literature for linear models can only be found under the low-dimensional setting, see the sup $F$ test and its extensions in \cite{andrews1993} and \cite{Bai1998}. However, since the sup $F$ test is based on OLS, it is degenerate under the high-dimensional setting where the dimension of covariates $p$ is larger than the sample size $n$. In addition, even if $p$ is smaller than $n$, the sup $F$ test can incur large size distortion for the case where $p$ and $n$ are on comparable scales.

To fill the gap in the literature, in this paper, we propose a valid and optimal change-point testing procedure, named QF-CUSUM, for high-dimensional linear models. Our main contributions are three-fold. First, to our best knowledge, QF-CUSUM is the first valid change-point testing procedure for high-dimensional linear models that can control the type-I error at any given level. Second, through a well-designed randomization component, the proposed test is theoretically sound for temporally dependent observations, which is more appealing and realistic for real data applications. Third, we show that QF-CUSUM is optimal in the sense that it achieves the minimax lower bound of the detection boundary for a wide class of high-dimensional linear models. In addition, we establish some new consistency results of the Lasso~\citep{Tibshirani1996} under the change-point setting, which may be of independent interest. 

The rest of the paper is organized as follows.   \Cref{sec:QF_CUSUM_Dep} proposes the QF-CUSUM test and establishes its asymptotic distribution under both the null and alternative hypotheses. The matching detection lower bound  of QF-CUSUM is further established as well.  The optimal testing procedure requires knowledge of population quantities and in \Cref{sec:adaptive}, we develop an adaptive testing procedure to overcome this limitation. In 
\Cref{sec:localization}, we further show  that  the  change point  estimator based on the QF-CUSUM   test statistics    achieves the sharp localization error bound. 
Finally in \Cref{sec:numerical}, we conduct extensive numerical experiments and a real data example to examine the finite sample performance of QF-CUSUM under both null and alternative hypotheses.   \Cref{sec:conclusion} concludes.

\section{A Quadratic Form based CUSUM Test}\label{sec:QF_CUSUM_Dep}

We first formally introduce the change-point testing problem for the high-dimensional linear model. Denote $y_i\in \mathbb{R}$ as the univariate response and $x_i\in \mathbb{R}^p$ as the $p$-dimensional covariate. Given sequentially observed data $\{(x_i,y_i)\}_{i=1}^n$, the high-dimensional linear model assumes
\begin{align}\label{eq:model}
	y_i=x_i^\top \beta_i^* +\epsilon_i, \text{ for } i=1,2,\cdots,n,
\end{align}
where $\beta_i^*$ is the regression coefficient and $\epsilon_i$ is the random noise. Under the null hypothesis, the regression coefficient $\beta_i^*$ stays unchanged and we have
$$H_0: \beta_1^*=\beta_2^*=\cdots=\beta_n^*.$$
Under the alternative hypothesis, there exists $K\geq 1$ unknown change-points $1\leq\eta_1<\eta_2<\cdots<\eta_K<n$ such that
$$H_a: \beta_1^*=\cdots=\beta^*_{\eta_1}\neq \beta^*_{\eta_1+1}=\cdots=\beta^*_{\eta_2}\neq \cdots\neq \beta_{\eta_K+1}^*=\cdots=\beta^*_n.$$

Given a generic interval $\I \subseteq (0,n]$, denote its cardinality as $|\I|$. To handle the high-dimensional setting where $p>|\I|$, in this following, we work with the Lasso estimator such that
\begin{align}\label{eq:interval lasso}
	\widehat \beta_ \I  : = \argmin_{  \beta \in \mathbb R^p  } \frac{1}{|\I| }  \sum_{i\in \I }  \left(  y_i  -  x_i ^\top  \beta  \right)^2 
	+ \frac{ \lambda }{  \sqrt {  |\I| }  } \| \beta\|_1,  
\end{align}
where $\|\cdot\|_1$ is the $l_1$-norm and $\lambda $ is the Lasso tuning parameter to be specified later. Define $\beta_{\I}^*=\frac{1}{|\I|}\sum_{i\in\I}\beta_i^*$, by \Cref{assume:beta}(a), we have $\beta_{\I}^*$ is unique minimizer of the population squared loss function $\mathbb{E}(\sum_{i\in\I} (y_i-x_i^\top \beta)^2)$ and thus can be viewed as the population version of $\widehat \beta_\I$. In addition, denote the sample covariance matrix estimated based on the interval $\I$ as $\widehat \Sigma_{ \mathcal I } := \frac{1}{ | \mathcal I  | } \sum_{i \in \mathcal I }  x_ix_i^ \top.$

\Cref{lemma:interval lasso} provides a uniform consistency result for the Lasso estimator $\widehat \beta_\I$ defined in \eqref{eq:interval lasso} w.r.t. its population quantity $\beta_{\I}^*$ across all intervals $\I\subseteq (0,n]$, which is of independent interest. \Cref{lemma:interval lasso} is used extensively in the technical proof to analyze the later proposed quadratic form. 
 
\begin{lemma} \label{lemma:interval lasso}
Let $ \zeta> 0$ be any fixed constant in $(0,1)$. Suppose \Cref{assume:beta} and \Cref{assume:order_betamixing} hold. There exists $\lambda = C_\lambda \sqrt { \log p }$ with some sufficiently large constant $C_\lambda $, such that under both $H_0$ and $H_a$, we have with probability at least $1-n^{-5}$, for all $ \mathcal I  \subseteq (0, n] $ with $| \mathcal I | \ge \zeta n  $, it holds that 
	\begin{align*}
	 \| \widehat \beta_\I -\beta^*_\I \| _2^2 \le \frac{C \s\log p}{n }, ~ \| \widehat \beta_\I  -\beta^*_\I \| _1  \le     C \s\sqrt {  \frac{\log p} {n } }, ~ \| (\widehat \beta_\I -\beta^*_\I)_{S_\I^c} \| _1 \le  3 \| (\widehat \beta _\I  -\beta^*_\I )_{S_\I } \| _1   .
	\end{align*}
where $C$ is an absolute constant and $S^c_\I=\{1,2,\cdots,p\} \setminus S_\I$ with $S_\I$ being the support set of $\beta_\I^*.$
\end{lemma}

\subsection{A bias-corrected quadratic form}
To ease presentation, in the following, we introduce the quadratic form based CUSUM~(QF-CUSUM) using the single change-point alternative as a motivating example. In other words, we assume under $H_a$, there is only one unknown change-point $\eta$ such that
$$\beta_i^*=\beta^{(1)} \text{ for } 1\leq i\leq \eta \text{ and } \beta_i^*=\beta^{(2)} \text{ for } \eta+1\leq i\leq n.$$
We remark that the presented method and theory apply to the {\it multiple} change-point alternative.

For two regression coefficients $\beta^{(1)}$ and $\beta^{(2)}$, it is natural to directly measure their difference via the $l_2$-norm $\|\beta^{(1)}-\beta^{(2)}\|_2^2.$ However, under the regression context, an alternative  and more relevant  quantity is the quadratic form ${(\beta^{(1)}-\beta^{(2)})^\top \Sigma (\beta^{(1)}-\beta^{(2)})}$, which equals to ${\text{Var}(x_i^\top(\beta^{(1)}-\beta^{(2)}))}$ as $\text{Cov}(x_i)=\Sigma$. Note that under Assumption \ref{assume: model assumption}(a), we have that
$$c_x\|\beta^{(1)}-\beta^{(2)}\|_2^2 \leq {(\beta^{(1)}-\beta^{(2)})^\top \Sigma (\beta^{(1)}-\beta^{(2)})} \leq C_x \|\beta^{(1)}-\beta^{(2)}\|_2^2.$$
Thus, in terms of theoretical magnitude, $\|\beta^{(1)}-\beta^{(2)}\|_2^2$ and ${(\beta^{(1)}-\beta^{(2)})^\top \Sigma (\beta^{(1)}-\beta^{(2)})}$ are the same and both can capture the change in the regression coefficient.

However, compared to $\|\beta^{(1)}-\beta^{(2)}\|_2^2$, the quadratic form ${(\beta^{(1)}-\beta^{(2)})^\top \Sigma (\beta^{(1)}-\beta^{(2)})}$ further incorporates the covariance structure $\Sigma$ of $x$ and thus can better reflect the difference between two regression models $y=x^\top \beta^{(1)}+\epsilon$ and $y=x^\top \beta^{(2)}+\epsilon$.   Therefore   ${(\beta^{(1)}-\beta^{(2)})^\top \Sigma (\beta^{(1)}-\beta^{(2)})}$ is preferred for change-point testing. We remark that ${(\beta^{(1)}-\beta^{(2)})^\top \Sigma (\beta^{(1)}-\beta^{(2)})}$ is closely related to the explained variance in the regression literature, see for example \cite{Cai2020}, where the explained variance is defined as $\beta^\top\Sigma\beta$ for a regression model $y=x^\top \beta+\epsilon$.

Given a potential change-point location $t$, to estimate the quadratic form, a natural choice is the plug-in estimator. Specifically, based on the Lasso estimator \eqref{eq:interval lasso}, we can obtain $\widehat{\beta}_{(0,t]}$ from $\{(x_i,y_i)\}_{i=1}^t$ and $\widehat{\beta}_{(t,n]}$ from $\{(x_i,y_i)\}_{i=t+1}^n$. Denote $\widehat{\Delta}_t=\widehat{\beta}_{(0,t]}-\widehat{\beta}_{(t,n]}$, one possible plug-in estimator we can use is
\begin{align*}
	(\widehat{\Delta}_t^\top \widehat{\Sigma}_{(0,t]}\widehat{\Delta}_t+\widehat{\Delta}_t^\top\widehat{\Sigma}_{(t,n]} \widehat{\Delta}_t)/2,
\end{align*}
which intuitively estimates $\Delta_t^{*\top} \Sigma \Delta_t^*$ with $\Delta_t^*={\beta}_{(0,t]}^*-{\beta}_{(t,n]}^*$.

However, as will be made clear in the proof, due to the use of Lasso, the plug-in estimator has an intrinsic bias term that makes it technically difficult to analyze its asymptotic distribution under the null and alternative hypothesis. In particular, we have
\begin{align*}
	&\widehat{\Delta}_t^\top \widehat{\Sigma}_{(0,t]} \widehat{\Delta}_t -\Delta_t^{*\top} \Sigma \Delta_t^*\\
   =&2\widehat{\Delta}_t^\top\widehat{\Sigma}_{(0,t]}(\widehat{\beta}_{(0,t]}-{\beta}_{(0,t]}^*)-
   2\widehat{\Delta}_t^\top\widehat{\Sigma}_{(0,t]}(\widehat{\beta}_{(t,n]}-{\beta}_{(t,n]}^*)+\Delta_t^{*\top} (\widehat{\Sigma}_{(0,t]}-\Sigma)\Delta_t^* -(\widehat{\Delta}_t-\Delta_t^*)^\top\widehat{\Sigma}_{(0,t]}(\widehat{\Delta}_t-\Delta_t^*).
\end{align*}
Due to the bias of the Lasso estimator, the asymptotic behavior of the first two terms in the above equation is technically difficult to analyze. 

Thus, we further introduce bias correction terms to the plug-in estimator. Specifically, the first two terms can be estimated by $\frac{2}{t}\sum_{i=1}^t\widehat{\Delta}_t^\top x_i(x_i^\top\widehat{\beta}_{(0,t]}-y_i)$ and $\frac{2}{n-t}\sum_{i=t+1}^n\widehat{\Delta}_t^\top x_i(x_i^\top\widehat{\beta}_{(t,n]}-y_i)$, where the key idea is to make use of $\mathbb{E}(y_i|x_i)=x_i^\top\beta^*_i$. The bias-corrected quadratic form of $\widehat{\Delta}_t^\top \widehat{\Sigma}_{(0,t]}\widehat{\Delta}_t$ is thus defined as
$\widehat{\Delta}_t^\top \widehat{\Sigma}_{(0,t]}\widehat{\Delta}_t+\frac{2}{t}\sum_{i=1}^t\widehat{\Delta}_t^\top x_i(y_i-x_i^\top\widehat{\beta}_{(0,t]})-\frac{2}{n-t}\sum_{i=t+1}^n\widehat{\Delta}_t^\top x_i(y_i-x_i^\top\widehat{\beta}_{(t,n]}).$ Based on the same arguments, we have that the bias-corrected quadratic form of $\widehat{\Delta}_t^\top\widehat{\Sigma}_{(t,n]} \widehat{\Delta}_t$ can be defined as $\widehat{\Delta}_t^\top\widehat{\Sigma}_{(t,n]} \widehat{\Delta}_t+\frac{2}{t}\sum_{i=1}^t\widehat{\Delta}_t^\top x_i(y_i-x_i^\top\widehat{\beta}_{(0,t]})-\frac{2}{n-t}\sum_{i=t+1}^n\widehat{\Delta}_t^\top x_i(y_i-x_i^\top\widehat{\beta}_{(t,n]}).$

Combine the above arguments and sum the two bias-corrected quadratic forms of $\widehat{\Delta}_t^\top \widehat{\Sigma}_{(0,t]}\widehat{\Delta}_t$ and $\widehat{\Delta}_t^\top\widehat{\Sigma}_{(t,n]} \widehat{\Delta}_t$, we estimate $\Delta_t^{*\top} \Sigma \Delta_t^*$ using the bias-corrected quadratic form
\begin{align}\label{eq:bias_qf}
	\frac{1}{2}(\widehat{\Delta}_t^\top \widehat{\Sigma}_{(0,t]}\widehat{\Delta}_t+\widehat{\Delta}_t^\top\widehat{\Sigma}_{(t,n]} \widehat{\Delta}_t)+\frac{2}{t}\sum_{i=1}^t\widehat{\Delta}_t^\top x_i(y_i-x_i^\top\widehat{\beta}_{(0,t]})-\frac{2}{n-t}\sum_{i=t+1}^n\widehat{\Delta}_t^\top x_i(y_i-x_i^\top\widehat{\beta}_{(t,n]}),
\end{align}
which later serves as the building block of the proposed QF-CUSUM test.

\textbf{A Goodness of Fit viewpoint}: The proposed QF in \eqref{eq:bias_qf} may seem complicated, however, we show that it is indeed intuitive and interpretable by rewriting \eqref{eq:bias_qf} as a goodness of fit~(GoF) statistic. Specifically, given a time point $t$, under $H_0$ with no change-point, we have that $\widehat{\beta}_t:=(\widehat{\beta}_{(0,t]}+\widehat{\beta}_{(t,n]})/2$ can serve as a valid estimator for the regression coefficient. Thus, to differentiate $H_0$ and $H_a$, we can examine the magnitude of the following GoF statistic,
\begin{align*}
	\frac{1}{t}\left(\sum_{i=1}^t(y_i-x_i^\top \widehat{\beta}_{(0,t]})^2-\sum_{i=1}^t(y_i-x_i^\top \widehat{\beta}_{t})^2\right)+\frac{1}{n-t}\left(\sum_{i=t+1}^n(y_i-x_i^\top \widehat{\beta}_{(t,n]})^2-\sum_{i=t+1}^n(y_i-x_i^\top \widehat{\beta}_{t})^2\right),
\end{align*}
where we declare the existence of change-points once the GoF statistic is ``too large" for some $t=1,\cdots,n$. Elementary algebra shows that the GoF statistic is nothing but $\eqref{eq:bias_qf}/2.$

\subsection{QF-CUSUM}\label{subsec:qf_cusum_proof}
Based on the bias-corrected quadratic form \eqref{eq:bias_qf}, in this section, we propose QF-CUSUM for change-point testing in high-dimensional linear models.

Using \Cref{lemma:interval lasso}, it is intuitive to see that under $H_0,$ since $\Delta_t^*=0$ for all $t$, we have that the bias-corrected QF \eqref{eq:bias_qf} is of order $O_p({\s\log p}/{n})$, which is in general degenerate and does not follow any pivotal distribution. To overcome this issue and design a non-degenerate test under $H_0$, we adopt a randomization strategy to inject additional variance into the bias-corrected QF, see similar approaches used in \cite{Cai2020} to construct confidence intervals for high-dimensional linear models. We refer to \cite{lopes2011more}, \cite{srivastava2016raptt} and \cite{li2020randomized} for other types of randomized tests proposed in the high-dimensional literature.

Specifically, we generate a sequence of random variables $\{\xi_i\}_{i=1}^n$ such that $\xi_i\overset{i.i.d.}{\sim} N(0,\sigma_\xi^2)$, where the variance level $\sigma_{\xi}^2$ will be specified later in \Cref{theorem:main_beta}. We then define the randomized bias-corrected QF as
\begin{align}\label{eq:qf_random}
	\mathcal{S}_n(t)=&\frac{1}{2}(\widehat{\Delta}_t^\top \widehat{\Sigma}_{(0,t]}\widehat{\Delta}_t+\widehat{\Delta}_t^\top\widehat{\Sigma}_{(t,n]} \widehat{\Delta}_t)\nonumber\\
	+&\frac{1}{t}\sum_{i=1}^t(2\widehat{\Delta}_t^\top x_i+\xi_i)(y_i-x_i^\top\widehat{\beta}_{(0,t]})-\frac{1}{n-t}\sum_{i=t+1}^n(2\widehat{\Delta}_t^\top x_i+\xi_i) (y_i-x_i^\top\widehat{\beta}_{(t,n]}),
\end{align}
and further define the proposed QF-CUSUM statistic as
\begin{align}\label{eq:qf_cusum}
	\mathcal T_n(t)=&\frac{1}{\sigma_{\epsilon}\sigma_\xi}\sqrt{\frac{t(n-t)}{n}}\mathcal{S}_n(t),
\end{align}
where we combine $\mathcal{S}_n(t)$ with the classical CUSUM weight function. We reject the null hypothesis if $\mathcal{T}_n(t)$ is larger than a pre-specified threshold.

Before presenting the formal theoretical results, we first discuss some high-level intuition of the QF-CUSUM test in \eqref{eq:qf_cusum}. Using \Cref{lemma:interval lasso} and some additional technical lemmas in the Appendix, we can show that under $H_0$, $\mathcal T_n(t)$ can be approximately written as
\begin{align*}
	\T_n(t)=\frac{1}{\sigma_{\epsilon}\sigma_\xi}\sqrt{\frac{t(n-t)}{n}}\left[\frac{1}{t}\sum_{i=1}^t \epsilon_i\xi_i -\frac{1}{n-t}\sum_{i=t+1}^n\epsilon_i\xi_i \right] + O_p\left(\frac{ \s\log p}{\sqrt{n}}\right)+O_p\left(\frac{ \s\log p}{\sqrt{n}\sigma_\xi}\right).
\end{align*}
The first term takes the form of the classical CUSUM statistic and we later control the last two noise terms with additional conditions in \Cref{assume:order_betamixing}. 

 {\bf Remark 1.} We remark that the randomization strategy is devised to derive a pivotal and non-degenerate asymptotic distribution under $H_0$. More importantly, as will be seen later in this  section, the randomized error $\{\xi_i\}_{i=1}^n$ is deliberately designed to avoid the estimation of the long-run variance~(LRV) under temporal dependence, which is highly challenging even in low-dimensional change point settings~\citep{shao2010testing}. In the literature, another strategy to avoid a degenerate test under $H_0$ is to combine the designed test with another pivotal (possibly less powerful) test statistic, see for example \cite{Fan2015}. However, in our current setting of high-dimensional change-point testing, it is not obvious how to construct another pivotal test and such a strategy will likely involve the estimation of LRV under temporal dependence. We thus do not pursue this direction.

  \subsection{High-dimensional regression time series}
In practice, temporal dependence is the norm rather than exception for sequentially observed data. Thus, in this section, we  study the behavior of the QF-CUSUM test $\T_n(t)$ under the context of temporal dependence. Specifically, we follow the $\beta$-mixing framework in \cite{Wong2020} for high-dimensional linear models, which also allows for heavy-tailed observations.

  We begin our discussion by reviewing the concept of $\beta$-mixing and define the sub-Weibull family, which includes the independent sub-Gaussian family   as a special case.

Given two sigma fields $\mathcal{F}_1$ and $\mathcal{F}_2$, we define the $\beta$-mixing coefficient as
$$ \beta  (\mathcal{F}_1, \mathcal{F}_2) = \sup \frac{1}{2} \sum_{ i=1}^I \sum_{j=1}^J
|\mathbb P(A_i\cap B_j) -\mathbb P(A_i) \mathbb P(B_j) |,$$
where the supremum is over all pairs of finite partitions $\{A_1, \ldots, A_I \}$ and $ \{B_1 , \ldots, B_J\}$ of the sample space $\Omega$ such that $ A_i \in \mathcal{F}_1$, $B_j \in \mathcal{F}_2$ for all $i,j$. Thus, given a sequence of random vectors $\{ Z_t\}_{t=-\infty}^{\infty}$, we can define its $\beta$-mixing coefficient at lag $l\in \mathbb{N}$ as
$$ \beta(l) =  \sup_{t\in \mathbb Z } \ \beta (\sigma( \{ Z_{-\infty :t }) \}, \sigma( \{ Z_{  t+l:\infty  })\}    ),$$
where $\sigma(\cdot)$ denotes the generated sigma field.

We now introduce the sub-Weibull family. For any $\gamma \in (0,2]$\footnote{The sub-Weibull family further incorporates the case of $\gamma>2$. However, a $\gamma>2$ implies a lighter tail than the sub-Gaussian family, which is not useful in practice. Thus, we only consider $\gamma \in (0,2]$.}, a random variable $x\in\mathbb{R}$ is called sub-Weibull with parameter $\gamma$ if there exists a constant $K>0$ such that
$$ ( \mathbb E  |x |^d)^{1/d} \le K d^{1/\gamma} \text{ for all }  d \ge \min\{ 1,\gamma\}.$$
We further define the sub-Weibull$(\gamma)$ norm to be $\| x \|_{\psi_\gamma } : = \sup_{ d \ge 1 } (\mathbb E|x|^d )^{1/d} d^{-1/\gamma}.$ A random vector $x \in \mathbb R^p$ is said to be a sub-Weibull$(\gamma)$ random vector if $ \| x\|_{\psi_\gamma} : = \sup_{v\in S^{p-1}} \|v^\top x\|_{\psi_\gamma }<\infty,$
where $ S^{p-1}$ is the unit sphere in $\mathbb R^p$. Note that sub-Weibull with $\gamma=2$ reduces to the sub-Gaussian family and with $\gamma=1$ reduces to the sub-Exponential family.

We proceed by imposing some mild assumptions on the high-dimensional linear model.

\begin{assumption} \label{assume:beta}
	The observations $\{(x_i,y_i)\}_{i=1}^n$ follow Model \eqref{eq:model} with strictly stationary covariates $\{x_i\}_{i=1}^n$ and random noise $\{\epsilon_i\}_{i=1}^n.$ 
	\\
	{\bf a.} (Eigenvalue) Denote $\Sigma=\text{Cov}(x_i)$, there exists absolute constant $c_x$ and $C_x$ such that the minimal and maximal eigenvalues of $\Sigma$ satisfy 
	$ \Lambda_{\min}   (\Sigma)\ge c_x >0  $   and $\Lambda_{\max} (\Sigma ) \le C_x < \infty .   $ 
	\\	  
	{\bf b.} (Noise) The random noise $\epsilon_i$ is independent of $x_i$ and we have $\mathbb E(\epsilon_i)=0$ and  $ \text{Var}(\epsilon_i) =\sigma^2_\epsilon>0$.
	\\
	{\bf c.} (Mixing) The data $\{(x_i,\epsilon_i)\}_{i=1}^n$ is geometrically $\beta$-mixing; i.e., there exist positive constants
	$c $ and $\gamma_1$ such that the $\beta$-mixing coefficient of $\{(x_i,\epsilon_i)\}_{i=1}^n$ satisfies $\beta(l) \le \exp(-c l^{\gamma_1})$ for all $l\in \mathbb N. $
	\\
	{\bf d.} (Sub-Weibull) The covariate $x_i$ and random noise $\epsilon_i$ follow sub-Weibull distributions with parameter $\gamma_2$; i.e., $\|x_i\|_{\psi_{\gamma_2}} \le K_X$ and $\|\epsilon_i\|_{\psi_{\gamma_2}} \le K_\epsilon$ for some absolute constants $K_X$ and $K_\epsilon.$
	\\
	{\bf e.}  (Infill) Under $H_a$, there exist a fixed set of (relative) change-points 
	$ 0<  \eta_1^* < \eta_2^* < \ldots <\eta_K^* <1 $ 
	such that $\eta_k=\lfloor n\eta_k^*\rfloor$ for $k=1,\cdots,K.$
\end{assumption}

Given the two positive constants $\gamma_1$ and $\gamma_2$ in \Cref{assume:beta}, we further define a key quantity $\gamma = \left( \frac{1}{\gamma_1} + \frac{2}{\gamma_2}\right)^{-1}$ and introduce \Cref{assume:order_betamixing}, which extends the highly-dimensional regression  signal-to-noise condition  to  the temporal independence setting.
\begin{assumption}\label{assume:order_betamixing} 
We have  that 
$$	\frac{ \max\{ \log^{1/\gamma}(p),  ( \s\log p )^{ \frac{2}{\gamma}- 1}  \} }{n} \to 0.$$ 
  Furthermore, it holds that 
	$$\frac{\s\log p  }{\sqrt n } \to 0 . $$	
\end{assumption}
\noindent As discussed in \cite{Wong2020}, the parameter $\gamma$ measures the difficulty brought by the temporal dependence and heavy-tailed condition, where a smaller $\gamma$ corresponds to a more difficult problem and thus requires a stronger assumption as in the first part of \Cref{assume:order_betamixing} than the usual conditions for the consistency of Lasso. Note that for sub-Gaussian random variables with temporal independence, we have $\gamma_1=\infty$ and $\gamma_2=2$. Thus the first part of  \Cref{assume:order_betamixing} essentially reduces to  $ \s\log(p)/n \to 0 $, which is the optimal  signal-to-noise condition for the  highly-dimensional linear  regression model. 
 
The second part of \Cref{assume:order_betamixing} requires that $\s\log p=o(\sqrt{n})$, which is needed to control the noise term of $\T_n(t)$ as discussed above and thus facilitates the derivation of the limiting distribution. This is a standard assumption in the literature  that concerns   the limiting distribution of the high-dimensional linear models, see e.g.\ \cite{Javanmard2013}, \cite{Geer2014}, \cite{Zheng2019} and \cite{Cai2020}.
   
\subsection{Theoretical guarantees}\label{subsec:qf_betamixing}

In the following, we establish the theoretical guarantees of QF-CUSUM under the null and alternative hypothesis with \Cref{assume:beta} and \Cref{assume:order_betamixing}.  Throughout the rest of the manuscript, we will assume  without loss of generality that $p \ge  n^{\alpha}$ for some $\alpha >0$. This is a   convenient assumption commonly used in the literature for high-dimensional linear models. We remark that for $p=o(n^{\alpha})$, e.g.\ $p=O(1)$ or $p=\log n$, all of our theoretical results continue to hold, up to a logarithmic factor of $n$.  Below in  \Cref{theorem:main_beta}, we establish a process convergence result for $ \mathcal T_n(t)$ under the null hypothesis $H_0$.

To begin, using \Cref{lemma:interval lasso} and additional technical results established in the Appendix,  we show that under $H_0$, 
 $\mathcal T_n(t)$ can   be approximately written as
\begin{align*}
	\T_n(t)=\frac{1}{\sigma_{\epsilon}\sigma_\xi}\sqrt{\frac{t(n-t)}{n}}\left[\frac{1}{t}\sum_{i=1}^t \epsilon_i\xi_i -\frac{1}{n-t}\sum_{i=t+1}^n\epsilon_i\xi_i \right] + O_p\left(\frac{s\log p}{\sqrt{n}}\right)+O_p\left(\frac{s\log p}{\sqrt{n}\sigma_\xi}\right).
\end{align*}
Define $\sigma_{\epsilon L}^{2}=\sum_{t=-\infty}^\infty \text{Cov}(\epsilon_0,\epsilon_t)$, which is the long-run variance~(LRV) of $\{\epsilon_i\}_{i=1}^n$ due to temporal dependence. It is well known that $\lim_{t\to\infty}\text{Var}(\frac{1}{t}\sum_{i=1}^t \epsilon_i)=\sigma_{\epsilon L}^2>\sigma_\epsilon^2$ and the estimation of LRV can be challenging as it requires a user-specified bandwidth parameter~\citep{shao2010testing}.

Fortunately, note that due to the i.i.d.\ nature of the randomized error $\{\xi_i\}_{i=1}^n$, $\T_n(t)$ automatically avoids this difficulty. Specifically, the first term still converges to the Gaussian process defined in \Cref{theorem:main_beta}, as $\text{Var}(\frac{1}{t}\sum_{i=1}^t \epsilon_i\xi_i)=\sigma_\epsilon^2\sigma_\xi^2/t$ and $\text{Var}(\frac{1}{n-t}\sum_{i=t+1}^n \epsilon_i\xi_i)=\sigma_\epsilon^2\sigma_\xi^2/(n-t)$ do not depend on the LRV $\sigma_{\epsilon L}^{2}$ but only on the marginal variance $\sigma_\epsilon^2$. Denote the change size $ \kappa_k: =  \| \beta^* _{\eta_k+1} - \beta_{\eta_k} ^* \|_2$ for $k=1,2,\cdots, K.$  \Cref{theorem:main_beta} provides the theoretical guarantees for QF-CUSUM under both the null and alternative hypotheses.

\begin{theorem} \label{theorem:main_beta}
	Let $\zeta>0$ be any fixed constant in $(0,1/2)$. Suppose \Cref{assume:beta} and \Cref{assume:order_betamixing} hold,  and $\lambda = C_\lambda \sqrt { \log p }$ for some sufficiently large constant $C_\lambda $.  Under $H_0$, we have that
	$$\T_n(\lfloor nr \rfloor) \Rightarrow \mathcal{G}(r), \text{ over } r \in [\zeta, 1-\zeta], \text{ and }\max_{t=\lfloor n\zeta\rfloor, \lfloor n\zeta\rfloor+1, \cdots, \lfloor n(1-\zeta)\rfloor}\T_n(t) \overset{d}{\to} \sup_{r\in [\zeta,1-\zeta]}\mathcal{G}(r).$$
	where $\mathcal G(r)$ is the Gaussian process such that $\mathcal G(r)=[B(r)-rB(1)]/\sqrt{r(1-r)}$ for $r\in (0,1)$ and $B(\cdot)$ is the standard Brownian motion. Under $H_a$,  suppose  that  
	  $\sigma_\xi= A_n\cdot \frac{ \s\log p}{\sqrt{n}}$     for any diverging sequence $A_n \to \infty $  and  the maximum change size satisfies 
	\begin{align} \label{eq:detection boundary beta}
	\max_{1\le k \le K } \kappa _k^2 \ge  B _n \frac{\s\log p}{n }
	\end{align}
	for any diverging  sequence $ B _n \to \infty $ and $B_n/A_n \to \infty  $ as $n \to \infty$. We have that
	\begin{align*}
		\mathbb{P}\left(\max_{t=\lfloor n\zeta\rfloor, \lfloor n\zeta\rfloor+1, \cdots, \lfloor n(1-\zeta)\rfloor}\T_n(t) >\mathcal{G}_\alpha(\zeta)\right) \to 1 \text{ as } n \to \infty.
	\end{align*}	
\end{theorem}

Note that \Cref{theorem:main_beta} holds for any interval $ [\zeta, 1-\eta]$ for  any $\zeta>0$, but does not hold for the interval $[0,1]$. This is a well-known phenomenon for the CUSUM statistic, see for example \cite{andrews1993}. In fact, we have that $\sup_{r\in[0,1]} \mathcal{G}(r) = \infty$~(Corollary 1 in \cite{andrews1993}), resulting in a trivial critical value $\infty$ for $\T_n(t)$. Thus, we require the interval $[\zeta,1-\zeta]$ to be bounded away from zero and one.

For any fixed $\zeta \in (0,1/2)$, $\sup_{r\in [\zeta,1-\zeta]}\mathcal{G}(r)$ is a pivotal distribution. For a given type-I error rate $\alpha$, define $\mathcal{G}_\alpha(\zeta)$ such that $ \mathbb P(\sup_{ r \in[\zeta,1-\zeta ]} \mathcal G(r) \ge   \mathcal G_{\alpha }(\zeta)) = \alpha.$ We reject the null hypothesis if
$$\max_{t=\lfloor n\zeta\rfloor, \lfloor n\zeta\rfloor+1, \cdots, \lfloor n(1-\zeta)\rfloor}\T_n(t) >\mathcal{G}_\alpha(\zeta).$$
By the first part of \Cref{theorem:main_beta}, this gives a test with asymptotically correct type I error rate.

The second part of \Cref{theorem:main_beta} states that QF-CUSUM can detect changes in the regression coefficient once the {\it maximum} change size is  larger than ${\s\log p}/{n}$ by any   diverging order $B_n$. Note that our result holds for any fixed $\zeta \in (0,1/2)$ without requiring that the true change-point with the maximum jump size is located between $\lfloor n\zeta\rfloor$ and $\lfloor n(1-\zeta)\rfloor$. We show in \Cref{subsec_optim} that \eqref{eq:detection boundary beta} is the optimal detection boundary that can be achieved by any valid test. We remark that \Cref{theorem:main_beta} indicates that QF-CUSUM works under both single and multiple change-point alternatives and allows multiscale changes, i.e.\ different magnitude of jump sizes among $\{\kappa_1,\kappa_2,\cdots,\kappa_K\}$.   

We can further establish an asymptotic distribution for $\T_n(t)$ under the alternative hypothesis with proper centering and scaling. However, due to temporal dependence, the scaling factor is more complicated as it now involves LRV of the random noise $\{\epsilon_i\}_{i=1}^n$ and covariate $\{x_i\}_{i=1}^n.$ Define 
\begin{align*} 
	\mu (t)  =   \Delta^{*\top}_t  \Sigma \Delta^*_t \text{ and } \widetilde\sigma^2(t) =\sigma_\epsilon^2    \sigma_\xi^2    + 4\text{Var}\left(\frac{1}{t}\sum_{i=1}^t x_i^\top \Delta_t^{* }\epsilon_i\right)    +   \frac{1}{4}\text{Var} \left(\frac{1}{t}\sum_{i=1}^t \Delta_t^{* }(x_ix_i^\top-\Sigma)\Delta_t^* \right),
\end{align*}
and further define $$	\widetilde\psi _L (t) =  \widetilde\sigma^2(t)  +   \frac{1}{t  }   \sum_{i=1}^t   \sigma_\xi^2  ( \beta_i^* -   \beta^*_{ (0, t ] }   ) ^\top \Sigma   ( \beta_i^* -   \beta^*_{ (0, t ] }   ), ~ \widetilde\psi _R (t)  =  \widetilde\sigma^2(t)  +   \frac{1}{ n- t  }   \sum_{i=t +1}^{n } \sigma_\xi^2  ( \beta_i^* -   \beta^*_{ (t , n]  }   ) ^\top \Sigma   ( \beta_i^* -  \beta^*_{ (t , n]  }   ) .$$
  \Cref{theorem:alternative 2_beta} provides the limiting distribution of bias-corrected QF statistics  $\mathcal S_n$ defined in \eqref{eq:qf_random}  under the alternative hypothesis $H_a$. 

\begin{theorem}  \label{theorem:alternative 2_beta} 
	Let $\zeta>0$ be any fixed constant in $(0,1/2)$. Suppose \Cref{assume:beta} and \Cref{assume:order_betamixing} hold, and $\lambda = C_\lambda \sqrt { \log p }$ for some sufficiently large constant $C_\lambda $. Under $H_a$, we have that for any fixed $r  \in [\zeta, 1-\zeta]  $, it holds that
	$$   \bigg(  \sqrt {  \frac{ \widetilde\psi_L (\lfloor r n\rfloor ) }{   \lfloor r n\rfloor   }   +  \frac{  \widetilde\psi_R (\lfloor r n\rfloor  )  } {   n-\lfloor r n\rfloor   }  } \bigg)^{-1}    \left(\mathcal{S}_n(\lfloor r n\rfloor ) -  \mu (\lfloor r n\rfloor )\right) \overset{d}{\to} N(0,1).$$
	In addition,  the empirical process 
	$$  \sqrt n    \big(\mathcal{S}_n(\lfloor r n\rfloor  ) -  \mu (\lfloor r n\rfloor ) \big)$$
	converges to a centered Gaussian process on $r\in [\zeta, 1-\zeta]$.
\end{theorem}

\subsection{Lower bound of the detection boundary}\label{subsec_optim}
 
In this section, we further establish a matching lower bound for QF-CUSUM, which states that, if the maximum change size $\max_{1\leq k\leq K} \kappa_k^2 $ is too small, no test can consistently distinguish between $H_0$ and $H_a$. This together with the results in \Cref{theorem:main_beta}   shows that our test procedure is asymptotically optimal (up to an arbitrarily slowly diverging factor).


\begin{theorem} \label{lemma:lower bound for detection}
	Suppose the observations $\{(x_i,y_i)\}_{i=1}^n $ are generated according to Model \eqref{eq:model} and that  $\s\log(p) /\sqrt n \to 0  $  and   $\s\le p ^{\alpha}  $ for any $\alpha <1/2$. In addition,   suppose that  \Cref{assume:beta} and \Cref{assume:order_betamixing} hold.
	For $0<\zeta<1/2$, consider 
	$$ H_0  :     \beta^*_1 =   \ldots  = \beta_n^* \quad
	\text{and}  \quad  H_a (b)  :   H_a \text{ holds and } \max_{1\le k \le K } \kappa_k^2 \ge b \frac{\s\log(p)} {n}  , $$
	where $b$ is some positive constant. 
	Let $   \psi  $ be any test function  mapping $\{ (x_i,y_i)\}_{i=1}^n $  to $\{ 0,1\}$. If 
	$ b $ is a sufficiently small constant, we have
	$$ \liminf_{n,p  \to \infty }\inf_{   \psi}  \sup_{P  \in \mathcal P_1 (b) } \mathbb E _ 0 (  \psi) + \mathbb E _ {P } (1-    \psi)\to 1   ,$$
	where $  \mathcal P_1 (b)$ denotes the class of distributions satisfying $H_a(b)$, $\mathbb{E}_0(\psi)$ gives the type-I error of $\psi$ under $H_0$ and $\mathbb{E}_P(1-\psi)$ gives the power of $\psi$ when the observations are generated according to $P.$
\end{theorem}

 \Cref{theorem:main_beta} and \Cref{lemma:lower bound for detection} together  imply that, under additional mild conditions on the sparsity parameter $\s$, the  QF-CUSUM  statistics can optimally distinguish   $H_0$ and $H_a$ and achieve the asymptotic  detection boundary. To the best of our knowledge, no  existing testing procedure can achieve this optimal detection boundary in the high-dimensional regression change point  model  with the presence of temporal dependence.

\section{An  adaptive testing procedure}
\label{sec:adaptive}
\subsection{Estimating $\sigma_\epsilon^2 $}
In this section, we provide a method to   consistently  estimate the parameter $ \sigma_\epsilon^2 = Var(\epsilon_i) $.  Note that for any interval $\I$ with diverging length,  
$  \frac{1}{|\I| } \sum_{i\in \I } \epsilon_i^2 \to  Var(\epsilon_i)   $ in probability.   
Since
$ \epsilon_i^2 = (y_i-x_i^\top \beta _i ^*)^2 $, a natural plug-in estimator for $\sigma_\epsilon^2$  is 
$$ \frac{1}{|\I| } \sum_{i\in \I } (y_i-x_i^\top \widehat \beta _\I )^2  , $$
where  $\widehat \beta_ \I $ is  defined   in  \eqref{eq:interval lasso}. This leads us to propose the procedure outlined in \Cref{algorithm:variance}, and  the consistency of \Cref{algorithm:variance}  is justified in \Cref{lemma:consistent estimate of variance under dependence}. 
\begin{algorithm}[ht]
\begin{algorithmic}
	\INPUT Data $\{ y_i, x_i \}_{i =1}^{n }$, tuning parameter  $\lambda$.
	
\For { $m$ in $\{1 ,\ldots,  \lfloor \log(n)\rfloor\}$}	
	\State $$ \mathcal J_m = \begin{cases}  \bigg(  \{ m-1 \}  \cdot \lfloor n/\log(n)  \rfloor      ,  \ 
	 m \cdot  \lfloor n/\log(n)  \rfloor     \bigg  ]   &\text{when }  m  \le  \lfloor \log(n)\rfloor-1,
	 \\\bigg(    \{ m-1 \}  \cdot \lfloor n/\log(n)  \rfloor      ,  \ 
	n  \bigg ]  &\text{when }  m  = \lfloor \log(n)\rfloor  .
	 \end{cases} 
	 $$
	 \State Set  
	 $$ \widehat    \sigma^2_ m   =    \frac{1}{|\mathcal J _m | } \sum_{i\in \mathcal J _m  } (y_i-x_i^\top \widehat \beta _{\mathcal J _m}  )^2   $$
	 \State where 
	 \begin{align} 
	\widehat \beta_{\mathcal J _m}    : = \argmin_{  \beta \in \mathbb R^p  } \frac{1}{|{\mathcal J _m}  | }  \sum_{i\in {\mathcal J _m}   }  \left(  y_i  -  x_i ^\top  \beta  \right)^2 
	+ \frac{ \lambda }{  \sqrt {  |{\mathcal J _m}  | }  } \| \beta\|_1,  
\end{align}
	\EndFor
 
\State Set  	
$   \widehat  \sigma_\epsilon^2   $ to be the median of $\{ \widehat    \sigma^2_ m\}_{m=1} ^{\lfloor \log(n)\rfloor} $.   
 	
	\OUTPUT  The estimated  variance $\widehat  \sigma_\epsilon^2 $.  
 	
 \caption{Variance Estimation.  }
\label{algorithm:variance}
\end{algorithmic}
\end{algorithm}

\begin{theorem} \label{lemma:consistent estimate of variance under dependence}  Suppose \Cref{assume:beta} and \ref{assume:order_betamixing} hold and that
$   \s\log(p)\log(n)   =o(n)   . $   Let $ \widehat \sigma_\epsilon^2$ be the output of \Cref{algorithm:variance} with  $\lambda = C_\lambda \sqrt { \log p }$ for some sufficiently large constant $C_\lambda $.   Then under $H_0$ or $H_a$, 
$$  \widehat  \sigma_\epsilon^  2    \overset{\p}{\to} \sigma_\epsilon^2 .  $$

\end{theorem}
 
\Cref{lemma:consistent estimate of variance under dependence} demonstrates that we can consistently estimate \( \sigma_\epsilon^2 \) without prior knowledge of the unknown change points. Estimating the variance \( \sigma_\epsilon^2 \) is crucial in high-dimensional regression, as it is a key parameter for many inference tools, including confidence intervals \citep{cai2017confidence}, signal-to-noise ratio estimation \citep{verzelen2018adaptive}, and genetic relatedness \citep{guo2019optimal}. To the best of our knowledge, we are the first to provide a consistent estimator for \( \sigma_\epsilon^2 \) when the high-dimensional regression time series are both non-stationary and temporally dependent.

\subsection{Estimating the sparsity}
 Let \( S_i \) denote the support set of \( \beta_i^* \), as defined in \Cref{assume: model assumption}{\bf c}. We impose the following assumption on the coefficients of the regression parameters:

\begin{assumption}  \label{assume:recovery signal size} 
For every \( i \in \{1, \ldots, n\} \), suppose that for some constant \( 0 < \theta < 1/4 \),
$$ \min_{j \in S_i} |\beta_i^*(j)| \geq n^{-\theta}. $$
\end{assumption}

\Cref{assume:recovery signal size} ensures that the non-zero entries of the regression coefficients do not diminish faster than \( n^{-1/4} \). This is a mild assumption, commonly found in the high-dimensional variable-screening literature, such as in \cite{fan2008sure} and \cite{fan2010sure}. In \Cref{algorithm:sparsity}, we provide a procedure to consistently estimate the order of the sparsity parameter \( \s \).

\begin{algorithm}[ht]
\begin{algorithmic}
	\INPUT Data $\{ y_i, x_i \}_{i =1 }^n  $, tuning parameter  $\lambda $.
	
\For { $m$ in $\{1 ,\ldots,  \lfloor \log(n)\rfloor\}$}	
	\State $$ \mathcal J_m = \begin{cases}  \bigg(  \{ m-1 \}  \cdot \lfloor n/\log(n)  \rfloor      ,  \ 
	 m \cdot  \lfloor n/\log(n)  \rfloor     \bigg  ]   &\text{when }  m  \le  \lfloor \log(n)\rfloor-1,
	 \\\bigg(    \{ m-1 \}  \cdot \lfloor n/\log(n)  \rfloor      ,  \ 
	n  \bigg ]  &\text{when }  m  = \lfloor \log(n)\rfloor  .
	 \end{cases} 
	 $$
	 \State Set  
	 $$\widehat S_{\mathcal J_m }   =  \bigg\{j:   |\widehat  \beta _{\mathcal J_m }  (j) | \ge   n^{-1/4}  \bigg \}   $$
	 \State where 
	 \begin{align*} 
	\widehat \beta_{\mathcal J _m}    : = \argmin_{  \beta \in \mathbb R^p  } \frac{1}{|{\mathcal J _m}  | }  \sum_{i\in {\mathcal J _m}   }  \left(  y_i  -  x_i ^\top  \beta  \right)^2 
	+ \frac{ \lambda }{  \sqrt {  |{\mathcal J _m}  | }  } \| \beta\|_1.  
\end{align*}
 \State Let $\widehat \s_m = |\widehat S_{\mathcal J_m }    |   $
	\EndFor

\State Set  	
$   \widehat  \s    $ to be the median of $\{ \widehat    \s_m \}_{m=1} ^{\lfloor \log(n)\rfloor} $.   
 	 
	\OUTPUT     $\widehat \s$.  
 	
 \caption{Variable-screening for   time series  data. }
\label{algorithm:sparsity}
\end{algorithmic}
\end{algorithm}

\begin{theorem}\label{lemma:consistent estimate of sparisty under dependence}
 Suppose \Cref{assume:beta}, \ref{assume:order_betamixing} and  \ref{assume:recovery signal size} hold.   Let $\widehat \s  $ be the output of  \Cref{algorithm:sparsity} with  $\lambda = C_\lambda \sqrt { \log p }$ for some sufficiently large constant $C_\lambda $.  Then under     $H_0$ or  $H_a$, it holds that 
$$ \p(  \s \le \widehat \s \le \log(n) \s  ) \ge 1-n^{-2}.$$
\end{theorem}

\subsection{An adaptive procedure}
In this section, we propose a practical procedure for change-point testing using the pre-estimates of variance and sparsity. Let \( \widehat{\sigma}_\epsilon^2 \) be the output of \Cref{algorithm:variance}, and let \( \widehat{\s} \) be the output of \Cref{algorithm:sparsity}. Define  
\[ \widehat{\sigma}_\xi = A_n \frac{\widehat{\s} \log(p)}{\sqrt{n}}, \]
where \( A_n \) is any slowly diverging sequence, and \( \xi_i' \overset{i.i.d.}{\sim} N(0, \widehat{\sigma}_\xi^2) \).
We further define the following test statistic:
\begin{align}\label{eq:practical qf_random 2}
	\mathcal{S}'_n(t) &= \frac{1}{2}(\widehat{\Delta}_t^\top \widehat{\Sigma}_{(0,t]} \widehat{\Delta}_t + \widehat{\Delta}_t^\top \widehat{\Sigma}_{(t,n]} \widehat{\Delta}_t) \nonumber \\
	&\quad + \frac{1}{t} \sum_{i=1}^t (2\widehat{\Delta}_t^\top x_i + \xi_i')(y_i - x_i^\top \widehat{\beta}_{(0,t]}) - \frac{1}{n-t} \sum_{i=t+1}^n (2\widehat{\Delta}_t^\top x_i + \xi_i')(y_i - x_i^\top \widehat{\beta}_{(t,n]}),
\end{align}
and the corresponding CUSUM-based statistic:
\begin{align}\label{eq:practical qf_cusum}
	\mathcal{T}_n'(t) = \frac{1}{\widehat{\sigma}_\epsilon \widehat{\sigma}_\xi} \sqrt{\frac{t(n-t)}{n}} \mathcal{S}'_n(t).
\end{align}

 \begin{theorem} \label{theorem:practical main_beta}
	Let \( \zeta > 0 \) be any fixed constant in \( (0,1/2) \). Suppose \Cref{assume:beta}, \Cref{assume:order_betamixing}, and \Cref{assume:recovery signal size} hold, and \( \lambda = C_\lambda \sqrt{ \log p } \) for some sufficiently large constant \( C_\lambda \). Under \( H_0 \), we have that
	$$
	\mathcal{T}_n'(\lfloor nr \rfloor) \Rightarrow \mathcal{G}(r), \text{ over } r \in [\zeta, 1-\zeta], \quad \text{and} \quad \max_{t = \lfloor n\zeta \rfloor, \lfloor n\zeta \rfloor + 1, \dots, \lfloor n(1-\zeta) \rfloor} \mathcal{T}_n'(t) \overset{d}{\to} \sup_{r \in [\zeta, 1-\zeta]} \mathcal{G}(r),
	$$
	where \( \mathcal{G}(r) \) is the Gaussian process defined in \Cref{theorem:main_beta}. Under \( H_a \), suppose the maximum change size satisfies 
	$$
	\max_{1 \leq k \leq K } \kappa_k^2 \geq B_n \frac{\s \log p}{n}
	$$
	for some diverging sequence \( B_n \to \infty \) and \( B_n / A_n \to \infty \) as \( n \to \infty \). We then have that
	\begin{align*}
		\mathbb{P} \left( \max_{t = \lfloor n\zeta \rfloor, \lfloor n\zeta \rfloor + 1, \dots, \lfloor n(1-\zeta) \rfloor} \mathcal{T}_n'(t) > \mathcal{G}_\alpha(\zeta) \right) \to 1 \quad \text{as} \quad n \to \infty.
	\end{align*}
\end{theorem}

\Cref{theorem:practical main_beta} implies that \( \mathcal{T}_n'(t) \) has the same asymptotic properties as the original QF-CUSUM statistic \( \mathcal{T}_n(t) \), while \( \mathcal{T}_n'(t) \) does not require the prior knowledge of  the population parameters such as \( \sigma_\epsilon^2 \) or \( \s \) as inputs.

\section{Change point localization}
\label{sec:localization}

In what follows, we study change-point localization based on the QF-CUSUM testing statistics. Specifically, suppose that under the alternative hypothesis, there exists a change point \( 0 < \eta < n \) such that:
\[
H_a: \beta_1^* = \cdots = \beta_\eta^* \neq \beta_{\eta+1}^* = \cdots = \beta_n^*.
\]
Let 
\begin{align}\label{eq:localization qf_random}
	\widetilde{\mathcal{S}}_n(t) &= \frac{1}{2} \left( \widehat{\Delta}_t^\top \widehat{\Sigma}_{(0,t]} \widehat{\Delta}_t + \widehat{\Delta}_t^\top \widehat{\Sigma}_{(t,n]} \widehat{\Delta}_t \right) \nonumber\\
	&\quad + \frac{1}{t} \sum_{i=1}^t (2 \widehat{\Delta}_t^\top x_i)(y_i - x_i^\top \widehat{\beta}_{(0,t]}) - \frac{1}{n-t} \sum_{i=t+1}^n (2 \widehat{\Delta}_t^\top x_i)(y_i - x_i^\top \widehat{\beta}_{(t,n]}),
\end{align}
and define the change-point estimator as:
\begin{align}\label{eq:change estimation}
	\widehat{\eta} = \argmax_{\lfloor \zeta n \rfloor \leq t \leq \lfloor (1-\zeta) n \rfloor} \left\{ \frac{(n-t)t}{n} \right\} \widetilde{\mathcal{S}}_n(t).
\end{align}

The localization   statistic \( \widetilde{\mathcal{S}}_n(t) \)  is derived by removing the user-generated random variables \( \{ \xi_i \}_{i=1}^n \) from the original QF-CUSUM statistic \( \mathcal{S}_n(t) \). The following theorem establishes the consistency of the of our proposed estimator $\widehat \eta$.

 \begin{theorem}
 \label{theorem:localization bound}
 For any generic interval $ \I \subset (0, n]$, let  $\widehat \beta_\I $ be defined as in \ref{eq:interval lasso} with  $\lambda= C_\lambda \sqrt {\log(p)}$ for sufficiently large constant $C_\lambda$.  Suppose \Cref{assume:beta} and \ref{assume:order_betamixing} hold and that $\eta \in \{ \lfloor \zeta n \rfloor  , \ldots , \lfloor (1- \zeta) n \rfloor \}   $. Then with probability at least $ 1-n^{-2}$, it holds that 
$$ |\widehat \eta -\eta| \le C\frac{\s\log(p)}{\kappa^2 },$$ 
 where $C$ is a sufficiently large constant independent of $n$.
 \end{theorem}
Note that in \Cref{theorem:localization bound}, the localization error bound    is derived under conditions of temporal dependence and heavy-tailed distributions.  It matches the sharpest localization error bound in the literature (see \cite{rinaldo2021localizing}), which was attained in the setting of temporal independence.

\section{Numerical Studies}\label{sec:numerical}
In this section, we conduct extensive numerical experiments to investigate the performance of the proposed  QF-CUSUM for change-point testing in high-dimensional linear models. \Cref{subsec_size} examines the finite-sample size of the test and \Cref{subsec_power} studies the power performance of the test under single and multiple change-points. \Cref{subsec_realdata} presents a real data application on macroeconomics to further illustrate the practical utility of the proposed test.

To our best knowledge, there is no existing change-point test designed for high-dimensional linear models available in the literature. Thus, we instead compare QF-CUSUM with three change-point estimation algorithms for high-dimensional linear models: BSA in \cite{Leonardi2016}, VPC in \cite{Wang2021a} and SGL in \cite{Zhang2015a}. We emphasize that the comparison is for reference purposes only, as QF-CUSUM focuses on testing while BSA, VPC and SGL are designed for estimation.

BSA estimates change-points via the minimization of an $l_0$-penalized information criteria, where the minimizer is searched via a heuristic procedure based on binary segmentation. VPC transforms change-point estimation for high-dimensional linear models to change-point detection in the mean of a univariate time series via a projection strategy. SGL recasts change-point estimation  into a sparse group Lasso~\citep{Simon2013} problem. We remark that theoretical guarantees of the three algorithms all require temporal independence.

For all three algorithms, we keep the tuning parameters at the default setting recommended by the authors. We rule that BSA rejects $H_0$ if the set of estimated change-points is non-empty, and vice-versa. Same applies to VPC and SGL.

\textbf{Estimation of key quantities for QF-CUSUM}: To operationalize QF-CUSUM, we need to select the tuning parameter $\lambda$ of the Lasso estimation and further estimate the noise level $\sigma_\epsilon$ and sparsity level $\s$. To do so, we make the mild assumption that the first and last $\zeta$-proportion of the observations, $\{(x_i,y_i)\}_{i=1}^{\zeta n}$ and $\{(x_i,y_i)\}_{i=(1-\zeta)n}^{n}$, are stationary segments without change. As discussed before, this is a common assumption in the change-point testing literature and is also needed for the process convergence result in \Cref{theorem:main null}, see \cite{andrews1993} and references therein. For the choice of the trimming parameter $\zeta$, following the recommendation in \cite{andrews1993}, we set $\zeta=0.15$ throughout this section.

Given $\{(x_i,y_i)\}_{i=1}^{\zeta n}$, we conduct a standard 10-fold cross-validation to select tuning parameters $\lambda_{pre}$. Based on $\lambda_{pre}$, we estimate the sparsity level $\widehat{\s}_{pre}$ of $\beta^{pre}$ in $\{(x_i,y_i)\}_{i=1}^{\zeta n}$ as the number of nonzero entries in $\widehat{\beta}^{pre}$ and estimate the noise level via $\widehat{\sigma}_{\epsilon}^{pre}=\sqrt{\sum_{i=1}^{\zeta n}(y_i-x_i^\top \widehat{\beta}^{pre})^2/(\zeta n-\widehat{\s}_{pre})}$. We estimate $\widehat{\s}_{post}$ and  $\widehat{\sigma}_{\epsilon}^{post}$ for $\{(x_i,y_i)\}_{i=(1-\zeta)n}^{n}$ in the same way. Finally, we set $\lambda=(\lambda_{pre}+\lambda_{post})/2$, $\widehat{\s}=(\widehat{\s}_{pre}+\widehat{\s}_{post})/2$ and $\widehat{\sigma}_{\epsilon}=(\widehat{\sigma}_{\epsilon}^{pre}+\widehat{\sigma}_{\epsilon}^{post})/2.$ Based on \Cref{assume:order}, we set the variance level of the random error as $\sigma_\xi=\widehat{\s}\log p/\sqrt{n}\times \log\log n.$

As can be seen from \Cref{assume:order} and \Cref{theorem:main_beta}, for the validity of QF-CUSUM under $H_0$, the theoretical choice for $ \sigma_\xi $ only requires the knowledge of the upper bound of  $ \s$, though an unnecessarily large upper bound will negatively impact the power of QF-CUSUM. For high-dimensional linear models, theoretical guarantees for the estimation of $\s$ have been well studied under mild conditions in the variable selection literature, see \cite{meinshausen2006high}, \cite{wainwright2009sharp}, \cite{fan2010selective} and references therein. In addition, consistency for the estimation of the variance level $\sigma^2_\epsilon$ has also been established in the high-dimensional literature, see for example \cite{Sun2012} and \cite{Guo2019}.

\subsection{Size performance}\label{subsec_size}
We simulate the data from the high-dimensional linear model
\begin{align*}
	y_i=x_i^\top \beta^*+\epsilon_i, ~~ i=1,2,\cdots,n,
\end{align*}
where $x_i\in\mathbb{R}^p$ denotes the $p$-dimensional covariates and $\epsilon_i$ is the noise term with $\sigma^2_{\epsilon}=1$.

We vary the sample size $n$ across $\{200,400\}$ and the dimension $p$ across $\{100,200,400\}$. For the covariance matrix $\Sigma$ of $x_i$, we consider two cases: Toeplitz $\Sigma_{ij}=0.6^{|i-j|}$ or compound symmetric $\Sigma_{ij}=0.3\mathbb{I}(i\neq j)+\mathbb{I}(i=j).$ Denote $s$ as the sparsity of $\beta^*$, we vary $s$ across $\{5,10\}$. Given $s$ and $p$, define $\beta^o(i)=i/s$ for $i=1,2,\cdots,s$ and $\beta^o(i)=0$ for $s<i\leq p$. We set $\beta^*=3\beta^o/\sqrt{\beta^{o\top}\Sigma\beta^o}$ to keep the quadratic form at the same magnitude across different simulation settings.

As for the temporal dependence among $\{x_i\}_{i=1}^n$ and $\{\epsilon_i\}_{i=1}^n$, we consider three cases: independence~(Ind.), auto-regressive~(AR) dependence, and moving-average~(MA) dependence. For AR dependence, we generate $\{x_i\}_{i=1}^n$ via $x_i=0.3x_{i-1}+\sqrt{1-0.3^2}\mathbf{e}_i$ with $\mathbf{e}_i\overset{i.i.d.}{\sim} N(0,\mathbb{I}_p)$ and generate $\{\epsilon_i\}_{i=1}^n$ via $\epsilon_i=0.3\epsilon_{i-1}+\sqrt{1-0.3^2}{e}'_i$ with ${e}'_i\overset{i.i.d.}{\sim} N(0,1)$. For MA dependence, we generate $\{x_i\}_{i=1}^n$ via $x_i=(\mathbf{e}_i+0.4\mathbf{e}_{i-1})/\sqrt{1+0.4^2}$ with $\mathbf{e}_i\overset{i.i.d.}{\sim} N(0,\mathbb{I}_p)$ and generate $\{\epsilon_i\}_{i=1}^n$ via $\epsilon_i=(e'_i+0.4{e}'_{i-1})/\sqrt{1+0.4^2}$ with ${e}'_i\overset{i.i.d.}{\sim} N(0,1)$.

For each simulation setting with different $(n,p,s,\Sigma)$ and temporal dependence, we set the target size at 5\% and repeat the experiments 500 times. \Cref{tab:size} summarizes the empirical size of QF-CUSUM across different simulation settings. In general, the size of QF-CUSUM is reasonable and improves with the sample size $n$. When the sample size is small ($n=200$), a large dimension $p$ seems to cause an undersize of the test, and Toeplitz covariance exhibits more notable oversize than the CS covariance. However, these phenomena alleviate significantly when $n$ increases, suggesting the validity of our asymptotic theory. The temporal dependence does inflate the empirical size, however, the inflation is not significant especially for $n=400.$ We note that \Cref{theorem:main null} requires $\s\log(p)/\sqrt{n}\to 0$ for the convergence of QF-CUSUM to the Gaussian process, which is not always the case for the simulation settings we consider. However, the empirical size achieved here is still reasonable, indicating the robustness of QF-CUSUM.

\begin{table}[ht]
	\centering
	\begin{tabular}{lcc|cc|cc}
		\hline\hline
		&\multicolumn{2}{c|}{Ind.} & \multicolumn{2}{c|}{AR} & \multicolumn{2}{c}{MA}\\\hline 
	    $(n,p,s)~ \Large/ ~\Sigma$	& T & CS & T & CS & T & CS \\ 
		\hline
		(200,100,5) & 6.60 & 5.20 & 7.20 & 5.20 & 8.80 & 3.80 \\ 
		(200,200,5) & 3.60 & 1.60 & 7.20 & 3.80 & 4.80 & 2.20 \\ 
		(200,400,5) & 2.80 & 1.40 & 4.80 & 2.40 & 3.60 & 2.20 \\ 
		(200,100,10) & 6.00 & 6.00 & 6.80 & 6.20 & 8.20 & 5.60 \\ 
		(200,200,10) & 2.20 & 2.60 & 4.20 & 4.00 & 5.00 & 2.80 \\ 
		(200,400,10) & 2.20 & 2.20 & 1.80 & 1.60 & 1.00 & 2.00 \\ \hline
		(400,100,5) & 6.40 & 4.20 & 5.40 & 5.80 & 7.20 & 4.60 \\ 
		(400,200,5) & 4.80 & 2.80 & 6.60 & 5.60 & 7.40 & 5.40 \\ 
		(400,400,5) & 4.80 & 4.40 & 5.80 & 4.80 & 7.00 & 4.80 \\ 
		(400,100,10) & 6.60 & 4.80 & 5.60 & 5.40 & 6.60 & 5.20 \\ 
		(400,200,10) & 4.00 & 3.00 & 5.40 & 5.20 & 7.40 & 6.00 \\ 
		(400,400,10) & 4.40 & 4.40 & 5.80 & 4.20 & 5.00 & 3.60 \\ 
		\hline\hline
	\end{tabular}
	\caption{Empirical size of QF-CUSUM averaged over 500 experiments under different simulation settings. T stands for Toeplitz and CS stands for compound symmetric.} 
    \label{tab:size}
\end{table}

\Cref{tab:size_competing} summarizes the empirical size of BSA, VPC and SGL. Interestingly, the three algorithms exhibit quite different patterns. BSA almost never rejects the null hypothesis with an empirical size close to 0, which is later shown to negatively impact its ability to detect small-scale changes. VPC also under-rejects under temporal independence, but suffers from false positives for the case of temporal dependence, especially for the MA setting, where the rejection rate can be as high as 15\%. SGL suffers from over-rejection quite a bit across all simulation settings.

\begin{table}[ht]
	\centering
	\setlength{\tabcolsep}{0.4em} 
	\begin{tabular}{lcccccc|cccccc|cccccc}
		\hline\hline
		&\multicolumn{6}{c|}{Ind.} & \multicolumn{6}{c|}{AR} & \multicolumn{6}{c}{MA}\\\hline 
		&\multicolumn{2}{c|}{BSA} & \multicolumn{2}{c|}{VPC} & \multicolumn{2}{c|}{SGL} &\multicolumn{2}{c|}{BSA} & \multicolumn{2}{c|}{VPC} & \multicolumn{2}{c|}{SGL}
		&\multicolumn{2}{c|}{BSA} & \multicolumn{2}{c|}{VPC} & \multicolumn{2}{c}{SGL}\\\hline
$(n,p,s)~ \Large/ ~\Sigma$	& T & CS & T & CS & T & CS & T & CS & T & CS & T & CS & T & CS & T & CS & T & CS \\ 
\hline
		(200,100,5) & 0.0 & 0.0 & 1.2 & 0.0 & 70 & 78 & 0.0 & 0.0 & 11.8 & 8.2 & 71 & 76 & 0.0 & 0.0 & 12.6 & 8.2 & 72 & 74 \\ 
		(200,200,5) & 0.0 & 0.0 & 2.0 & 0.4 & 59 & 79 & 0.0 & 0.0 & 11.8 & 6.4 & 60 & 75 & 0.0 & 0.0 & 13.0 & 8.0 & 59 & 80 \\ 
		(200,400,5) & 0.0 & 0.0 & 1.6 & 0.4 & 75 & 76 & 0.0 & 0.0 & 9.4 & 6.2 & 75 & 74 & 0.0 & 0.0 & 12.2 & 6.0 & 79 & 79 \\ 
		(200,100,10) & 0.0 & 0.0 & 0.2 & 0.2 & 60 & 77 & 0.0 & 0.0 & 7.8 & 6.0 & 55 & 75 & 0.0 & 0.0 & 8.4 & 7.2 & 54 & 74 \\ 
		(200,200,10) & 0.0 & 0.0 & 0.4 & 0.0 & 60 & 77 & 0.0 & 0.0 & 8.4 & 5.2 & 66 & 71 & 0.0 & 0.0 & 8.6 & 5.6 & 62 & 74 \\ 
		(200,400,10) & 0.0 & 0.0 & 0.2 & 0.0 & 73 & 70 & 0.0 & 0.0 & 7.0 & 5.4 & 76 & 68 & 0.0 & 0.0 & 6.4 & 5.2 & 75 & 71 \\ \hline
		(400,100,5) & 0.2 & 0.4 & 2.4 & 0.0 & 92 & 91 & 0.0 & 0.2 & 13.4 & 7.2 & 94 & 92 & 0.2 & 0.4 & 15.2 & 11.4 & 93 & 91 \\ 
		(400,200,5) & 0.0 & 0.0 & 1.8 & 0.2 & 90 & 84 & 0.0 & 0.0 & 12.6 & 8.0 & 90 & 84 & 0.0 & 0.0 & 12.6 & 9.6 & 87 & 84 \\ 
		(400,400,5) & 0.0 & 0.0 & 1.0 & 0.0 & 81 & 82 & 0.0 & 0.0 & 12.8 & 6.2 & 81 & 81 & 0.0 & 0.0 & 15.2 & 8.2 & 82 & 85 \\ 
		(400,100,10) & 0.2 & 0.0 & 1.4 & 0.4 & 88 & 85 & 0.4 & 0.2 & 8.8 & 6.0 & 89 & 86 & 0.4 & 0.0 & 12.2 & 8.0 & 89 & 87 \\ 
		(400,200,10) & 0.0 & 0.0 & 0.2 & 0.0 & 87 & 73 & 0.0 & 0.0 & 8.2 & 5.4 & 88 & 72 & 0.0 & 0.0 & 12.8 & 5.6 & 83 & 76 \\ 
		(400,400,10) & 0.0 & 0.0 & 0.0 & 0.0 & 77 & 76 & 0.0 & 0.0 & 6.8 & 5.2 & 82 & 73 & 0.0 & 0.0 & 10.4 & 5.4 & 82 & 76 \\ 		
		\hline\hline
	\end{tabular}
	\caption{Empirical size of BSA, VPC and SGL averaged over 500 experiments under different simulation settings. T stands for Toeplitz and CS stands for compound symmetric.} 
\label{tab:size_competing}
\end{table}

\subsection{Power performance}\label{subsec_power}
In this section, we further investigate the power performance of QF-CUSUM under both single and multiple change-point settings. We further compare with BSA and VPC. We remove SGL from comparison as it is quite sensitive to false positives as shown in \Cref{subsec_size}. For the single change-point scenario, we consider
\begin{align*}
y_i=\begin{cases}
	&x_i^\top \beta^* +\epsilon_i,  \text{ if } 1\leq i\leq n/2\\
	&x_i^\top \beta^*(1+\kappa) +\epsilon_i,  \text{ if } n/2+1\leq i\leq n,
\end{cases}	
\end{align*}
and for the multiple change-point scenario, we consider
\begin{align*}
	y_i=\begin{cases}
		&x_i^\top \beta^* +\epsilon_i,  \text{ if } 1\leq i\leq n/3\\
		&x_i^\top \beta^*(1+\kappa) +\epsilon_i,  \text{ if } n/3+1\leq i\leq 2n/3,\\
		&x_i^\top \beta^* +\epsilon_i, \text{ if } 2n/3+1\leq i\leq n.
	\end{cases}	
\end{align*}
We follow the same simulation setting as in \Cref{subsec_size}. The only difference is that we vary $\kappa$ to control the change size. For the same $\kappa$, the change size across different simulation settings $(n,p,s,\Sigma)$ is the same as $\beta^*$ is normalized (see \Cref{subsec_size} for details). We vary the change size $\kappa^2$ across $\{0,0.125,0.25,0.5,0.75,1\}$ to study the power performance. Note that we reduce to the no change-point setting in \Cref{subsec_size} for $\kappa=0.$

\Cref{fig:power_ind_single}~(left column) gives the size-adjusted power curves of QF-CUSUM under the single change-point case. To conserve space, we present the result with temporal independence. \Cref{subsec:additional_num} further presents results under the AR and MA dependence, where the phenomenon observed is essentially the same, suggesting the robustness of QF-CUSUM to temporal dependence. In general, with the same signal $\kappa^2$, the power of QF-CUSUM decreases as the dimension $p$ and sparsity $s$ grow and increases as the sample size $n$ increases. In addition, changes under the CS covariance is more difficult to be detected than the ones under the Toeplitz covariance.

For comparison, \Cref{fig:power_ind_single}~(right column) gives the power curves of BSA and VPC under the same simulation setting. To conserve space, we only provide the result for $n=200$ as the result for $n=400$ is similar. Compared to QF-CUSUM, VPC provides similar~(and for some cases slightly better) power performance under the single change-point setting. On the other hand, BSA seems to suffer from power loss when the change size is small and is notably sensitive to the dimension $p.$

\Cref{fig:power_ind_multiple}~(left column) gives the size-adjusted power curves of QF-CUSUM for the case of multiple change-points. The phenomenon is similar to the one in \Cref{fig:power_ind_single}. Compared to \Cref{fig:power_ind_single}, it can be seen that QF-CUSUM experiences power loss due to the presence of non-monotonic change. However, QF-CUSUM is still able to detect the structural breaks for strong signals, which can be seen as numerical evidence for \Cref{corollary:iid power}.

For comparison, \Cref{fig:power_ind_multiple}~(right column) gives the power curves of BSA and VPC under the same simulation setting. Compared to QF-CUSUM, VPC provides similar but slightly worse power performance, especially when the change size is small. BSA again suffers from power loss and is notably sensitive to the dimension $p.$

\begin{figure}
	\centering
	\vspace{-0.8cm}
	\begin{subfigure}[b]{0.4\textwidth}
		\centering
		\includegraphics[width=\textwidth, angle=270]{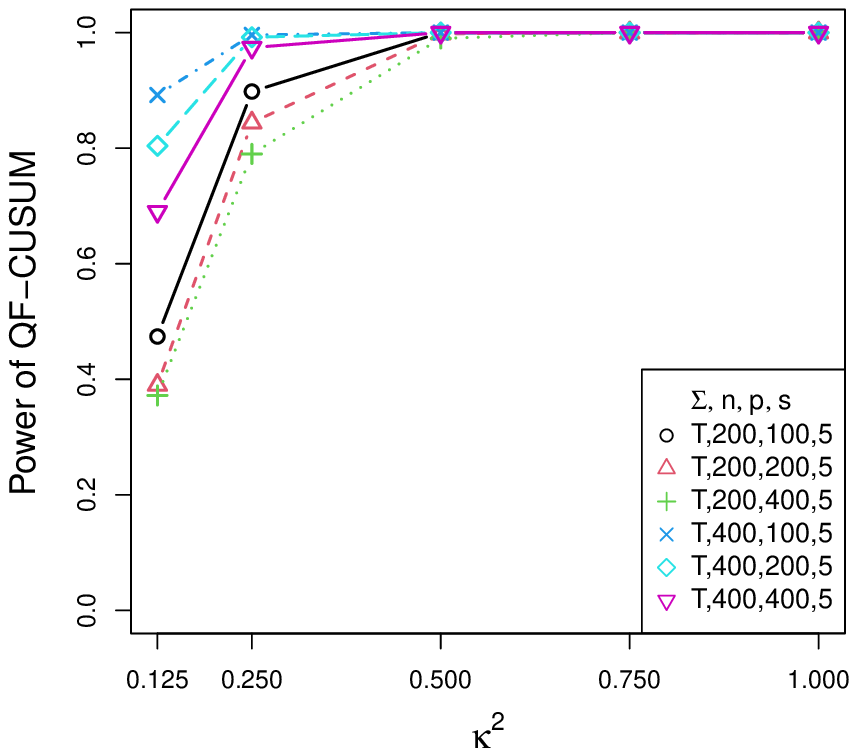}
	\end{subfigure}
	\hspace{-0.1cm} 	\vspace{-0.8cm}
	\begin{subfigure}[b]{0.4\textwidth}
		\centering
		\includegraphics[width=\textwidth, angle=270]{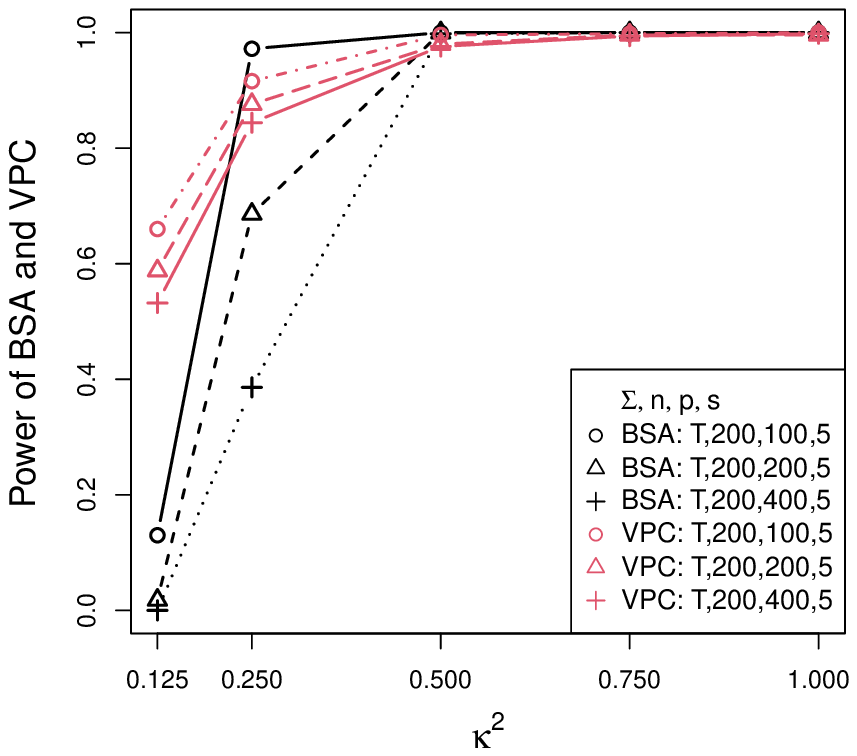}
	\end{subfigure}
	\vspace{-0.8cm}
	\begin{subfigure}[b]{0.4\textwidth}
		\centering
		\includegraphics[width=\textwidth, angle=270]{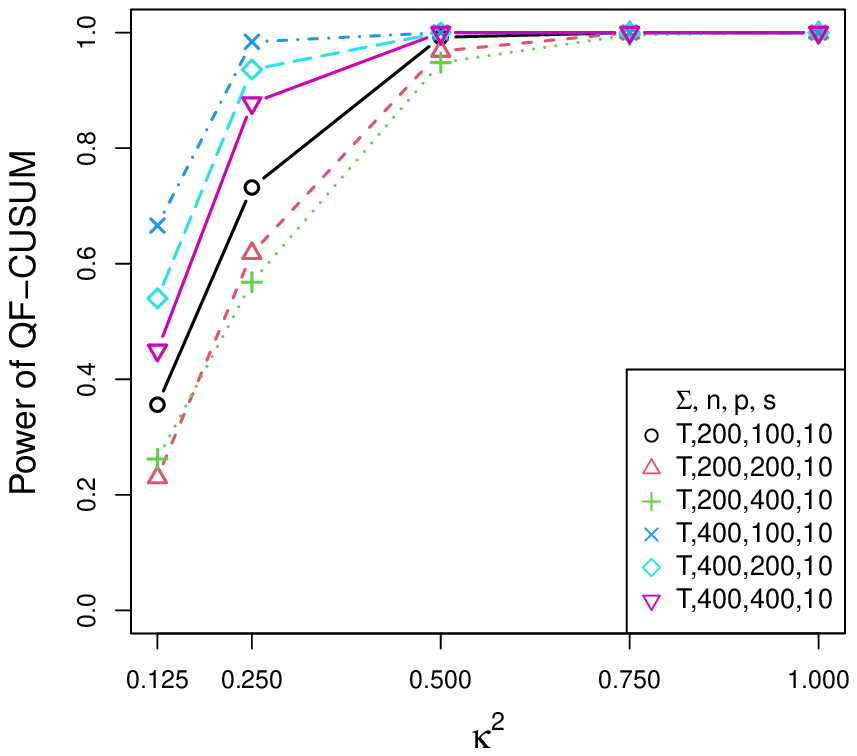}
	\end{subfigure}
	\hspace{-0.1cm}
	\begin{subfigure}[b]{0.4\textwidth}
		\centering
		\includegraphics[width=\textwidth, angle=270]{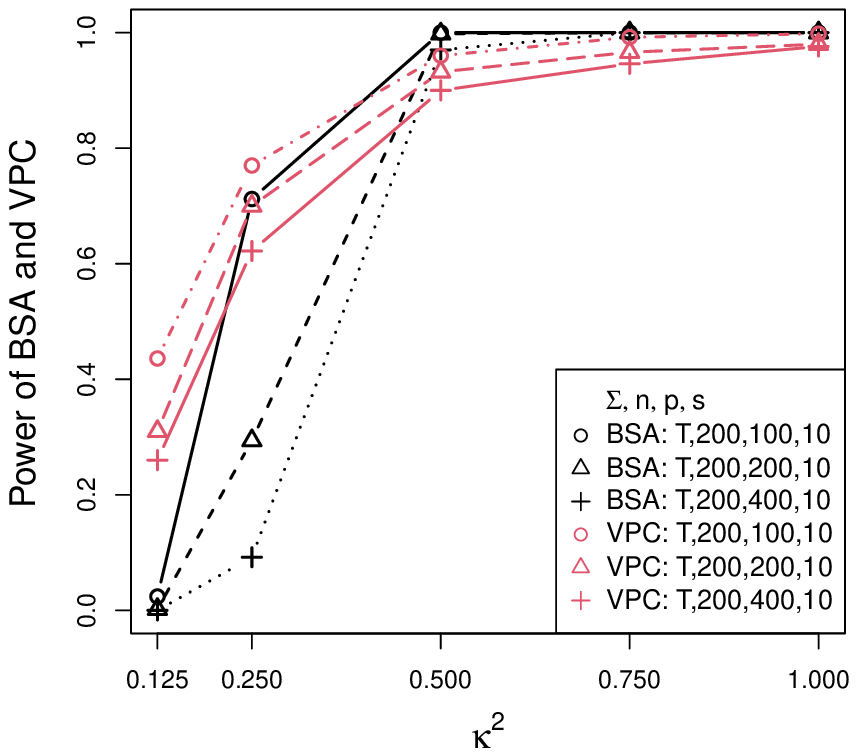}
	\end{subfigure}
	\vspace{-0.8cm}
	\begin{subfigure}[b]{0.4\textwidth}
		\centering
		\includegraphics[width=\textwidth, angle=270]{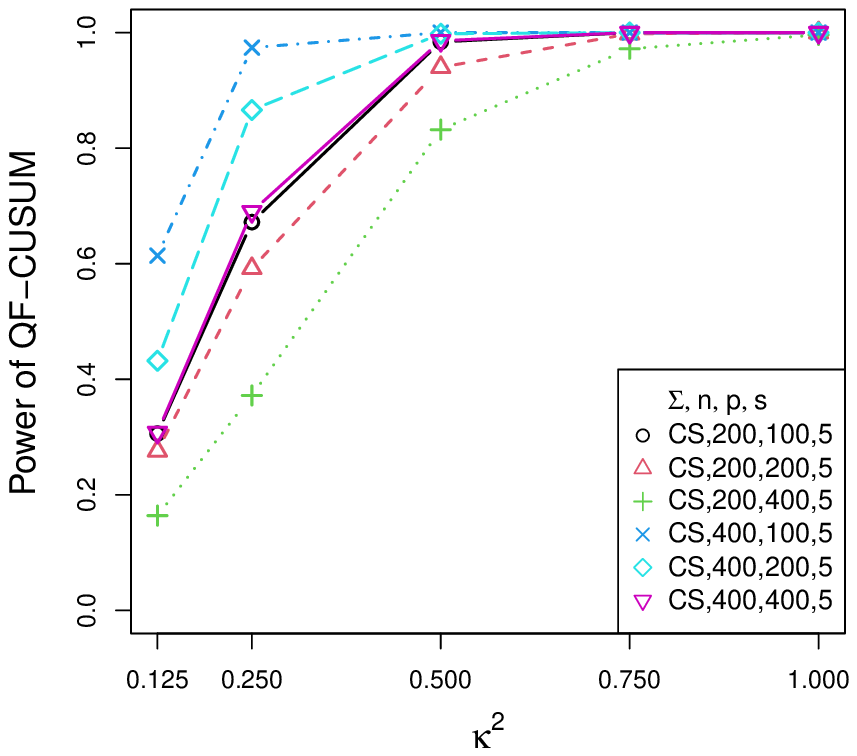}
	\end{subfigure}
	\hspace{-0.1cm}
	\begin{subfigure}[b]{0.4\textwidth}
		\centering
		\includegraphics[width=\textwidth, angle=270]{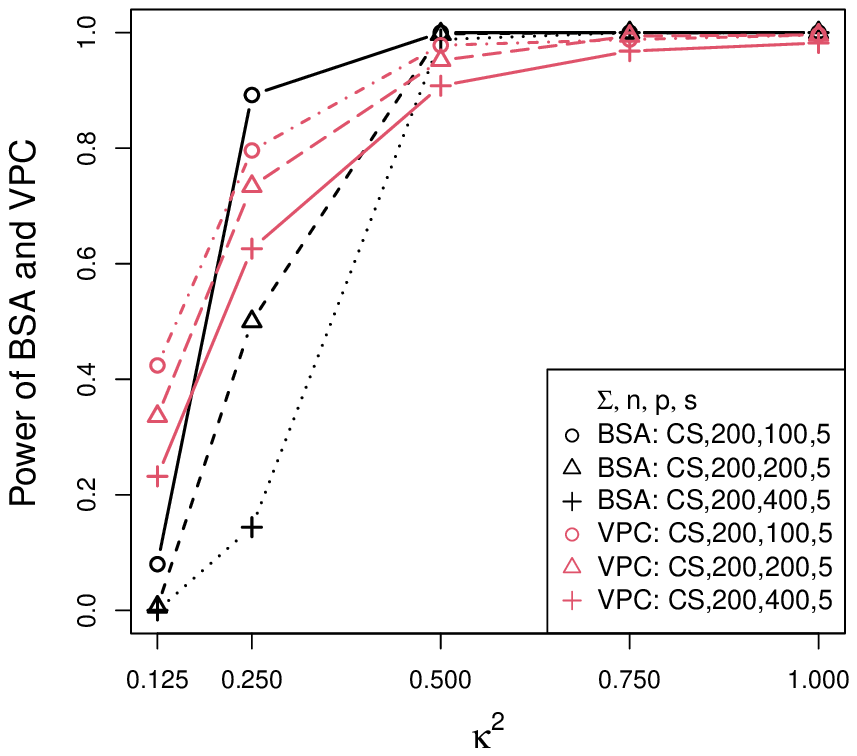}
	\end{subfigure}
	\vspace{-0.2cm}
	\begin{subfigure}[b]{0.4\textwidth}
		\centering
		\includegraphics[width=\textwidth, angle=270]{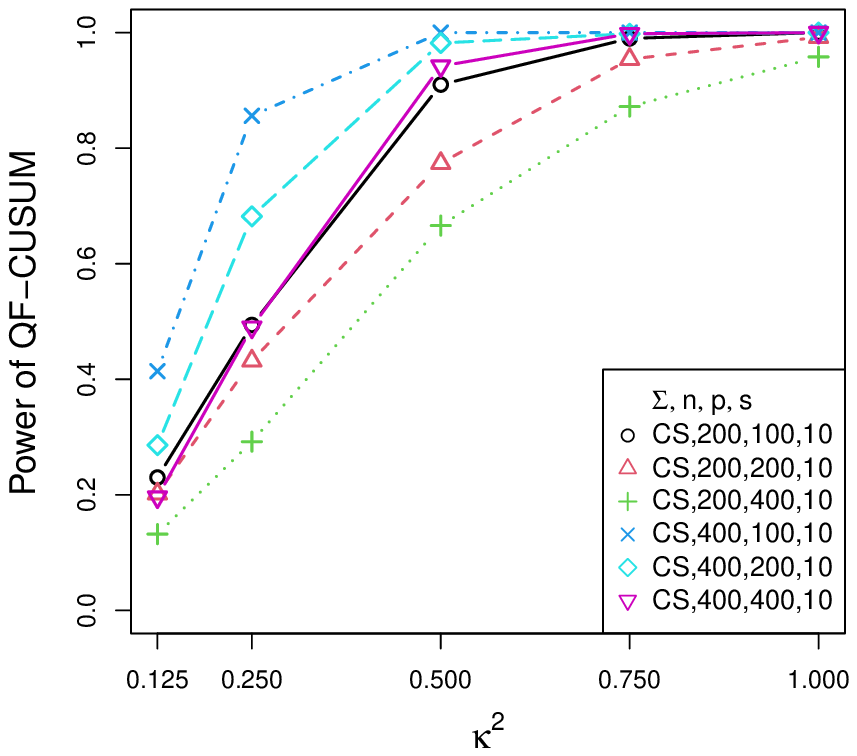}
	\end{subfigure}
	\hspace{-0.1cm}
	\begin{subfigure}[b]{0.4\textwidth}
		\centering
		\includegraphics[width=\textwidth, angle=270]{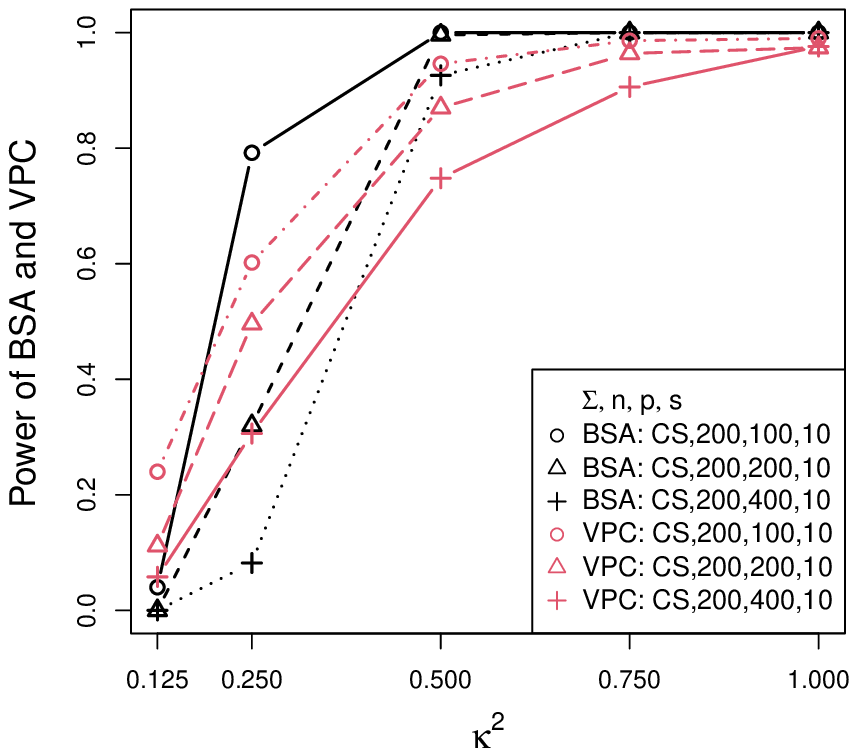}
	\end{subfigure}	
	\caption{Power performance of QF-CUSUM~(left column), BSA and VPC~(right column) for different simulation settings under the single change-point case with temporal independence.}
	\label{fig:power_ind_single}
\end{figure}

\begin{figure}
	\centering
	\vspace{-0.8cm}
	\begin{subfigure}[b]{0.4\textwidth}
		\centering
		\includegraphics[width=\textwidth, angle=270]{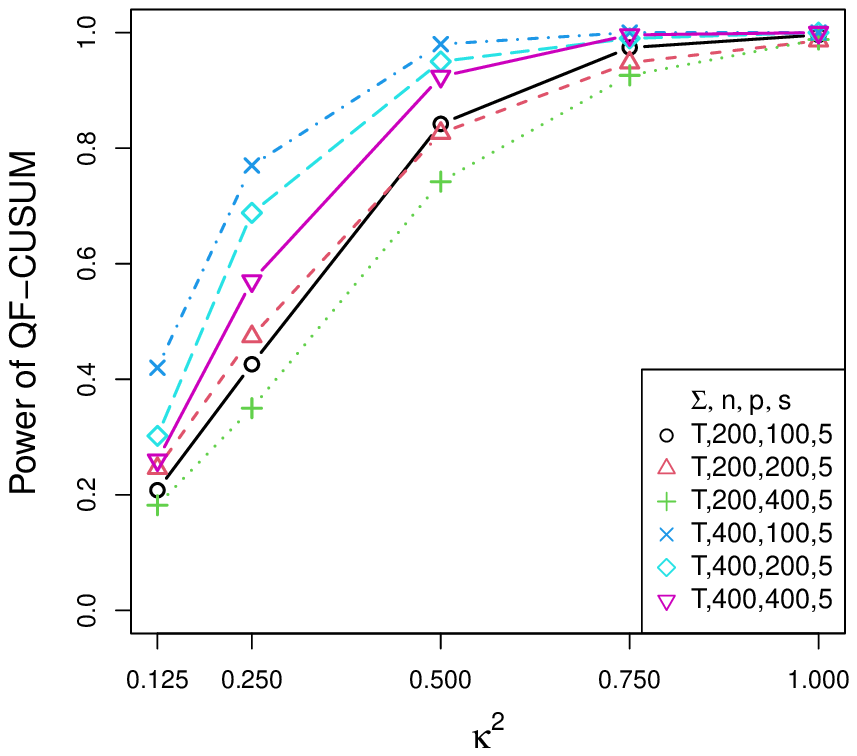}
	\end{subfigure}
	\hspace{-0.1cm} 	\vspace{-0.8cm}
	\begin{subfigure}[b]{0.4\textwidth}
		\centering
		\includegraphics[width=\textwidth, angle=270]{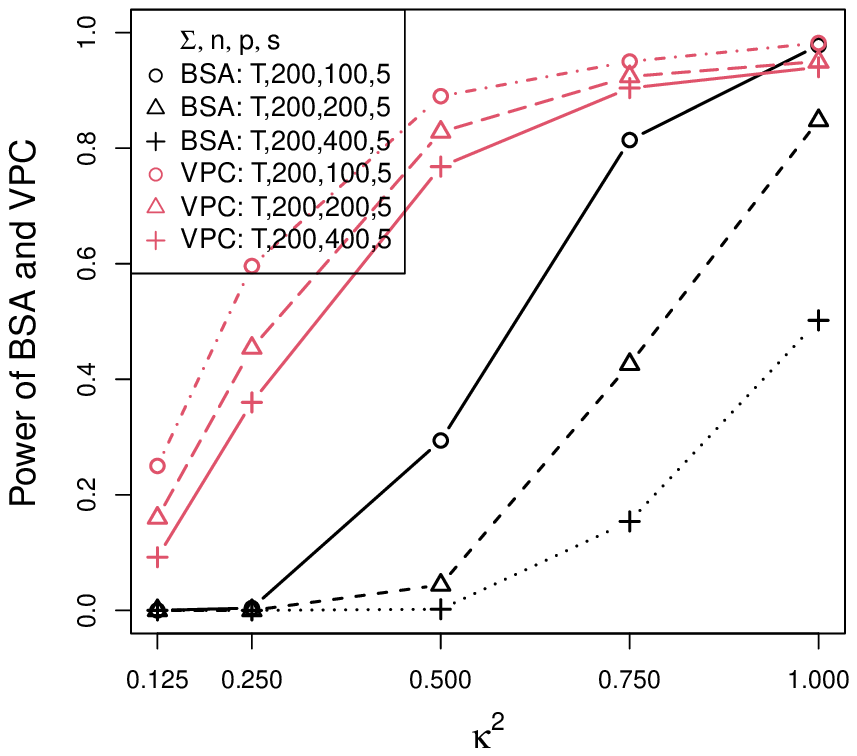}
	\end{subfigure}
	\vspace{-0.8cm}
	\begin{subfigure}[b]{0.4\textwidth}
		\centering
		\includegraphics[width=\textwidth, angle=270]{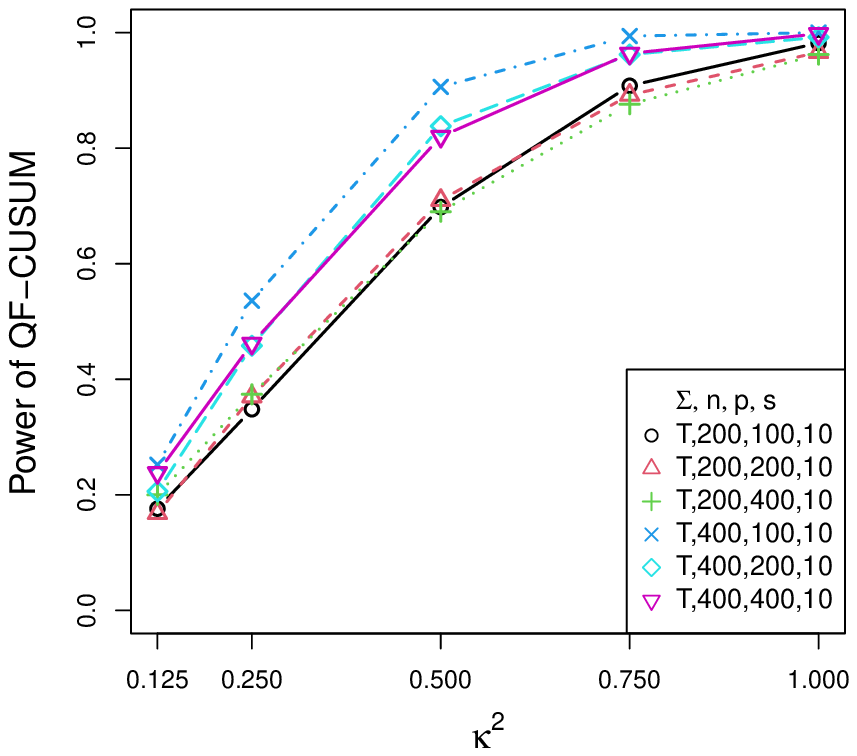}
	\end{subfigure}
	\hspace{-0.1cm}
	\begin{subfigure}[b]{0.4\textwidth}
		\centering
		\includegraphics[width=\textwidth, angle=270]{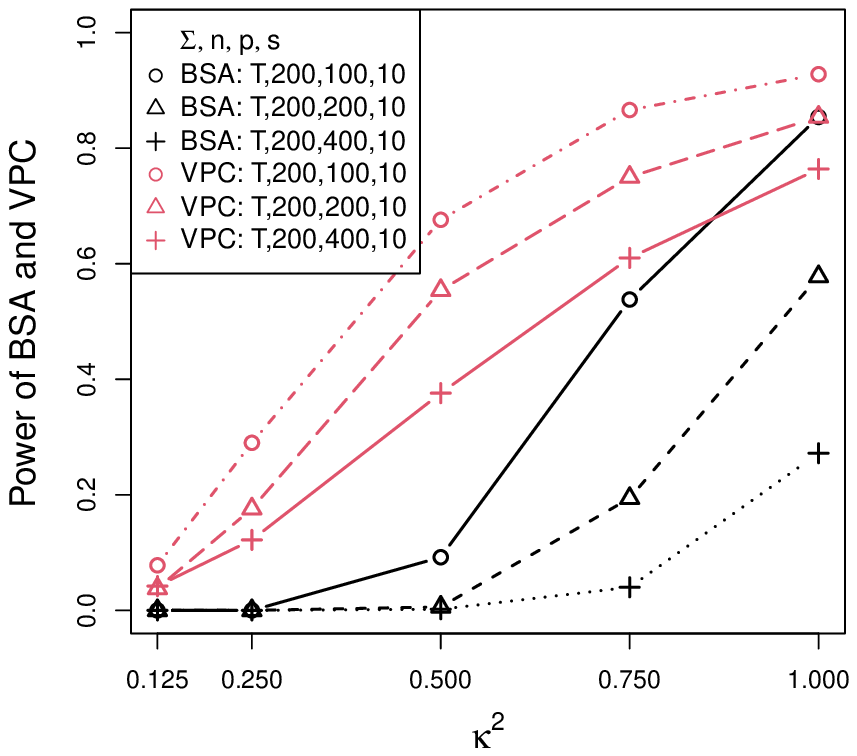}
	\end{subfigure}
	\vspace{-0.8cm}
	\begin{subfigure}[b]{0.4\textwidth}
		\centering
		\includegraphics[width=\textwidth, angle=270]{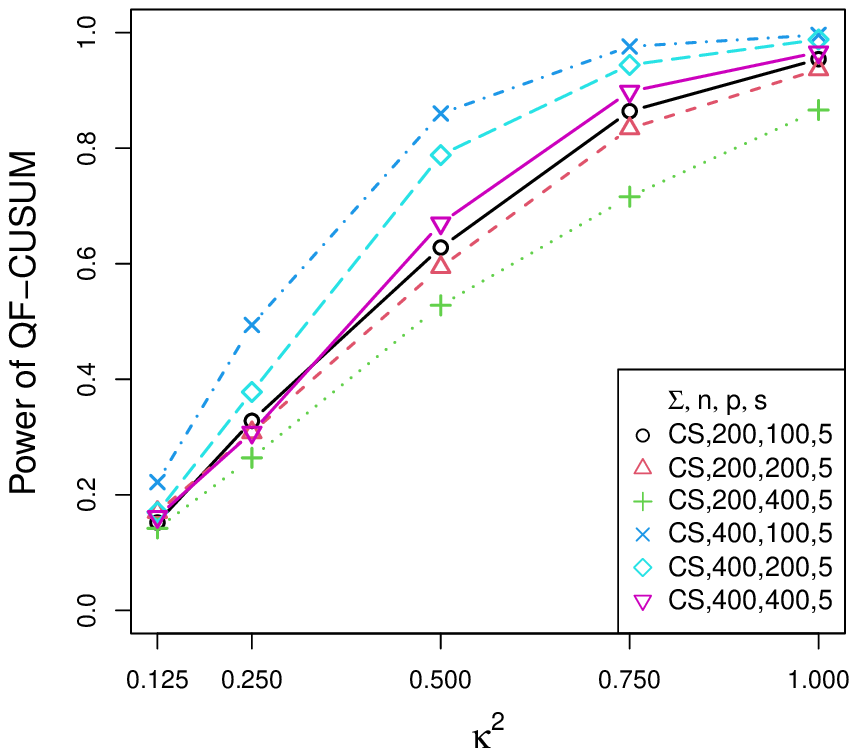}
	\end{subfigure}
	\hspace{-0.1cm}
	\begin{subfigure}[b]{0.4\textwidth}
		\centering
		\includegraphics[width=\textwidth, angle=270]{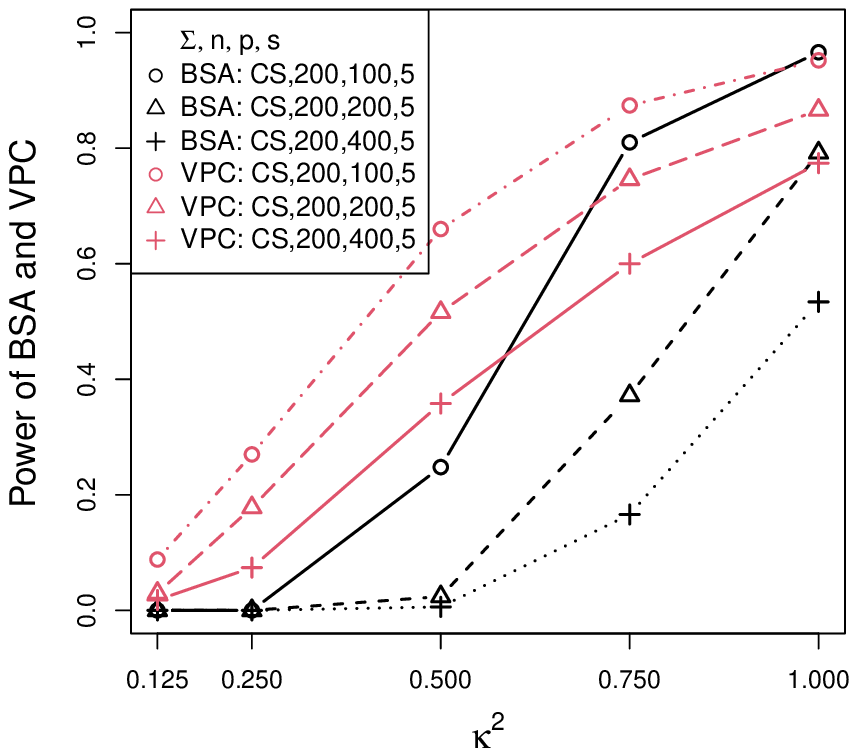}
	\end{subfigure}
	\vspace{-0.2cm}
	\begin{subfigure}[b]{0.4\textwidth}
		\centering
		\includegraphics[width=\textwidth, angle=270]{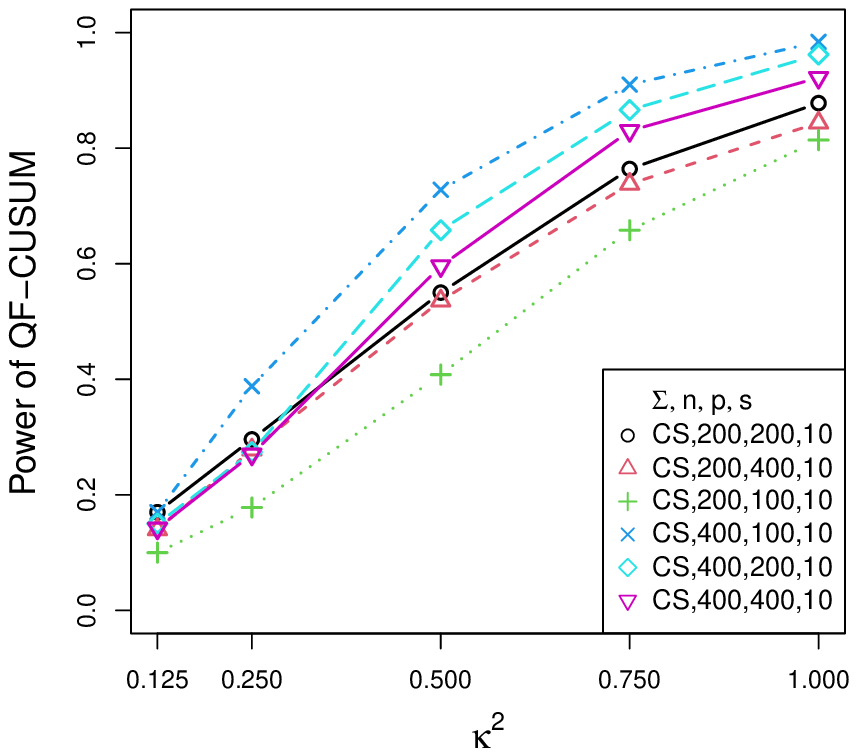}
	\end{subfigure}
	\hspace{-0.1cm}
	\begin{subfigure}[b]{0.4\textwidth}
		\centering
		\includegraphics[width=\textwidth, angle=270]{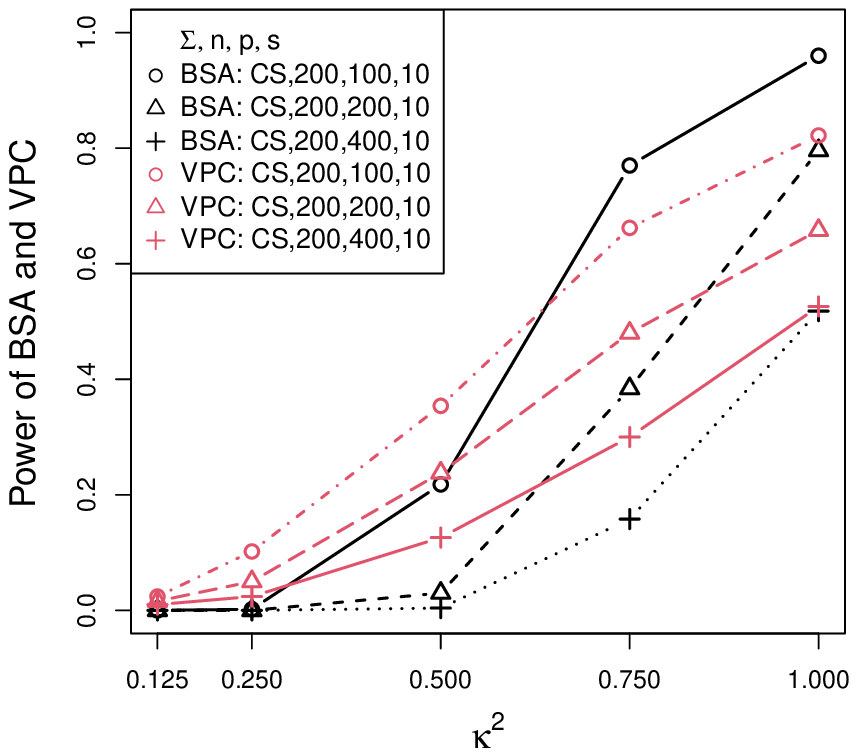}
	\end{subfigure}	
	\caption{Power performance of QF-CUSUM~(left column), BSA and VPC~(right column) for different simulation settings under the multiple change-point case with temporal independence.}
	\label{fig:power_ind_multiple}
\end{figure}

\textbf{Localization result}: QF-CUSUM is proposed for change-point testing, i.e.\ to detect the existence of change-points. On the other hand, after the null hypothesis is rejected, it is possible to further use the test statistic to estimate the change-point location. Though this is not the focus of our paper, we give a brief numerical illustration here for the case of single change-point estimation. Specifically, once QF-CUSUM rejects $H_0$, we estimate the location of the change-point as
\begin{align}\label{eq:cp_loc_est}
	\widehat{\eta}=\max_{t=\lfloor n\zeta\rfloor, \lfloor n\zeta\rfloor+1, \cdots, \lfloor n(1-\zeta)\rfloor}\widetilde{\T}_n(t),
\end{align}
where $\widetilde{\T}_n(t)$ equals to $\T_n(t)$ defined in \eqref{eq:qf_cusum} with $\xi_i\equiv 0$ for all $i$. We remove the randomized error $\xi$ for change-point localization as it is mainly introduced to ensure a pivotal and non-degenerate asymptotic distribution under $H_0$ but does not play a role under $H_a$.

For illustration, we focus on the simulation setting with temporal independence and Toeplitz covariance matrix, and set $\s=5$. We vary the sample size $n$ across $\{200,400\}$, the dimension $p$ across $\{100,200,400\}$ and vary the change size $\kappa^2$ across $\{0.25,0.5,1\}$. \Cref{fig:power_ind_loc} provides the kernel density estimation for $\widehat{\eta}/n$ based on 500 experiments for each simulation setting. As can be seen, $\widehat{\eta}$ centers around the true change-point and its accuracy improves as $n$ increases and is relatively robust to the dimension $p.$

\begin{figure}
	\centering
	\includegraphics[angle=270, width=0.9\textwidth]{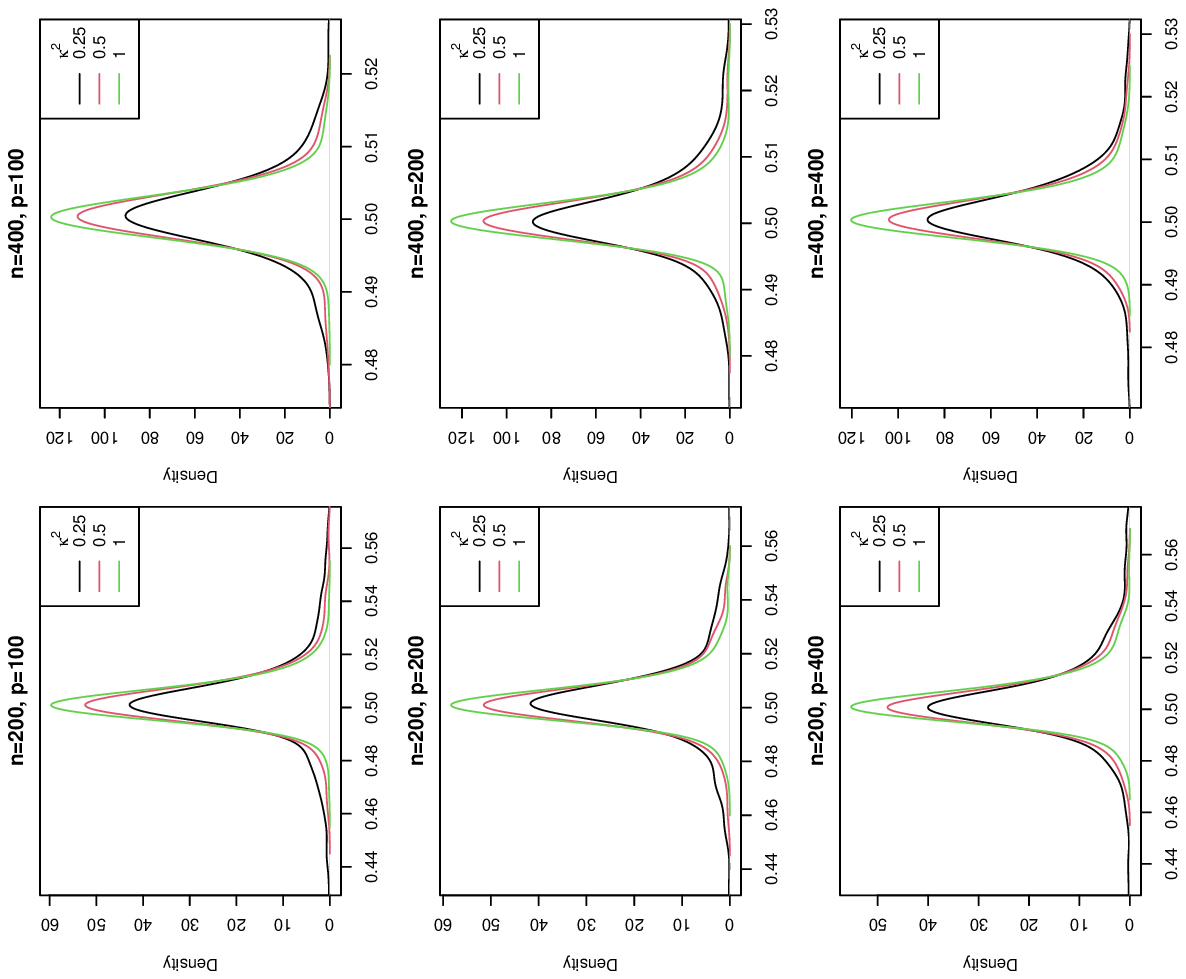}
	\caption{Localization performance of QF-CUSUM for difference simulation settings under the single change-point case with temporal independence and $\Sigma=$T and $s=5$.}
	\label{fig:power_ind_loc}
\end{figure}

\subsection{Real data application}\label{subsec_realdata}
In this section, we apply QF-CUSUM to test if there is any structural break in the relationship between the monthly growth rate of the US industrial production~(IP) index, an important indicator of macroeconomic activity, and 127 other macroeconomic variables from Federal Reserve Economic Database~(FRED-MD)\footnote{The dataset is publicly available at \url{https://research.stlouisfed.org/econ/mccracken/fred-databases}.}~\citep{McCracken2016}. The empirical analysis in \cite{He2022} suggests that a high-dimensional linear model based on the 127 predictors are overall significant for forecasting IP. Here, we further test if there is any change in the high-dimensional linear model.

Specifically, our response variable is $y_t=\log(\text{IP}_t/\text{IP}_{t-1})\times 100$, where $\text{IP}_t$ denotes the US industrial production index for the month $t$. In other words, $y_t$ measures the monthly growth rate of IP~(in percentage scale). The high-dimensional covariate includes 127 macroeconomic variables recorded for the month $t-1.$ We transform the raw data, i.e.\ the response and each individual covariate, into stationary time series and remove outliers using the MATLAB codes provided on the FRED-MD website, see also \cite{McCracken2016} for more details. We focus our analysis on the time period from June 2005 to March 2022, which includes $n=200$ observations.

The implementation of QF-CUSUM follows the same setting as in Sections \ref{subsec_size}-\ref{subsec_power}. The only difference is that we set the trimming parameter $\zeta=24/200=0.12$ instead of $0.15$ to make the choice more interpretable~(i.e.\ 2 years). The result based on $\zeta=0.15$ is essentially the same. \Cref{fig:realdata}~(left) plots the QF-CUSUM statistic $\T_n(t)$ computed over the sample, where the null hypothesis of no change-point is clearly rejected at the 5\% nominal level. \Cref{fig:realdata}~(right) further plots the statistic $\widetilde{\T}_n(t)$ defined in \eqref{eq:cp_loc_est}, which gives an estimated change-point at Oct.\ 2019, close to the beginning of the Covid-19 pandemic. Note that the two statistics $\T_n(t)$ and $\widetilde{\T}_n(t)$ are similar to each other, indicating that the randomized error $\{\xi_i\}_{i=1}^n$ is rather negligible compared to the change size, which thus suggests the robustness of our finding.

We further implement BSA, VPC and SGL for comparison. BSA detects no change-point. VPC estimates two change-points, one at Jul.\ 2009 and one at May 2019. SGL estimates three change-points, Oct.\ 2005, Jun.\ 2020 and Apr.\ 2021. These results provide additional evidence that the Covid-19 pandemic seems to alter the relationship between the US industrial production index and the macroeconomic predictors considered. 

\begin{figure}
	\centering
	\hspace{-2cm}
	\includegraphics[angle=270, width=1\textwidth]{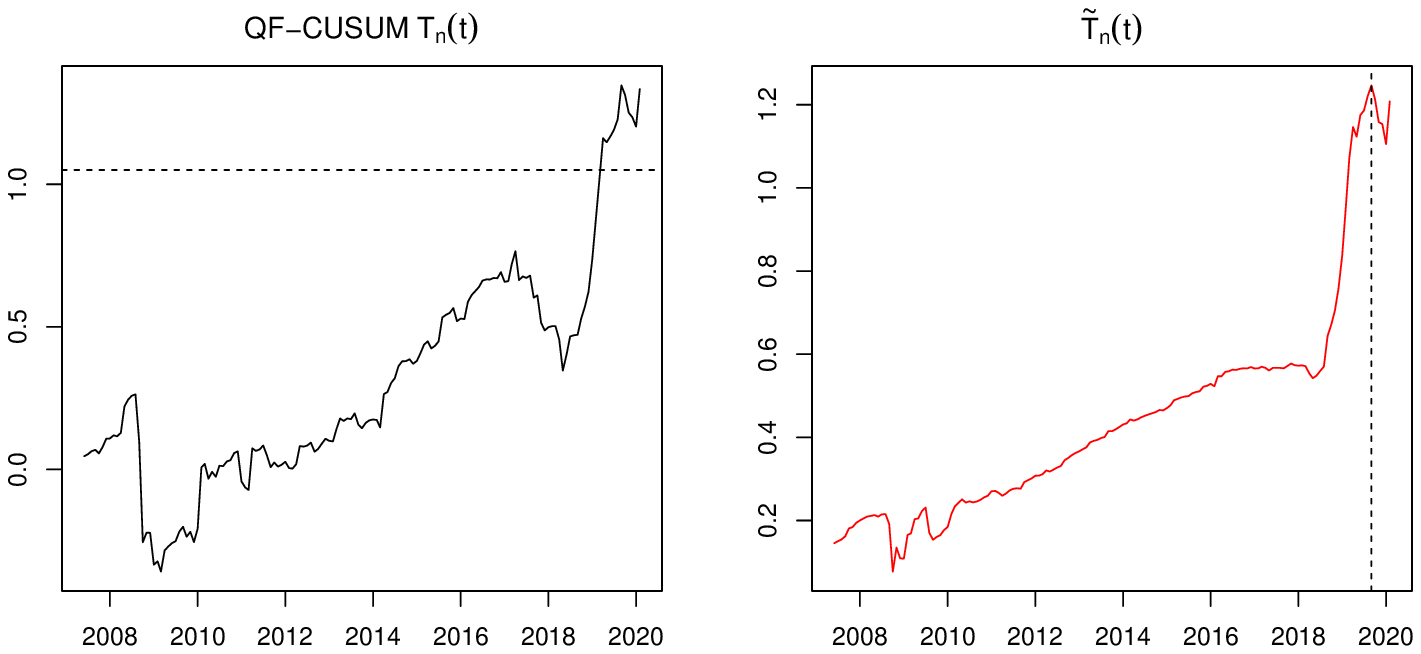}
	\caption{Left: The QF-CUSUM statistic $\T_n(t)$ computed on the FRED-MD data. The horizontal dashed line marks the 95\% critical value of QF-CUSUM. Right: The statistic $\widetilde{\T}_n(t)$ defined in \eqref{eq:cp_loc_est}. The vertical dashed line marks October 2019.}
	\label{fig:realdata}
\end{figure}

\section{Conclusion}\label{sec:conclusion}
In this paper, we study the problem of change-point testing for high-dimensional linear models. We propose QF-CUSUM, a quadratic-form based CUSUM test to inspect the stability of the regression coefficients in a high-dimensional linear model. QF-CUSUM can control the asymptotic type-I error at any desired level and is theoretically sound for temporally dependent observations. Furthermore, QF-CUSUM achieves the optimal minimax lower bound of the detection boundary for a wide class of high-dimensional linear models. Through extensive numerical experiments and a real data application in macroeconomics, we demonstrate the effectiveness and utility of QF-CUSUM. The focus of this paper is change-point testing. As for future research, a potential direction is to further investigate the theoretical properties of the change-point estimator based on QF-CUSUM, such as its consistency and optimal convergence rate.
 
\bibliographystyle{apalike}
\bibliography{reference}

\clearpage
\appendix
\section*{\LARGE Appendix}
Throughout the appendix, we  assume that $p\ge  n^{\alpha} $ for some $\alpha>0$. It follows that
$$ \log(pn) = O(  \log(p )). $$
This is a  convenience assumption commonly used in the high-dimensional statistical literature.  
\\
\\
In the appendix, we first justify \Cref{theorem:main_beta} in the i.i.d. sub-Gaussian setting. Then we use this framework to justify  \Cref{theorem:main_beta} in the beta-mixing sub-Weibull setting. 
\section{Main Results related to i.i.d. sub-Gaussian setting}
 Consider the following assumptions.

\begin{assumption}\label{assume: model assumption}
The observations $\{(x_i,y_i)\}_{i=1}^n$ follow model \eqref{eq:model} with independently and identically distributed covariates $\{x_i\}_{i=1}^n$ and random noise $\{\epsilon_i\}_{i=1}^n.$ 
\\  
 {\bf a.} (Sub-Gaussian) The random  covariate  $x_i$ is a p-dimensional sub-Gaussian random vector with  $\mathbb E(x_i)=0  $ and  $  \| x_i\|_{\psi_2 } =K_X $.  The random  noise $\epsilon_i$ is a sub-Gaussian random variable independent of $x_i$ with  $\mathbb E(\epsilon_i)=0$, $ \text{Var}(\epsilon_i ) = \sigma_\epsilon^2$ and  $  \| \epsilon_i\|_{\psi_2 } =K_\epsilon $.   Here, $K_X$ and $K_\epsilon$ are absolute constants $<\infty$.
\\
	{\bf b.} (Eigenvalue)   Denote $\Sigma=\text{Cov}(x_i)$, there exist  absolute constant $c_x$ and $C_x$ such that the minimal and maximal eigenvalues of $\Sigma$ satisfy 
	  $ \Lambda_{\min}   (\Sigma)\ge c_x >0  $   and $\Lambda_{\max} (\Sigma ) \le C_x < \infty .   $ 
\\	  
	{\bf c.} (Sparsity) For each $i=1,\ldots, n$, there exists a support set $S_i \subseteq\{1, \ldots, p \}$  such that 
	  $$  \beta^* _ { i,j} =0 \text{ for all } j \not \in S_i.$$
	  In addition, there exists a sparsity parameter $1\leq \s\leq p$ such that the cardinality of the support set satisfies $|S_i|\le \s$ for all $i=1,2,\cdots, n.$
\\
	{\bf d.} (Infill) Under $H_a$, there exist a fixed set of (relative) change-points 
	$ 0<  \eta_1^* < \eta_2^* < \ldots <\eta_K^* <1 $ 
	such that $\eta_k=\lfloor n\eta_k^*\rfloor$ for $k=1,\cdots,K.$
\\
	{\bf e.} (SNR) There exists a sufficiently large absolute constant $C_{snr}$ such that   $n \ge C_{snr}\s\log(p) $.
\end{assumption}

\begin{assumption}\label{assume:order}
	We have that $$\frac{\s\log p  }{\sqrt n } \to 0 \text{ and } \sigma_\xi= A_n\cdot \frac{ \s\log p}{\sqrt{n}}$$ for some diverging sequence $A_n \to \infty.$ 
\end{assumption}     
  
Our goal is to show the following theorem, which can be considered as a simplified version of  \Cref{theorem:main_beta}  in
 the i.i.d. sub-Gaussian setting.

\begin{theorem} \label{theorem:main null}
	Let $\zeta>0$ be any fixed constant in $(0,1/2)$. Suppose \Cref{assume: model assumption} and \Cref{assume:order} hold, and $\lambda = C_\lambda \sqrt { \log p }$ for some sufficiently large constant $C_\lambda $. Under $H_0$, we have that
	$$\T_n(\lfloor nr \rfloor) \Rightarrow \mathcal{G}(r), \text{ over } r \in [\zeta, 1-\zeta],$$
	where $\mathcal G(r)=[B(r)-rB(1)]/\sqrt{r(1-r)}$ for $r\in (0,1)$ and $B(\cdot)$ is the standard Brownian motion. Furthermore, we have
	$$\max_{t=\lfloor n\zeta\rfloor, \lfloor n\zeta\rfloor+1, \cdots, \lfloor n(1-\zeta)\rfloor}\T_n(t) \overset{d}{\to} \sup_{r\in [\zeta,1-\zeta]}\mathcal{G}(r).$$
\end{theorem}

We now provide the power result of QF-CUSUM under $H_a$. Denote the change size $ \kappa_k: =  \| \beta^* _{\eta_k+1} - \beta_{\eta_k} ^* \|_2$ for $k=1,2,\cdots, K.$ \Cref{corollary:iid power} provides the detection boundary of QF-CUSUM.

\begin{theorem}   \label{corollary:iid power}
    Let $\zeta>0$ be any fixed constant in $(0,1/2)$. Suppose \Cref{assume: model assumption} and \Cref{assume:order} hold, and $\lambda = C_\lambda \sqrt { \log p }$ for some sufficiently large constant $C_\lambda $. Under $H_a$, suppose the maximum change size satisfies 
	\begin{align} \label{eq:detection boundary i.i.d}\max_{1\le k \le K } \kappa _k^2 \ge  B _n \frac{\s\log p}{n }
	\end{align}
	for some diverging  sequence $ B _n \to \infty $ and $B_n/A_n \to \infty  $ as $n \to \infty$. We have that
	\begin{align*}
		\mathbb{P}\left(\max_{t=\lfloor n\zeta\rfloor, \lfloor n\zeta\rfloor+1, \cdots, \lfloor n(1-\zeta)\rfloor}\T_n(t) >\mathcal{G}_\alpha(\zeta)\right) \to 1 \text{ as } n \to \infty.
	\end{align*}
\end{theorem}

\subsection{Generally conditions for \Cref{theorem:main null} and \Cref{corollary:iid power}}

We begin by introducing the following conditions.
\begin{condition} \label{condition:deviation}
Let $ \zeta \in (0,1/2)$ be  any  given constant and that $ \{ \xi_{i} \}_{i=1}^n \overset{i.i.d.} {\sim}  N(0, \sigma_\xi^2)  $ be a collection of user-generated random variables. Suppose that  for all interval $ \mathcal I  =(0,t] $ or $(t, n]$ such that $| \mathcal I | \ge \zeta n  $, the following additional conditions hold. 
\\
\\
 {\bf a.}  Let $\mathcal C_S: =\{v \in \mathbb R^p: \| v_{S^c}\|_1 \le 3\|v_S\|_1 \} $.   Suppose that
 $$\big|v^\top \big( \widehat \Sigma_{\I } -\Sigma \big) v  \big |\le C  \sqrt { \frac{\s \log(pn) }{ |\mathcal I|  }} \|v\|_2^2  \quad  \text{for all} \quad v \in \mathcal C_S.$$
 \\
 \\
{ \bf b.} Suppose that 
$$  \bigg | \frac{1}{| \mathcal I | } \sum_{i\in \I }  \epsilon_i x_i^\top \beta  \bigg|
\le C  \sqrt { \frac{\log(pn )}{ |\mathcal I|   }}\| \beta\|_1 \quad \text{for all} \quad \beta \in \mathbb R^p .  $$
\\
\\
{\bf c.} Let $\{u_i \}_{i=1}^n  $ be a collection of non-random vectors.  Suppose that $$
  \left |  \frac{1}{|\I | } \sum_{i \in \I }     u^\top_i  x_i   x_i^\top \beta  - \frac{1}{|\I | } \sum_{i \in \I }      u^\top  _i  \Sigma \beta  \right| \le C \left(  \max_{1\le i \le n }\|u_i\|_2  \right) \sqrt {  \frac{\log(pn)}{ | \I |  } } \|\beta\|_1   \quad  \text{for all } \beta \in \mathbb R^p  .$$
  \\
  \\
  {\bf d.} Suppose that $$  \bigg | \frac{1}{| \mathcal I | } \sum_{i\in \I }  \xi_i x_i^\top \beta  \bigg|
\le C  \sigma_\xi  \sqrt { \frac{\log(pn )}{ |\mathcal I|   }}\| \beta\|_1 \quad \text{for all} \quad \beta \in \mathbb R^p .  $$
\end{condition}

\begin{condition} \label{condition:process}
For any $r\in (0,1) $ let $t =\lfloor r n\rfloor   $ and 
$$\mathcal G_n (r ) =\frac{ \sqrt { n  }}{\sigma_\xi \sigma_\epsilon} \bigg\{  \frac{1}{t}\sum_{i=1}^t \epsilon_i\xi_i - \frac{1}{ n-t} \sum_{i=t+1}^{n } \epsilon_i\xi_i  \bigg\}   .$$ 
Suppose that for $r \in(0,1)$, 
$$ \mathcal G_n(r ) \overset{\mathcal D } {\to}  \mathcal G (r ), $$
where the convergence is in Skorokhod topology  and $\mathcal G(r ) $ is a Gaussian Process on $(0,1)$ with covariance function 
$$\sigma(r,s) =  \frac{1  }{ r(1-s)} \quad \text{when}  \quad  0< s\le r <1 .$$ 

\end{condition}
 
In \Cref{section:distribution}, we show that \Cref{condition:deviation} and  \Cref{condition:process}  hold  with high probability when the random quantities  $ \{x_i,\epsilon_i \}_{i=1}^n $ satisfy  \Cref{assume: model assumption}. 
 
\subsection{Null distribution under $H_0$}
\textbf{Proof of \Cref{theorem:main null}}:
\begin{proof}
 We justify all the technical results supporting \Cref{theorem:main null}  under \Cref{condition:deviation} and \Cref{condition:process}.  Let   $  \mathcal R_1 (t) 
   , \mathcal R_2(t)    , \mathcal R_3(t)      $ be defined as \Cref{theorem:alternative} and $  \mathcal R_1' (t) 
   , \mathcal R_2'(t)    , \mathcal R_3'(t)      $ be defined as \Cref{theorem:alternative counterpart}. 
   \
   \\
So with $t= \lfloor rn\rfloor $,
\begin{align*}   
\mathcal  T_n(r) :=  \frac{\sqrt {n }   }{2\sigma_\epsilon\sigma_\xi   }   \Bigg\{ & ( \widehat \beta_ {(0, t] }  -   \widehat \beta_{(t, n] } )^\top     \widehat \Sigma _{(0, t] } 
( \widehat \beta_{(0, t] }  -   \widehat \beta_{(0, t] }  )  + 
\mathcal R_1 (t) 
  +2 \mathcal R_2(t)  -2 \mathcal R_3(t)       
\\
 + &    ( \widehat \beta_ {(0, t] }  -   \widehat \beta_{(t, n] } )^\top     \widehat \Sigma _{(t,n ] } 
( \widehat \beta_{(0, t] }  -   \widehat \beta_{(0, t] }  )  + 
\mathcal R_1 '  (t) 
  +2 \mathcal R_2 ' (t)  -2 \mathcal R_3 '(t)       \Bigg\} .
\end{align*}
  \
  \\
Observe that  under $H_0 $, $ \beta_\ot ^* =\beta_\tn ^* =\beta_1^* $ for all $t\in  \{ 1,\dots, n\}$.
\
\\
 So  for all 
$t$, 
\begin{align*} 
& ( \beta^*_\ot  - \beta^*_\tn  ) ^\top  \Sigma  ( \beta^*_\ot  - \beta^*_\tn  )   = 0 
   \\
   & \mathfrak T_{n, \xi}(t) =  \frac{2}{t} \sum_{i=1}^t \epsilon_i \xi_i , \quad \mathfrak  S_n (t ) = 0 , \quad \mathfrak T_{n, \xi} '(t) =  \frac{2}{n-t} \sum_{i=t+1}^n \epsilon_i \xi_i , \quad \mathfrak S_n '(t )  
 = 0,
 \end{align*} 
where $ \mathfrak T_{n, \xi}(t),   \mathfrak  S_n (t ) $ are defined in \Cref{theorem:alternative} and 
$ \mathfrak T '_{n, \xi}(t),   \mathfrak S'_n (t ) $  are defined in  \Cref{theorem:alternative counterpart}.
Consequently  by \Cref{theorem:alternative}, it holds that 
 \begin{align*}
&  \left|   \frac{\sqrt {n }   }{\sigma_\epsilon\sigma_\xi   }   \Bigg\{   ( \widehat \beta_ {(0, t] }  -   \widehat \beta_{(t, n] } )^\top     \widehat \Sigma _{(0, t] } 
( \widehat \beta_{(0, t] }  -   \widehat \beta_{(0, t] }  )  + 
\mathcal R_1 (t) 
  +2 \mathcal R_2(t)  -2 \mathcal R_3(t)        - \frac{2}{t} \sum_{i=1}^t \epsilon_i \xi_i \bigg\}  \right| 
  \\
  \le &  \frac{\sqrt  {   n} }{\sigma_\epsilon\sigma_\xi    } C_1 (1+\sigma_\xi ) \frac{\s \log(pn)}{ n}   
  =    C_1 \frac{ \s \log(pn)}{ \sigma_\epsilon\sigma_\xi    \sqrt n } + C _1\frac{\s\log(pn )}{ \sigma_\epsilon    \sqrt n }  
  \\
  =&o(1).
 \end{align*}
 By \Cref{theorem:alternative counterpart}, it holds that 
 \begin{align*}
& \left|    \frac{\sqrt {n }   }{\sigma_\epsilon\sigma_\xi   }   \Bigg\{   ( \widehat \beta_ {(0, t] }  -   \widehat \beta_{(t, n] } )^\top     \widehat \Sigma _{ \tn } 
( \widehat \beta_{(0, t] }  -   \widehat \beta_{(0, t] }  )  + 
\mathcal R_1' (t) 
  +2 \mathcal R_2'(t)  -2 \mathcal R_3'(t)       +  \frac{2}{ n-t} \sum_{i=t+1}^n  \epsilon_i \xi_i \bigg\}  \right| 
  \\
  \le &  \frac{\sqrt  {   n} }{\sigma_\epsilon\sigma_\xi    } C_2 (1+\sigma_\xi ) \frac{\s \log(pn)}{ n}   
  =    C_2 \frac{ \s \log(pn)}{ \sigma_\epsilon\sigma_\xi    \sqrt n } + C _2\frac{\s \log(pn)}{ \sigma_\epsilon    \sqrt n }  
  \\
  =&o(1).
  \end{align*}
 The above two displays together with \Cref{condition:process} directly give  the desired result. 
 \end{proof}

\subsection{Power analysis under $H_1$}
\textbf{Proof of \Cref{corollary:iid power}}:
\begin{proof}  Suppose $H_a$ holds. In view of \Cref{eq:equivalence Hypothesis}, there exists    $t' \in [\zeta n , (1-\zeta)n]  $  such that  
\begin{align}
\label{eq:quadratic size}
(\beta^*_{(0, t'] }- \beta_ { (t', n ]}  ^*) ^\top  \Sigma  (\beta^*_{(0, t']} - \beta_{ (t', n ] } ^* ) : = \kappa^2 \ge B_n \frac{\s \log(pn) }{ n} .
\end{align}
  It suffices to show that for sufficiently large $n$,  with probability goes to $1$,
 \begin{align} \label{eq:power signal upper bound 1}    
 &  \frac{ \sqrt n }{2\sigma_\epsilon \sigma_\xi }   \bigg\{  ( \widehat \beta_{(0, t'] } -   \widehat \beta_{ (t', n ]}  )^\top     \widehat \Sigma _{(0, t'] } ( \widehat \beta_{(0, t'] } -   \widehat \beta_{ (t', n ]}   )   + 
\mathcal R_1(t')  +2 \mathcal R_2(t')  -2  \mathcal  R_3(t')   \bigg\}    >  \mathcal G_\alpha    ,
\\\label{eq:power signal upper bound 2} 
&    \frac{ \sqrt n }{2\sigma_\epsilon \sigma_\xi }    \bigg\{ ( \widehat \beta_{(  t', n ] } -   \widehat \beta_{ (0, t' ]}  )^\top     \widehat \Sigma _{( t', n ] } ( \widehat \beta_{(  t', n ] } -   \widehat \beta_{ (0, t' ]}   )   + 
\mathcal R_1 ' (t')  +2 \mathcal R_2' (t')  - 2 \mathcal   R_3' (t')    \bigg\}   >  \mathcal G_\alpha    .
 \end{align}   
The justification of  \Cref{eq:power signal upper bound 1} is provided as the  justification of  \Cref{eq:power signal upper bound 2}  is the same by symmetry. 
\\
\\
{\bf Step 1.} Under the alternative $H_a$,  note that by  \Cref{theorem:alternative},  for all $t \in [\zeta n, (1-\zeta)n]$,
\begin{align*}     
& \Bigg |   ( \widehat \beta_\ot -   \widehat \beta_\tn)^\top     \widehat \Sigma _\ot ( \widehat \beta_\ot -   \widehat \beta_\tn )   + 
\mathcal R_1(t )  +2 \mathcal R_2(t )  -2  \mathcal  R_3(t )    
 \\
 & -     ( \beta^*_\ot  - \beta^*_\tn  ) ^\top  \Sigma  ( \beta^*_\ot  - \beta^*_\tn  )      - (\mathfrak T_{n, \xi} (t) + \mathfrak  S_n (t)  )     
\Bigg|  \le
  C_1 (1+\sigma_\xi ) \frac{ \s \log(pn)}{ n}  
   . \end{align*}  
   \
   \\
  Then by \Cref{eq:quadratic size},   
\begin{align} \label{eq:power signal upper bound}   &     ( \widehat \beta_{(0, t'] } -   \widehat \beta_{ (t', n ]}  )^\top     \widehat \Sigma _{(0, t'] } ( \widehat \beta_{(0, t'] } -   \widehat \beta_{ (t', n ]}   )   + 
\mathcal R_1(t  ')  +2 \mathcal R_2(t  '  )  -2  \mathcal  R_3(t  ' )     
 \\ \nonumber
\ge    &   \kappa^2      -  | \mathfrak T_{n, \xi}   (t') | -| \mathfrak   S_n(t')   |    
-  C_1 (1+\sigma_\xi ) \frac{\s \log(pn)}{ n} .
 \end{align} 
 So it suffices to show that 
\begin{align} \label{eq:power signal upper bound2}   \kappa^2      -  | \mathfrak T_{n, \xi}  (t') | -| \mathfrak   S_n(t')   |    
-  C_1 (1+\sigma_\xi ) \frac{ \s \log(pn)}{ n}   \ge  \frac{ 2 \mathcal G_{\alpha   } \sigma_\epsilon\sigma_\xi }{ \sqrt  {   n} }  . 
\end{align} 
\
\\
{\bf Step 2.}
By definition of $\mathfrak T_{n, \xi} (t') $, 
\begin{align*} \mathfrak T_{n, \xi} (t')  = &\frac{2}{t '} \sum_{i=1}^{ t '}      x_i  ^\top     (  \beta^*_ {(0, t'] }    - \beta^*_ {(t', n]}   )      \epsilon_i    +  \frac{2}{t'} \sum_{i=1}^ { t '}      x_i  ^\top     (  \beta^*_i  - \beta^*_{(0, t']}    )      \xi_i        +  \frac{2}{t'} \sum_{i=1}^ { t '}    \xi_i\epsilon_i   
\\
  + &   ( \beta^*_{(0, t'] }   -     \beta_{(t', n]}   ^*  )  ^\top  (  \widehat \Sigma  _{(0, t'] }   -\Sigma  )  (  \beta^*_{(0, t'] }    -    \beta ^* _{(t', n]}     )   +
 \frac{2}{  t '} \sum_{i=1}^{t'}    (  \beta_i  ^* -\beta_{(0, t'] }   ^*    ) ^\top x_i   x_i^\top ( \beta_{(0, t'] }  ^*  -  \beta_{(t', n]}   ^*  )  . 
\end{align*} 
\
\\
By \Cref{assume: model assumption}, $\{ x_i\}_{i=1}^n   $ and   $\{ \epsilon_i\}_{i=1}^n   $ are sub-Gaussian random variables. Let $D _n $ be some slowly diverging sequence to be specified. By standard sub-exponential tail  bounds and the fact that
$ t' \in [\zeta n , (1-\zeta) n]$, it holds that with probability goes to $1$, 
\begin{align*}   &  \bigg| \frac{2}{t'} \sum_{i=1}^ {t'}     x_i  ^\top     (  \beta^*_{(0, t'] }    - \beta^*_ {(t', n]} )      \epsilon_i \bigg|  \le C_2 \sqrt {  \frac{D_n   }{n} }  \sqrt {(\beta^*_{(0, t'] } - \beta_{ (t', n ]}  ^*) ^\top  \Sigma  (\beta_{(0,t']}^*- \beta_{ (t', n ]}  ^* ) }   = C_2'  \sqrt {  \frac{D _n }{n} }  \kappa,
\\
 &(  \beta^*_{(0, t'] }  -     \beta_{ (t', n ]}   ^*  )  ^\top  (  \widehat \Sigma  _{(0, t'] }  -\Sigma  )  (  \beta^*_{(0, t'] }  -    \beta ^* _{ (t', n ]}    )   \le C_3  \sqrt  {  \frac{ D_n }   { t'}  }   \|  \beta^*_{(0, t'] }  -     \beta_{ (t', n ]}   ^*  \|_2^2 
 \le C_3' \sqrt  {  \frac{  D_n  }   {  n}  }    \kappa^2 ,   
\\
 &    \bigg|  \frac{2}{t'} \sum_{i=1}^{ t'}     \xi_i\epsilon_i  \bigg|  \le  C_4 \frac{\sigma_\xi  \sqrt { D_n  }  }{\sqrt n }\quad \text{and} \quad 
       \bigg| \frac{2}{t'} \sum_{i=1}^{t'}    x_i^\top(      \beta^*_i-  \beta^*_{(0, t'] } 
 )\xi_i  
\bigg|   \le  C_4' \frac{\sigma_\xi \sqrt {D_n } }{\sqrt n }.
 \end{align*} 
 In addition,  
 since 
 $$  \frac{1}{ t' } \sum_{i= 1  } ^ {t'}      
   ( \beta_ {(0, t'] }  ^{* } -\beta_{ (t', n ]}  ^*)  \Sigma   (\beta_ i ^* - \beta^*_{ (0, t' ]}   )   =0,
    $$ 
   by standard sub-exponential tail bounds, it holds that \begin{align*}
&\Bigg|   \frac{1}{ t ' } \sum_{i=1 } ^ {t'}      ( \beta_ {(0, t'] }  ^{* } -\beta_{ {(t', n]}  }  ^*)  x_ix_i^\top   (\beta_ i ^* - \beta^*_{ ( 0, t'  ]}   )  \Bigg| 
\\ 
= & \Bigg|   \frac{1}{ t' } \sum_{i= 1  } ^ {t'}     ( \beta_ {(0, t'] }  ^{* } -\beta_{ (t', n ]}  ^*)  x_ix_i^\top   (\beta_ i ^* - \beta^*_{ ( 0, t'  ]}  )  
-   \frac{1}{ t' } \sum_{i= 1  } ^ {t'}   ( \beta_ {(0, t'] }  ^{* } -\beta_{ (t', n ]}  ^*)  \Sigma   (\beta_ i ^* - \beta^*_{ ( 0, t'  ]}  )  \Bigg| 
\\
\le &  C_5  \sqrt {  \frac{ D_n  }{ t' } }  \kappa     \max_{1\le i \le n } \|\beta_ i ^* - \beta^*_{ ( 0, t'  ]}   \|_2 \le C_5'  \sqrt {  \frac{ D_n  }{n  } }  \kappa   ,
\end{align*} 
where $$  \sum_{i= 1  } ^ {t'}   ( \beta_ {(0, t'] }  ^{* } -\beta_{ (t', n ]}  ^*)  \Sigma   (\beta_ i ^* - \beta^*_{ ( 0, t'  ]}   ) =0$$
is used in the equality.  
So 
$$|\mathfrak T_n(t')  | \le C_6 \sqrt {  \frac{D_n  }{n} }  \kappa  +C_6 '  \sqrt  {  \frac{  D_n  }   {  n}  }    \kappa^2 +C_6'' \frac{\sigma_\xi  \sqrt {D_n }  }{\sqrt n } .$$
\
\\
By definition of $\mathfrak S_n  (t') $,
$$ \frac{1}{2 }\mathfrak  S_n(t')   =       \frac{1}{n-t '  } \sum_{i=t ' +1} ^n x_i^\top   ( \beta_ {(0, t'] }   ^{* } -\beta_{ (t', n ]}  ^*)  \epsilon_i 
  +  \frac{1}{n-t' } \sum_{i=t'+1} ^n     ( \beta_{ (t', n ]}  ^* - \beta_ {(0, t'] }  ^{* }  )  x_ix_i^\top   (\beta_ i ^* - \beta^*_{ (t', n ]}   )  .   $$
Note that    by standard sub-exponential tail bounds, it holds that with probability goes to $1$, 
  $$  \Bigg|  \frac{1}{n-t '  } \sum_{i=t ' +1} ^n x_i^\top   ( \beta_ {(0, t'] }   ^{* } -\beta_{ (t', n ]}  ^*)  \epsilon_i   \Bigg|  \le   \sqrt {  \frac{D_n }{n-t' } }  \kappa   $$
  and 
\begin{align*}
&\Bigg|   \frac{1}{n-t' } \sum_{i=t'+1} ^n     ( \beta_ {(0, t'] }  ^{* } -\beta_{ (t', n ]}  ^*)  x_ix_i^\top   (\beta_ i ^* - \beta^*_{ (t', n ]}   )  \Bigg| 
\\ 
= & \Bigg|   \frac{1}{n-t' } \sum_{i=t'+1} ^n     ( \beta_ {(0, t'] }  ^{* } -\beta_{ (t', n ]}  ^*)  x_ix_i^\top   (\beta_ i ^* - \beta^*_{ (t', n ]}   )  
-   ( \beta_ {(0, t'] }  ^{* } -\beta_{ (t', n ]}  ^*)  \Sigma   (\beta_ i ^* - \beta^*_{ (t', n ]}   )  \Bigg| 
\\
\le &  C_7  \sqrt {  \frac{ D_n  }{n-t' } }  \kappa     \max_{1\le i \le n } \|\beta_ i ^* - \beta^*_{ (t', n ]}  \|_2 \le C_7'  \sqrt {  \frac{D_n }{n  } }  \kappa   .
\end{align*}
So  with probability goes to $1$, 
$$|\mathfrak S_n (t') | \le C _8\sqrt {  \frac{ D_n }{n} }  \kappa    .$$
\
\\
{\bf Step 3.} Suppose $D_n $ is some slowly divergent sequence such that 
$D_n =o(\log(n)) $ and 
$ \sqrt { D_n} \le B_n/ A_n $, where we note that by assumption,  $B_n /A_n \to \infty $.
 So 
$$ \eqref{eq:power signal upper bound}\ge \kappa^2 -C_9 \left(   \sqrt {  \frac{D_n  }{n  } }  \kappa  + 
   \sqrt {  \frac{ D_n }{n  } }  \kappa ^2   +    \frac{ \sqrt { D_n } \sigma_\xi}{\sqrt n }  + (1+\sigma_\xi ) \frac{\s \log(pn)}{ n}    \right)    \ge \frac{1}{2}\kappa^2,$$
   where the last inequality holds because $  \sigma_\xi = A_n    \frac{ \s \log(pn) }{ \sqrt  {n}  }$,
   $\kappa^2  \ge B_n \frac{\s \log(pn) }{   n  }$. Since $\kappa^2  \ge B_n \frac{ \s \log(pn) }{   n  }$ and  $  \sigma_\xi =  A_n   \frac{ \s \log(pn) }{ \sqrt  {n}  }$,  for sufficiently large $B_n$,
   $$ \eqref{eq:power signal upper bound}\ge \frac{1}{2}\kappa^2 >   \frac{ 2 \mathcal G_{\alpha   } \sigma_\epsilon  }{ \sqrt  {   n} }   A_n  \frac{ \s \log(pn)}{ \sqrt n } =  \frac{ 2 \mathcal G_{\alpha   } \sigma_\epsilon  \sigma_\xi  }{ \sqrt  {   n} } . $$
   This above inequality directly implies that the power of test goes $ 1$ under $H_1(B_n)$.
\end{proof}

\subsection{Additional Technical Results}

\begin{proposition}  \label{theorem:alternative} 
Suppose $\{ \xi  _i \}_{i=1}^n \overset {i.i.d.} {\sim} N(0, \sigma^2_\xi)$.  Suppose \Cref{assume: model assumption}    and \Cref{condition:deviation} hold. 
Let \begin{align*} 
 \mathcal R_1(t)= & \frac{2}{t } \sum_{i=1}^t \left(  y_i -  x_i  ^\top   \widehat \beta_{\ot }     \right)  \left(   x_i  ^\top  (\widehat \beta _{ \ot }   - \widehat \beta_{(t,n]}) +\xi _i  \right),
 \\
 \mathcal R_2 (t)  =&  \frac{1}{ n-t }   \sum_{ i = t+1}^n  x_i^\top \widehat  \beta _\tn   \left(    y_i    -   x_i^\top\widehat \beta_\tn  \right) ,
 \\
 \mathcal R_3(t)  = & \frac{1}{n-t }    \sum_{ i=t+1}^{n }  x_i^\top  \widehat \beta _\ot    \left( y_i    -   x_i^\top \widehat \beta_\tn  \right).
 \end{align*} 
 Denote 
\begin{align*} 
  \mathfrak T_{n, \xi} (t)   = &\frac{2}{t} \sum_{i=1}^ t     x_i  ^\top     (  \beta^*_\ot  - \beta^*_\tn  )      \epsilon_i    +  \frac{2}{t} \sum_{i=1}^ t     x_i  ^\top     (  \beta^*_i  - \beta^*_\ot   )      \xi_i        +  \frac{2}{t} \sum_{i=1}^ t    \xi_i\epsilon_i   
\\
  + &   ( \beta^*_\ot  -     \beta_\tn ^*  )  ^\top  (  \widehat \Sigma  _\ot  -\Sigma  )  (  \beta^*_\ot  -    \beta ^* _\tn   )   +
 \frac{2}{  t } \sum_{i=1}^t x_i^\top (  \beta_i  ^* -\beta_\ot ^*    )   x_i^\top ( \beta_\ot ^*  -  \beta_\tn ^*  )  , 
 \\
  \mathfrak S_n    (t)  = &     \frac{2}{n-t } \sum_{i=t+1} ^n x_i^\top   ( \beta_\tn^* - \beta_ \ot^{* }  )  \epsilon_i 
  +  \frac{2}{n-t } \sum_{i=t+1} ^n     ( \beta_ \ot^{* } -\beta_\tn^*)  x_ix_i^\top   (\beta_ i ^* - \beta^*_\tn )       .
\end{align*} 
For all $t \in [\zeta n,  (1-\zeta) n ] $,  it holds that 
\begin{align*}    
& \Bigg |   ( \widehat \beta_\ot -   \widehat \beta_\tn)^\top     \widehat \Sigma _\ot ( \widehat \beta_\ot -   \widehat \beta_\tn )   + 
\mathcal R_1(t)  +2\mathcal R_2(t)  -2 \mathcal R_3(t)  
 \\
 -   & ( \beta^*_\ot  - \beta^*_\tn  ) ^\top  \Sigma  ( \beta^*_\ot  - \beta^*_\tn  )      -  \big\{   \mathfrak T_{n, \xi} (t) + \mathfrak  S_n(t)   \big\}    
\Bigg| 
 \\   \le
 & C (1+\sigma_\xi ) \frac{\s \log(pn)}{ n}  
   . \end{align*}
\end{proposition}  
  
\begin{proof} In view of  \Cref{lemma:test deviation bound}, it suffices to bound 
$| \mathcal R_1(t) - \mathcal R_4(t) | $.
Observe  that 
\begin{align*}  \mathcal R_1(t) - \mathcal R_4(t) =  \frac{2}{t} \sum_{i=1}^t \left(     y_i -  x_i  ^\top\widehat \beta  _\ot \right)    \xi_i      = &
 \frac{2}{t } \sum_{i=1}^t   \epsilon_i   \xi_i  + \frac{2}{t } \sum_{i=1}^t   x_i^\top(    \beta^*_i   - \beta^*_\ot + \beta^*_\ot -\widehat \beta_\ot  
 )\xi_i
   \end{align*} 
Note that by \Cref{condition:deviation}{\bf d} and \Cref{lemma:interval lasso},
\begin{align*}
\bigg| \frac{2}{t} \sum_{i=1}^t   x_i^\top(     \beta^*_\ot   -\widehat \beta_\ot  
 )\xi_i   \bigg| 
 \le C_1 \sigma_\xi \sqrt { \frac{\log(p) }{t } } \|  \beta^*_\ot   -\widehat \beta_\ot  \|_1 \le C_2\sigma_\xi   \frac{\s \log(pn) }{n} .
\end{align*}
\end{proof} 

 \begin{proposition}  \label{theorem:alternative counterpart} 
 Suppose $\{ \xi  _i \}_{i=1}^n  \overset {i.i.d.} {\sim} N(0, \sigma^2_\xi)$.  Suppose \Cref{assume: model assumption}    and \Cref{condition:deviation} hold. 
Let \begin{align*} 
 \mathcal R_1 ' (t)  = & \frac{2}{n-t } \sum_{i=t+1}^n \left(  y_i -  x_i  ^\top   \widehat \beta_{\tn }     \right)  \left(   x_i  ^\top  (\widehat \beta _{ \tn }   - \widehat \beta_{\ot })- \xi _i  \right),
 \\
 \mathcal R_2' (t)  =&  \frac{1}{  t }   \sum_{ i =1}^t   x_i^\top \widehat  \beta _\ot    \left(    y_i    -   x_i^\top\widehat \beta_\ot   \right) ,
 \\
 \mathcal R_3' (t)  = & \frac{1}{t }    \sum_{ i= 1}^{t  }  x_i^\top  \widehat \beta _\tn    \left( y_i    -   x_i^\top \widehat \beta_\ot   \right).
 \end{align*} 
 Denote 
\begin{align*} 
  \mathfrak T'_{n,\xi} (t)   = &\frac{2}{n-t} \sum_{i=t+1}^ n     x_i  ^\top     (  \beta^*_\tn  - \beta^*_\ot    )      \epsilon_i    +  \frac{2}{n-t} \sum_{i=t+1}^ n     x_i  ^\top     (  \beta^*_i  - \beta^*_\tn   )      \xi_i        -  \frac{2}{n-t} \sum_{i=t+1}^ n    \xi_i\epsilon_i   
\\
  + &   ( \beta^*_\tn   -     \beta_\ot  ^*  )  ^\top  (  \widehat \Sigma  _\tn   -\Sigma  )  (  \beta^*_\tn  -    \beta ^* _\ot   )   +
 \frac{2}{ n- t } \sum_{i=t+1}^n x_i^\top (  \beta_i  ^* -\beta_\tn ^*    )   x_i^\top ( \beta_\tn ^*  -  \beta_\ot ^*  )  , 
 \\
   \mathfrak S'_n  (t)  = &     \frac{2 }{ t  } \sum_{i=1 } ^ t  x_i^\top   ( \beta_\ot ^* - \beta_ \tn ^{* }  )  \epsilon_i 
  +  \frac{2 } {  t } \sum_{i= 1} ^t      ( \beta_ \tn ^{* } -\beta_\ot ^*)  x_ix_i^\top   (\beta_ i ^* - \beta^*_\ot  )       .
\end{align*} 
It holds that 
\begin{align*}    
& \Bigg |   ( \widehat \beta_\tn  -   \widehat \beta_\ot )^\top     \widehat \Sigma _\tn  ( \widehat \beta_\tn  -   \widehat \beta_\ot )   + 
\mathcal R_1 ' (t)  +2\mathcal R_2 ' (t)  -2 \mathcal R_3 ' (t)  
 \\
 -   & ( \beta^*_\tn  - \beta^*_\ot   ) ^\top  \Sigma  ( \beta^*_\tn   - \beta^*_\ot  )      -  \big\{   \mathfrak T'_{n,\xi} (t) + \mathfrak  S_n'(t)   \big\}    
\Bigg| 
 \\   \le
 & C (1+\sigma_\xi ) \frac{\s \log(pn)}{ n}  
   . \end{align*}
\end{proposition}   
 \begin{proof}
 Observe that
\Cref{theorem:alternative counterpart}  is the symmetric counterpart of \Cref{theorem:alternative}.
The proof of  \Cref{theorem:alternative counterpart} is identical to \Cref{theorem:alternative}  
 and is therefore omitted. 
 \end{proof}

\subsection*{Proofs Related to \Cref{theorem:alternative} and \Cref{theorem:alternative counterpart}}
Throughout this subsection, we assume that $\zeta\in (0, 1/2)$and  that \Cref{condition:deviation} and  \Cref{condition:process}  hold.  The justification of these conditions will be postponed to \Cref{section:distribution}
  
\begin{lemma} \label{lemma:bias term R2}
For $t\in \{ 1, \ldots, n\}$, denote 
$$ \mathcal R_2 (t)  =  \frac{1}{ n-t }   \sum_{ i = t+1}^n  x_i^\top \widehat  \beta _\tn   \left(    y_i    -   x_i^\top\widehat \beta_\tn  \right)   .$$
If $ n-t \ge \zeta n$, then  it holds that 
\begin{align*}
 \bigg| &\left(   \widehat  \beta_\tn ^{\top}  \Sigma  \widehat  \beta _\tn  + 2 \mathcal  R_2(t)      \right)   
- \beta_\tn^{*\top}  \Sigma  \beta^*_\tn  +  \Bigg\{ 
-2 \beta_{\tn}^{*\top} 
\left (\Sigma  -  \widehat \Sigma_\tn   \right ) ( \widehat \beta_\tn  -\beta^*_\tn )   
 \\ - &      \frac{ 2 }{n-t } \sum_{i=t+1}^n     \epsilon_i x_i^\top  \beta _\tn ^*        - 
\frac{   2 }{n-t } \sum_{i=t+1}^n          \beta_\tn     ^ {*\top}      x_i     x_i  ^\top 
 ( \beta_i ^*  - \beta^*_{(t,n]}  )     \Bigg \}    \bigg | \le  O \left( \frac{ \s \log(pn)}{n} \right )
\end{align*}
\end{lemma}

\begin{proof}
Note that 
\begin{align} \nonumber 
&  \widehat  \beta_\tn ^{\top}  \Sigma   \widehat  \beta _\tn  + 2 \mathcal  R_2   (t) - \beta_\tn^{*\top} \Sigma  \beta_\tn^{* }
 \\ 
 \label{eq:bias 1 term 1}
 =&    \frac{   2 }{n-t } \sum_{i=t+1}^n     x_i^\top \widehat  \beta _\tn   \left ( \epsilon_i + x_i^\top 
 ( \beta_i ^*  -  \beta_{(t,n]} ^* ) \right ) 
 \\ \label{eq:bias 1 term 2}
 + & 2\widehat \beta _\tn ^\top  \Sigma   ( \widehat \beta _\tn   -\beta^*_\tn   )  -2\widehat \beta _\tn ^\top  \widehat  \Sigma_\tn   ( \widehat \beta _\tn   -\beta^*_\tn   ) 
 \\ \label{eq:bias 1 term 3}
   -&  ( \widehat \beta_\tn -\beta^*_\tn  ) ^\top \Sigma  (\widehat \beta_\tn -\beta^*_\tn).
\end{align}
\
\\
{\bf Step 1.} For \Cref{eq:bias 1 term 1}, note that 
\begin{align*} & 
\bigg| \frac{ 2 }{n-t } \sum_{i=t+1}^n     x_i^\top \widehat  \beta _\tn     \epsilon_i   -   \frac{ 2 }{n-t } \sum_{i=t+1}^n     x_i^\top  \beta _\tn ^*     \epsilon_i  \bigg|
\\
 = &\bigg|   \frac{ 2 }{n-t } \sum_{i=t+1}^n     x_i^\top  ( \widehat  \beta _\tn -\beta^*_\tn )      \epsilon_i   \bigg| 
 \\
 \le & C_1 \sqrt { \frac{\log(pn)   }{  n-t } } \| \widehat  \beta _\tn -\beta^*_\tn  \| _1  
 \le C_1'    \frac{ \s \log(pn)   }{  n-t },
\end{align*} 
where the first inequality follows  from  \Cref{condition:deviation}{\bf b} and the  second  inequality  follows from \Cref{lemma:interval lasso}.  
In addition,
\begin{align*}
  &\bigg| \frac{   2 }{n-t } \sum_{i=t+1}^n      (  \widehat  \beta _\tn   - \beta^* _\tn ) ^\top  x_i     x_i  ^\top 
 ( \beta_i ^*  - \beta^*_{(t,n]}  ) -         \frac{   2 }{n-t } \sum_{i=t+1}^n    (  \widehat  \beta _\tn   - \beta^* _\tn ) ^\top  \Sigma   
 ( \beta_i ^*  - \beta^*_{(t,n]}  ) \bigg|
 \\
 \le & \max_{ t+1 \le i \le n } \| \beta_i ^*  - \beta^*_{(t,n]} \| _2   \sqrt { \frac{\log(pn) }{n-t } }   \|     \widehat  \beta _\tn   - \beta^* _\tn    \|_1  
 \\
 \le &  C_2    \frac{\s \log(pn)   }{  n-t }, 
\end{align*}
where the first inequality follows from \Cref{condition:deviation}{\bf c} and the second inequality follows from \Cref{lemma:interval lasso} and  \Cref{lemma:beta bounded 1}.  
Note that 
$$ \sum_{i=t+1}^n    (  \widehat  \beta _\tn   - \beta^* _\tn ) ^\top  \Sigma   
 ( \beta_i ^*  - \beta^*_{(t,n]}  )   =    (  \widehat  \beta _\tn   - \beta^* _\tn ) ^\top  \Sigma   
 \sum_{i=t+1}^n   ( \beta_i ^*  - \beta^*_{(t,n]}  )   =0 . $$
Consequently, 
\begin{align*} 
 \left | \eqref{eq:bias 1 term 1} -   \frac{ 2 }{n-t } \sum_{i=t+1}^n     \epsilon_i x_i^\top  \beta _\tn ^*        - 
\frac{   2 }{n-t } \sum_{i=t+1}^n          \beta_\tn     ^ {*\top}      x_i     x_i  ^\top 
 ( \beta_i ^*  - \beta^*_{(t,n]}  )  \right| \le C_3      \frac{\s \log(pn)   }{  n-t }. 
\end{align*} 
\
\\
{\bf Step 2.}
By \Cref{lemma:interval lasso}, 
\begin{align}
\label{eq:restricted vectors} \|  ( \widehat \beta_\tn  -\beta^*_\tn)_{S^c} \|_1\le 3 \|  ( \widehat \beta_\tn  -\beta^*_\tn )_S\|_1.
\end{align}
So 
\begin{align*}
&\frac{1}{2}\cdot \eqref{eq:bias 1 term 2} - \beta_{\tn}^{*\top} 
\left (\Sigma  - \widehat  \Sigma_\tn    \right ) ( \widehat \beta_\tn  -\beta^*_\tn )  
\\ = 
& ( \widehat \beta_\tn  -\beta^*_\tn )^\top 
 \left (\Sigma  - \widehat  \Sigma_\tn    \right ) ( \widehat \beta_\tn  -\beta^*_\tn ) 
\\
\le &C_4 \sqrt { \frac{\s \log(pn)   }{  n-t }}  \| \widehat \beta_\tn  -\beta^*_\tn \| _2 ^2  \le  C _4'     \frac{\s \log(pn)   }{  n-t }. 
\end{align*} 
where the first inequality follows from 
\Cref{condition:deviation}{\bf a}  and \Cref{eq:restricted vectors}, and the second inequality follows from \Cref{lemma:interval lasso} and the fact that 
$$  \frac{\s \log(pn)   }{  n-t } \le \frac{ \s \log(pn) }{ \zeta n } \le  \frac{1 }{ C_{snr} \zeta  }  <\infty .$$
\\
\\
{\bf Step 3.}
\Cref{lemma:interval lasso} implies that 
$$ ( \widehat \beta_\tn -\beta^*_\tn  ) ^\top \Sigma  (\widehat \beta_\tn -\beta^*_\tn) \le 
C_x^2 \| \widehat \beta_\tn -\beta^*_\tn \|_2^ 2 \le   C _5    \frac{\s \log(pn)   }{  n-t } .$$
\
\\
The desired result follows from putting all the previous steps together.

\end{proof}

\begin{lemma}\label{lemma:bias term R3}
For $t\in \{ 1, \ldots, n\}$, denote 
$$\mathcal R_3(t)  =  \frac{1}{n-t }    \sum_{ i=t+1}^{n }  x_i^\top  \widehat \beta _\ot    \left( y_i    -   x_i^\top \widehat \beta_\tn  \right)$$
If $ \min\{t, n-t \} \ge \zeta n$, then with probability goes to $1$, it holds that 
\begin{align*}
 \Bigg| &   \beta^{*\top} _\ot  \Sigma    ( \beta_\tn^*-\widehat \beta_\tn )   -\mathcal R_3(t)            +  
  \Bigg  \{  \frac{1}{n-t } \sum_{i=t+1} ^n x_i^\top    \beta_ \ot^{* } \epsilon_i    
  \\
  & +  \frac{1}{n-t } \sum_{i=t+1} ^n   \beta_\ot^{*  \top}  x_ix_i^\top   (\beta_ i ^* - \beta^*_\tn )   - \beta^{*\top} _\ot    (  \Sigma- \widehat \Sigma_\tn  )  (\beta_\tn^* -\widehat \beta_\tn )  \Bigg  \}    \Bigg| \le C  \frac {\s \log(pn)} {  n   }  .
\end{align*}
\end{lemma}

\begin{proof}
Note that 
\begin{align} \nonumber 
&   \beta^{*\top} _\ot  \Sigma    ( \beta_\tn^*-\widehat \beta_\tn )   -\mathcal R_3(t)    
\\ 
\label{eq:bias 2 term 1}
=& \beta^{*\top} _\ot  \Sigma    ( \beta_\tn^*-\widehat \beta_\tn )   -  \widehat \beta^{\top} _\ot  \widehat  \Sigma_\tn     ( \beta_\tn^*-\widehat \beta_\tn )
\\\label{eq:bias 2 term 2}
- & \frac{1}{n-t } \sum_{i=t+1} ^n x_i^\top  \widehat  \beta_ \ot \epsilon_i   - \frac{1}{n-t } \sum_{i=t+1} ^n  \widehat  \beta_ \ot ^\top x_ix_i^\top (\beta_ i^* - \beta^*_\tn )
 \end{align}
\
\\
{ \bf Step 1.}  Observe that 
\begin{align*}
\eqref{eq:bias 2 term 1}  =&  (\beta_\ot^* -\widehat \beta_\ot ) ^\top  \Sigma  (\beta_\tn^* -\widehat \beta_\tn ) 
+ \widehat \beta_\ot ^\top  ( \Sigma- \widehat \Sigma_\tn  )  (\beta_\tn^* -\widehat \beta_\tn )
\\
=
& (\beta_\ot^* -\widehat \beta_\ot ) ^\top  \Sigma   (\beta_\tn^* -\widehat \beta_\tn ) 
\\
+ & (  \widehat \beta_\ot  - \beta^*_\ot ) ^\top ( \Sigma- \widehat \Sigma_\tn  )  (\beta_\tn^* -\widehat \beta_\tn  )  
\\
+ &   \beta^{*\top} _\ot    (  \Sigma- \widehat \Sigma_\tn  )  (\beta_\tn^* -\widehat \beta_\tn ) .
\end{align*}
It holds that 
$$  (\beta_\ot^* -\widehat \beta_\ot )  ^\top  \Sigma  (\beta_\tn^* -\widehat \beta_\tn )   \le C_x \| \beta_\ot^* -\widehat \beta_\ot\|_2 
\| \beta_\tn^* -\widehat \beta_\tn \| _2  \le  C_1 \frac{\s\log (pn) }{ \sqrt { t(n-t) } } \le C_1' \frac{\s\log (pn) }{ n  }  ,$$
where the  second inequality follows from \Cref{lemma:interval lasso}. 
In addition, 
\begin{align} \nonumber&(  \widehat \beta_\ot  - \beta^*_\ot ) ^\top ( \Sigma- \widehat \Sigma_\tn  )  (\beta_\tn^* -\widehat \beta_\tn ) 
\\
\le& C_2 \sqrt { \frac{\log(pn)}{n } } \| \widehat \beta_\ot  - \beta^*_\ot  \|_2    \| \beta_\tn^* -\widehat \beta_\tn\| _1 \label{eq:using independence}
\\
\le & C _2 '\frac{ \s \log(pn) }{ \sqrt  { n(n-t) } } \le C_2'' \frac{ \s \log(pn)}{n } ,\nonumber
  \end{align}
  where the first inequality follows from  \Cref{condition:deviation}{\bf c} and the fact that the two collections of data $\{x_i,  \epsilon_i\}_{i=1}^{t} $ and $ \{x_i, \epsilon_i \}_{i=t+1}^n$  are   independent, and   the  second inequality follows from \Cref{lemma:interval lasso}.  So 
  $$ \left| \eqref{eq:bias 2 term 1}  - \beta^{*\top} _\ot    (  \Sigma- \widehat \Sigma_\tn  )  (\beta_\tn^* -\widehat \beta_\tn ) \right|  \le C_3 \frac{ \s \log(pn) }{ n } .$$
  \
  \\
  { \bf Step 2.} For \Cref{eq:bias 2 term 2},
  note that
  \begin{align*} 
  &\bigg| \frac{1}{n-t } \sum_{i=t+1} ^n x_i^\top  \widehat  \beta_ \ot \epsilon_i   - \frac{1}{n-t } \sum_{i=t+1} ^n x_i^\top    \beta_ \ot^{* } \epsilon_i  \bigg|
  \\=
  &\bigg| \frac{1}{n-t } \sum_{i=t+1} ^n x_i^\top  ( \widehat  \beta_ \ot -\beta_\ot^*)  \epsilon_i  \bigg| 
  \\
 \le & \sqrt { \frac{\log(pn)}{(n-t) }}  \| \widehat  \beta_ \ot -\beta_\ot^* \| _1 
 \le C\frac{\s \log(pn)}{\sqrt {t(n-t)}  }.
 \end{align*} 
 In addition,
 \begin{align*}
 &\bigg| \frac{1}{n-t } \sum_{i=t+1} ^n  (  \widehat  \beta_ \ot   -\beta_\ot^* )^\top x_ix_i^\top (\beta_ i^* - \beta^*_\tn ) 
 - \frac{1}{n-t } \sum_{i=t+1} ^n  (  \widehat  \beta_ \ot   -\beta_\ot^* )^\top  \Sigma   (\beta_i^* - \beta^*_\tn ) \bigg| 
 \\
 \le &  \sqrt { \max_{1\le i  \le n }\|  \beta_ i^* - \beta^*_\tn   \| _2  } C\sqrt { \frac {\log(pn)} {(n-t) }}  \|  \widehat  \beta_ \ot   -\beta_\ot^*\|_1
 \le C\s \frac {\log(pn)} { \sqrt { t(n-t) }  },
 \end{align*} 
 where the first inequality follows from \Cref{condition:deviation}{\bf c} and the second inequality follows from \Cref{lemma:interval lasso}.
Note that 
$$ \frac{1}{n-t } \sum_{i=t+1} ^n  (  \widehat  \beta_ \ot   -\beta_\ot^* )^\top  \Sigma   (\beta_i^* - \beta^*_\tn )  
=    (  \widehat  \beta_ \ot   -\beta_\ot^* )^\top  \Sigma  \frac{1}{n-t } \sum_{i=t+1} ^n  (\beta_ i^* - \beta^*_\tn ) =0.  $$
So
$$  \left| \eqref{eq:bias 2 term 2}  +  \frac{1}{n-t } \sum_{i=t+1} ^n x_i^\top    \beta_ \ot^{* } \epsilon_i     +  \frac{1}{n-t } \sum_{i=t+1} ^n   \beta_\ot^{*  \top}  x_ix_i^\top   (\beta_ i ^* - \beta^*_\tn )   \right|  \le C  \frac {\s \log(pn)} {  n   }.$$
\end{proof}

\begin{lemma} \label{lemma:bias 3}
For $t\in \{ 1, \ldots, n\}$,  let 
 \begin{align*} 
\mathcal R_4(t) = & \frac{2}{t } \sum_{i=1}^t \left(  y_i -  x_i  ^\top   \widehat \beta_{\ot }    \right)   x_i  ^\top  (\widehat \beta _{ \ot }   - \widehat \beta_{(t,n]})  . 
 \end{align*} 
 If $ \min\{t, n-t \} \ge \zeta n$,  then 
 \begin{align*}
 \Bigg | & ( \widehat \beta_\ot -   \widehat \beta_\tn)^\top     \widehat \Sigma _\ot ( \widehat \beta_\ot -   \widehat \beta_\tn )   + \mathcal R_4 (t) -
(\beta^*_\ot  -\widehat   \beta _\tn  ) ^\top 
 \Sigma  ( \beta^*_\ot  -  \widehat \beta _\tn    )    
 \\
& +\Bigg \{  -\frac{2}{  t } \sum_{i=1}^t x_i^\top ( \beta_\ot ^*  -   \beta_\tn ^*  )\epsilon_i   -( \beta_\ot ^*  -  \beta_\tn  ^*   )  ^\top  ( \widehat \Sigma_\ot  -\Sigma)( \beta_\ot ^*  -  \beta_\tn  ^*   )  
\\
& - \frac{2}{  t } \sum_{i=1}^t x_i^\top (  \beta_i  ^* -\beta_\ot ^*    )   x_i^\top ( \beta_\ot ^*  -  \beta_\tn ^*  )  \Bigg \}  \Bigg|  
 \le C\s \frac{\log(pn)}{n}. 
 \end{align*}
\end{lemma}

\begin{proof} Note that 
  \begin{align}  \nonumber 
 & ( \widehat \beta_\ot -   \widehat \beta_\tn)^\top     \widehat \Sigma _\ot ( \widehat \beta_\ot -   \widehat \beta_\tn )    -
(\beta^*_\ot  -\widehat   \beta _\tn  ) ^\top 
 \Sigma  ( \beta^*_\ot  -  \widehat \beta _\tn    )    
 \\ \label{eq:bias 3 term 1}
 = 
 &   2  ( \widehat \beta_\ot -   \widehat \beta_\tn)^\top     \widehat \Sigma _\ot  (\widehat \beta_\ot  -\beta^*_\ot  ) 
 \\ \label{eq:bias 3 term 2}
+ 
 & (\beta^*_\ot  -\widehat   \beta _\tn  ) ^\top 
 (\widehat \Sigma_\ot - \Sigma )  ( \beta^*_\ot  -  \widehat \beta _\tn    )     
 \\ \label{eq:bias 3 term 3}
 - &   \left   (\widehat \beta_\ot      -    \beta_\ot^*       \right  )^\top \widehat \Sigma_\ot 
 \left   (\widehat \beta_\ot      -    \beta_\ot^*            \right  )   
  \end{align} 
  \
  \\
{  \bf Step 1. } Note that 
 \begin{align*}
 \eqref{eq:bias 3 term 1}  + \mathcal R_4 (t) 
  = 
 & \frac{2}{  t } \sum_{i=1}^t x_i^\top (\widehat \beta_\ot - \widehat \beta_\tn )\epsilon_i  + 
 \frac{2}{  t } \sum_{i=1}^t x_i^\top (  \beta_i  ^* -\beta_\ot ^*    )   x_i^\top (\widehat \beta_\ot - \widehat \beta_\tn ).
 \end{align*} 
 It holds that 
\begin{align*} &\bigg|   \frac{2}{  t } \sum_{i=1}^t x_i^\top (\widehat \beta_\ot - \widehat \beta_\tn )\epsilon_i   -  \frac{2}{  t } \sum_{i=1}^t x_i^\top ( \beta_\ot ^*  -   \beta_\tn ^*  )\epsilon_i  \bigg| 
\\
 \le 
 & C_1\sqrt {\frac{\log(pn)}{t}} \left(  \| \widehat \beta_\ot -\beta_\ot ^*\|_ 1+   \|    \widehat \beta_\tn-\beta_\tn^* \|_1   \right) 
 \\
 \le &C_1' \frac{\s \log(pn)}{n },
 \end{align*}
 where the first inequality follows from \Cref{condition:deviation}{\bf b}, and the second inequality follows from \Cref{lemma:interval lasso}.
 In addition,
 \begin{align*} 
 & \left| \frac{2}{  t } \sum_{i=1}^t x_i^\top (  \beta_i  ^* -\beta_\ot ^*    )   x_i^\top (\widehat \beta_\ot - \widehat \beta_\tn )
  -\frac{2}{  t } \sum_{i=1}^t x_i^\top (  \beta_i  ^* -\beta_\ot ^*    )   x_i^\top ( \beta_\ot ^*  -  \beta_\tn ^*  )\right| 
  \\
\le  &\left| \frac{2}{  t } \sum_{i=1}^t x_i^\top (  \beta_i  ^* -\beta_\ot ^*    )   x_i^\top (\widehat \beta_\ot -  \beta_\ot ^*  )
  \right| + \left| \frac{2}{  t } \sum_{i=1}^t x_i^\top (  \beta_i  ^* -\beta_\ot ^*    )   x_i^\top ( \widehat \beta_\tn  -  \beta_\tn ^*  )\right|  
  \\
=
&\left|\frac{2}{  t } \sum_{i=1}^t x_i^\top (  \beta_i  ^* -\beta_\ot ^*    )   x_i^\top (\widehat \beta_\ot -  \beta_\ot ^*  )
  -   \frac{2}{  t } \sum_{i=1}^t  (  \beta_i  ^* -\beta_\ot ^*    ) ^\top   \Sigma (\widehat \beta_\ot -  \beta_\ot ^*  )  \right|   
+
  \\
  +    &\left| \frac{2}{  t } \sum_{i=1}^t x_i^\top (  \beta_i  ^* -\beta_\ot ^*    )   x_i^\top   ( \widehat \beta_\tn  -  \beta_\tn ^*  ) 
  -  \frac{2}{  t } \sum_{i=1}^t  (  \beta_i  ^* -\beta_\ot ^*    ) ^\top   \Sigma   ( \widehat \beta_\tn  -  \beta_\tn ^*  )   \right|   
  \\
  \le  &  C_2\max_{1\le i \le t } \| \beta_i  ^* -\beta_\ot ^*  \|  _2 \sqrt{ \frac{\log(pn)}{ t  } }  \left( \| \widehat \beta_\ot -  \beta_\ot ^* \| _1  
  + \| \widehat \beta_\tn -  \beta_\tn ^* \| _1\right) 
  \\
  \le 
  &  C_2'  \s\frac{\log(pn)}{ n }  ,
   \end{align*}
   where the  equality follows from the observation that 
   $$  \sum_{i=1}^t  (  \beta_i  ^* -\beta_\ot ^*    ) ^\top   \Sigma (\widehat \beta_\ot -  \beta_\ot ^*  )  =0 , \quad  \sum_{i=1}^t  (  \beta_i  ^* -\beta_\ot ^*    ) ^\top   \Sigma   ( \widehat \beta_\tn  -  \beta_\tn ^*  ) =0  ,$$
   the second to last inequality follows from  \Cref{condition:deviation}{\bf c}, and the last inequality follows from  \Cref{lemma:interval lasso}.
   \\
   \\
  { \bf Step 2.} It holds that 
\begin{align*}
&\eqref{eq:bias 3 term 2}   - ( \beta_\ot ^*  -  \beta_\tn  ^*   )  ^\top  ( \widehat \Sigma_\ot  -\Sigma)( \beta_\ot ^*  -  \beta_\tn  ^*   )   
\\ 
=
&
2 ( \beta_\tn  ^* - \beta_\ot ^*  )  ^\top  ( \widehat \Sigma_\ot  -\Sigma)( \widehat \beta_\tn -  \beta_\tn  ^*    ) + 
( \widehat \beta_\tn -  \beta_\tn  ^*    )  ^\top  ( \widehat \Sigma_\ot  -\Sigma)( \widehat \beta_\tn -  \beta_\tn  ^*    ),
\end{align*}

Note that 
\begin{align*}
& ( \beta_\tn  ^* - \beta_\ot ^*  )  ^\top  ( \widehat \Sigma_\ot  -\Sigma)( \widehat \beta_\tn -  \beta_\tn  ^*    ) 
\\
\le& C_3\|\beta_\tn  ^* - \beta_\ot ^*  \| _2 \sqrt { \frac{\log(pn) }{t }}  \| \widehat \beta_\tn -  \beta_\tn  ^*  \| _ 1 
\\
\le  & C_3' \s \frac{\log(pn) }{n }
\end{align*} 
where the first inequality follows from \Cref{condition:deviation}{\bf c} and the second inequality follows from \Cref{lemma:interval lasso}.
In addition, 
$$( \widehat \beta_\tn -  \beta_\tn  ^*    )  ^\top  ( \widehat \Sigma_\ot  -\Sigma)( \widehat \beta_\tn -  \beta_\tn  ^*    )  
\le C_4 \| \widehat \beta_\tn -  \beta_\tn  ^*  \|_2^2 \le C_4'\s \frac{\log(pn) }{n } .  $$
where the first inequality follows from \Cref{condition:deviation}{\bf a} and the second inequality follows from \Cref{lemma:interval lasso}.
So 
\begin{align*}
 \left|  \eqref{eq:bias 3 term 2}   - ( \beta_\ot ^*  -  \beta_\tn  ^*   )  ^\top  ( \widehat \Sigma_\ot  -\Sigma)( \beta_\ot ^*  -  \beta_\tn  ^*   )   \right| 
 \le C\s \frac{\log(pn) }{n } .
\end{align*} 
{\bf Step 3.} Similarly  
\begin{align*} \eqref{eq:bias 3 term 3}
\le & 
  \bigg |  (\widehat \beta_\ot      -    \beta_\ot^*        )^\top \{ \widehat \Sigma_\ot  -\Sigma\big\} 
    (\widehat \beta_\ot      -    \beta_\ot^*              )  \bigg| + \bigg |  (\widehat \beta_\ot      -    \beta_\ot^*         )^\top  \Sigma    (\widehat \beta_\ot      -    \beta_\ot^*              )   \bigg| 
 \\   
 \le 
 &
  C_5 \sqrt { \frac{\s \log(pn) }{ | \mathcal I | }}\| \widehat \beta_\ot      -    \beta_\ot^*     \|_2^2 +C_x  \| \widehat \beta_\ot      -    \beta_\ot^*     \|_2^2  
  \\
  \le&  C_5'       \frac{\s \log(pn) }{n },
\end{align*}
where the first inequality follows from \Cref{condition:deviation}{\bf a} and  the second inequality follows from \Cref{lemma:interval lasso}.
\end{proof}

\begin{theorem}   \label{lemma:test deviation bound} 
Suppose $\min\{ t, n-t\}\ge \zeta n$.
Let \begin{align*}  
 \mathcal R_4(t) =  \frac{2}{t } \sum_{i=1}^t \left(  y_i -  x_i  ^\top   \widehat \beta_{\ot }    \right)   x_i  ^\top  (\widehat \beta _{ \ot }   - \widehat \beta_{(t,n]})  . 
 \end{align*} 
  Denote 
\begin{align*}
\mathfrak  T_n (t)  = &\frac{2}{t} \sum_{i=1}^ t     x_i  ^\top     (  \beta^*_\ot  - \beta^*_\tn  )     \epsilon_i      +   ( \beta^*_\ot  -     \beta_\tn ^*  )  ^\top  (  \widehat \Sigma  _\ot  -\Sigma  )  (  \beta^*_\ot  -    \beta ^* _\tn   )      
\\
+ & \frac{2}{  t } \sum_{i=1}^t x_i^\top (  \beta_i  ^* -\beta_\ot ^*    )   x_i^\top ( \beta_\ot ^*  -  \beta_\tn ^*  )  ,  
\\
\mathfrak  S _n (t)   = &     \frac{2 }{n-t } \sum_{i=t+1} ^n x_i^\top   ( \beta_\tn^*-\beta_ \ot^{* }  )  \epsilon_i 
  +  \frac{2 }{n-t } \sum_{i=t+1} ^n     (\beta_\tn^* - \beta_ \ot^{* }  )  x_ix_i^\top   (\beta_ i ^* - \beta^*_\tn )       
\end{align*} 
It holds that 
\begin{align*}    
 \bigg| &  ( \widehat \beta_\ot -   \widehat \beta_\tn)^\top     \widehat \Sigma _\ot ( \widehat \beta_\ot -   \widehat \beta_\tn )    -  
 ( \beta^*_\ot  - \beta^*_\tn  ) ^\top 
 \Sigma  ( \beta^*_\ot  - \beta^*_\tn  )     
\\
&
+ 
\mathcal R_4(t)  +2 \mathcal R_2(t)  -2 \mathcal R_3(t)  
     -
 \big\{ \mathfrak  T_n (t) +\mathfrak S_n  (t)    \big\}
 \bigg|  
   \le
  C\frac{\s \log(pn)}{ n} 
   \end{align*}
\end{theorem} 
 
\begin{proof}
Observe that 
\begin{align} \nonumber 
& ( \widehat \beta_\ot -   \widehat \beta_\tn)^\top     \widehat \Sigma _\ot ( \widehat \beta_\ot -   \widehat \beta_\tn )   + 
\mathcal R_4(t)  +2 \mathcal R_2(t)  -2 \mathcal R_3(t)   
 -  ( \beta^*_\ot  - \beta^*_\tn  ) ^\top 
 \Sigma  ( \beta^*_\ot  - \beta^*_\tn  )      \\
 \label{eq:bias with tilde beta}
  =
&  ( \widehat \beta_\ot -   \widehat \beta_\tn)^\top     \widehat \Sigma _\ot ( \widehat \beta_\ot -   \widehat \beta_\tn )   + \mathcal R_4(t) -
(\beta^*_\ot  -\widehat   \beta _\tn  ) ^\top 
 \Sigma  ( \beta^*_\ot  -  \widehat \beta _\tn    )  
 \\
 \label{eq:bias with hat beta 2}
 + &  (\beta^*_\ot  -\widehat   \beta _\tn  ) ^\top 
 \Sigma  ( \beta^*_\ot  -  \widehat \beta _\tn    )    -  ( \beta^*_\ot  - \beta^*_\tn  ) ^\top 
 \Sigma  ( \beta^*_\ot  - \beta^*_\tn  )    +2  \mathcal R_2(t)   -2 \mathcal  R_3(t) .
\end{align}  
\
\\
{\bf Step 1.}
Note that 
\begin{align*}
\eqref{eq:bias with hat beta 2}
 =
 & 
  2  \beta^{*\top} _\ot \Sigma   (   \beta_\tn ^*  -\widehat \beta_\tn        ) 
  -2 \mathcal  R_3 
   \\
    +  &     \widehat \beta_\tn ^\top  \Sigma   \widehat \beta_\tn   - \beta_\tn ^{*\top}   \Sigma     \beta_\tn ^ {*\top}        +2 \mathcal  R_2 . 
 \end{align*} 
 By \Cref{lemma:bias term R2} and  \Cref{lemma:bias term R3}, it follows that 
 \begin{align*}
 \Bigg | &  \eqref{eq:bias with hat beta 2}
 + \Bigg \{  
- 2 \beta_{\tn}^{*\top} 
\left (\Sigma  -  \widehat \Sigma_\tn   \right ) ( \widehat \beta_\tn  -\beta^*_\tn )   
 -       \frac{ 2 }{n-t } \sum_{i=t+1}^n     \epsilon_i x_i^\top  \beta _\tn ^*        - 
\frac{   2 }{n-t } \sum_{i=t+1}^n          \beta_\tn     ^ {*\top}      X_i     x_i  ^\top 
 ( \beta_i ^*  - \beta^*_{(t,n]}  )     \Bigg \}  
 \\
 +& 2 \Bigg \{ \frac{ 1  }{n-t } \sum_{i=t+1} ^n x_i^\top    \beta_ \ot^{* } \epsilon_i     +  \frac{1}{n-t } \sum_{i=t+1} ^n   \beta_\ot^{*  \top}  x_ix_i^\top   (\beta_ i ^* - \beta^*_\tn )      -    \beta^{*\top} _\ot    (  \Sigma- \widehat \Sigma_\tn  )  (\beta_\tn^* -\widehat \beta_\tn )  \Bigg \} 
  \Bigg  | 
  \\ \le & C_1\frac{\s \log(pn)}{n}.
 \end{align*}  
\
\\
Note that 
\begin{align*}
&-\beta_{\tn}^{*\top} 
\left (\Sigma  -  \widehat \Sigma_\tn   \right ) ( \widehat \beta_\tn  -\beta^*_\tn )    -    \beta^{*\top} _\ot    (  \Sigma- \widehat \Sigma_\tn  )  (\beta_\tn^* -\widehat \beta_\tn )   
\\
= 
&( \beta^{* } _\ot -\beta_{\tn}^{*} ) ^\top  
\left (\Sigma  -  \widehat \Sigma_\tn   \right ) ( \widehat \beta_\tn  -\beta^*_\tn ) 
\\
\le &C_2  \|\beta^{* } _\ot -\beta_{\tn}^{*} \| _2     \sqrt   { \frac{  \log(p) } { n-t }}  \| \widehat \beta_\tn  -\beta^*_\tn\| _1
\\
\le & C_2'  \frac{ \s \log(pn) } {  \sqrt {(n-t) t}  } \le C_2''\frac{\s \log(pn) }{n }, 
\end{align*}   
where the first inequality follows from \Cref{condition:deviation}{\bf c} and the second inequality follows from \Cref{lemma:interval lasso}.
  So 
  \begin{align*}
 \Bigg | &  \eqref{eq:bias with hat beta 2}
-\Bigg \{   
      \frac{ 2 }{n-t } \sum_{i=t+1}^n     \epsilon_i x_i^\top  (   \beta _\tn ^*-\beta^*_\ot  )        + 
\frac{   2 }{n-t } \sum_{i=t+1}^n        (\beta_\tn     ^ {* }    -  \beta^*_\ot        )   ^\top  x_i     x_i  ^\top 
 ( \beta_i ^*  - \beta^*_{(t,n]}  )     \Bigg\}  
  \Bigg  |  \le & C_3\frac{\s \log(pn)}{n}.
 \end{align*}   
  \
  \\
 { \bf Step 2.}
  By \Cref{lemma:bias 3},
\begin{align*}
&  \Bigg |   \eqref{eq:bias with tilde beta}-\frac{2}{  t } \sum_{i=1}^t x_i^\top ( \beta_\ot ^*  -   \beta_\tn ^*  )\epsilon_i 
 - ( \beta_\ot ^*  -  \beta_\tn  ^*   )  ^\top  ( \widehat \Sigma_\ot  -\Sigma  )( \beta_\ot ^*  -  \beta_\tn  ^*   )  
\\ 
&  -\frac{2}{  t } \sum_{i=1}^t x_i^\top (  \beta_i  ^* -\beta_\ot ^*    )   x_i^\top ( \beta_\ot ^*  -  \beta_\tn ^*  )  \Bigg|  
 \le C_4 \frac{\s \log(pn)}{n}. 
 \end{align*} 
 The desired result follows from the previous two steps. 
\end{proof}

\subsection{Lasso Consistency}
\textbf{Proof of \Cref{lemma:interval lasso}}:
\begin{proof}   It follows from the definition of $\widehat{\beta}_\I$ that 
	\[
	 \frac{1}{|\I | }	\sum_{ i \in I} (y_i - x_i^{\top}\widehat{\beta}_\I )^2 + \frac{ \lambda}{ \sqrt{  |\I|  }}    \|\widehat{\beta} _\I\|_1 \leq  
	 \frac{1}{|\I | } \sum_{t \in \I} (y_i - x_i^{\top}\beta^*_\I )^2 + \frac{ \lambda}{ \sqrt{  |\I|  }}  \|\beta^*_\I  \|_1.
	\]
	This gives 
	\begin{align*}
		 \frac{1}{|\I | } \sum_{i \in \I} \bigl\{x_i ^{\top}(\widehat{\beta}_\I - \beta^*_\I )\bigr\}^2 +    \frac{ 2 }{|\I | }   \sum_{ i \in \I}(y_i - x_i^{\top}\beta^*_\I)x_i^{\top}(\beta^*_\I  - \widehat{\beta}_\I )  +  \frac{ \lambda}{ \sqrt{  |\I|  }}   \bigl\|\widehat{\beta}_\I \bigr\|_1 
		\leq  \frac{ \lambda}{ \sqrt{  |\I|  }}    \bigl\|\beta^*_\I \bigr\|_1,
	\end{align*}
	and therefore
	\begin{align}
		& \frac{1}{|\I | }	  \sum_{i  \in \I} \bigl\{x_i ^{\top}(\widehat{\beta} _\I  - \beta^*_\I )\bigr\}^2 + \frac{ \lambda}{ \sqrt{  |\I|  }} \bigl\|\widehat{\beta}_\I \bigr\|_1  \nonumber \\
		 \leq  &  \frac{ 2}{|\I | }	 \sum_{i \in  \I } \epsilon_i  x_i ^{\top}(\widehat{\beta}_\I  - \beta^*_\I ) 
+ 2   (\widehat{\beta}_\I  - \beta^*_\I )^{\top}\frac{  1 }{|\I | }	 \sum_{i\in   \I} x_ i  x_i^{\top}(    \beta^*_i -\beta^*_\I   )		 
		+  \frac{ \lambda}{ \sqrt{  |\I|  }}    \bigl\|\beta^*_\I \bigr\|_1   .	 \label{eq-lem10-pf-2}
	\end{align}
To bound	
		$\left\|\sum_{ i  \in \I} x_ i  x_i ^\top (\beta^*  _\I  -\beta^*_i ) \right\|_{\infty},  
$
	note  that  for any $j \in  \{ 1, \ldots, p\}$, the $j$-th entry of 
	\\$\sum_{i \in \I} x _ i x _i^\top (\beta^*_\I   -\beta_i )$ satisfies  
	\begin{align*}
		  \E \left\{\sum_{ i  \in \I}   x_{ij} x _ i ^\top (\beta^*_\I   - \beta^*_ i  )\right\} = \sum_{ i \in \I}  \E \{x_{ij} x _ i  ^\top \}\{\beta^*_\I  - \beta^*_ i \}  
		=     \mathbb{E}\{x _{1j}  x _1 ^\top \} \sum_{ i \in \I}\{\beta^*_\I    - \beta^*_ i  \} = 0.
	\end{align*}
	So   $  \E\{ \sum_{ i  \in \I} x_ i  x_i ^\top (\beta^*  _\I  -\beta^*_i )\} =0 \in \mathbb R^p.$ 
By \Cref{condition:deviation}{\bf c},
	\begin{align*}
	  \bigg|  (    \beta^*_i -\beta^*_\I   ) ^\top  \frac{  1 }{|\I | }	 \sum_{i\in   \I} x_i x_i^{\top} (\widehat{\beta}_\I  - \beta^*_\I )	 	   \bigg| \le & C_1 \big(\max_{1\le i \le n }  \|\beta^*_i -\beta^*_\I \|_2 \big) \sqrt { \frac{\log(pn) }{ | \I|  }} \| \widehat{\beta}_\I  - \beta^*_\I  \|_1   
	  \\
	  \le & C_2 \sqrt { \frac{\log(pn) }{ | \I| }} \| \widehat{\beta}_\I  - \beta^*_\I  \|_1  
	  \\
	  \le &  \frac{\lambda}{8\sqrt { |\I| } }   \|\widehat{\beta}_\I  - \beta^*_\I \|_1 
 	\end{align*}
 	where the second inequality follows from \Cref{lemma:beta bounded 1} and the last inequality follows from $\lambda = C_\lambda \sqrt { \log(pn) }$ with sufficiently large constant $C_\lambda$.
	In addition  by \Cref{condition:deviation}{\bf b},
	$$  \frac{ 2}{|\I | }	 \sum_{i \in  \I } \epsilon_i  x_i ^{\top}(\widehat{\beta}_\I  - \beta^*_\I )  \le C  \sqrt { \frac{ \log(pn) }{ |\I| } }\|\widehat{\beta}_\I  - \beta^*_\I \|_1  \le  \frac{\lambda}{8\sqrt { |\I| } }   \|\widehat{\beta}_\I  - \beta^*_\I \|_1 .$$
	So   \eqref{eq-lem10-pf-2}  gives 
	\begin{align*}
		  \frac{1}{|\I | }	  \sum_{i  \in \I} \bigl\{x_i ^{\top}(\widehat{\beta} _\I  - \beta^*_\I )\bigr\}^2 + \frac{ \lambda}{ \sqrt{  |\I|  }} \bigl\|\widehat{\beta}_\I \bigr\|_1  
		 \leq  \frac{\lambda}{4\sqrt { |\I| } }   \|\widehat{\beta}_\I  - \beta^*_\I \|_1  
		+  \frac{ \lambda}{ \sqrt{  |\I|  }}    \bigl\|\beta^*_\I \bigr\|_1   .	 
	\end{align*} 
Let $\Theta = \widehat  \beta _\I   - \beta^*_\I   $. The above inequality implies
\begin{align}
\label{eq:two sample lasso deviation 1} \frac{1}{|\I|} \sum_{i \in \I } \left( x_i^\top  \Theta \right)^2 + \frac{ \lambda}{2\sqrt{  |\I|  } }\|  (\widehat \beta _\I)   _{ S ^c}\|_1  
 \le & \frac{3\lambda}{2\sqrt{  |\I|  }  } \| ( \widehat  \beta _\I  -  \beta^*_\I  )    _{S} \| _1  ,
 \end{align}  
which also implies that 
$$ \frac{\lambda }{2}\| \Theta _{S^c}  \|_1 = \frac{ \lambda}{2 }\|   (\widehat \beta_\I)  _{ S ^c}\|_1    \le  
\frac{3\lambda}{2 } \|  ( \widehat  \beta _\I   - \beta^* _\I  ) _{S} \| _1  = \frac{3\lambda}{2 } \|  \Theta  _{S} \| _1 . $$
The above inequality and  \Cref{condition:deviation}{\bf a} imply that
$$ \frac{1}{|\I| } \sum_{i \in \I }  \left( x_i^\top  \Theta \right)^2  = \Theta^\top \widehat \Sigma _\I  \Theta  
\ge 
 \Theta^\top \Sigma \Theta - C_3\sqrt{ \frac{\s \log(pn)}{ |\I| }}  \| \Theta\|_2^2   \ge  \frac{c_x}{2} \|\Theta\|_2 ^2   ,$$
 where the last inequality follows from \Cref{assume: model assumption}{\bf e}.
Therefore \Cref{eq:two sample lasso deviation 1} gives
\begin{align}
\label{eq:two sample lasso deviation 2}  c'\|\Theta\|_2 ^2  + \frac{ \lambda}{2\sqrt{  |\I|  }  }\| ( \widehat \beta_\I  - \beta^*_\I ) _{ S ^c}\|_1  
 \le   \frac{3\lambda}{2\sqrt{  |\I|  }  } \| \Theta_{S} \| _1  \le \frac{3\lambda \sqrt \s }{2 \sqrt{  |\I|  } } \| \Theta  \| _2 
 \end{align}  
 and so 
 $$ \|\Theta\|_2  \le   \frac{C \lambda \sqrt  s}{\sqrt{| \I |} } . $$ 
 This  and the last inequality of  \Cref{eq:two sample lasso deviation 2} also implies that 
 $  \| \Theta_{S} \| _1 \le  \frac{C\lambda \s}{\sqrt{|\I| }}. $ 
 Since 
  $ \| \Theta_{S^c } \| _1 \le 3 \| \Theta_{S} \| _1 ,$ 
it also holds that 
$$\| \Theta  \| _1 = \| \Theta_{S  } \| _1 +\| \Theta_{S^c } \| _1  \le 4 \| \Theta_{S} \| _1 \le  \frac{4C\lambda \s}{\sqrt{|\I|} } .$$
 \end{proof}

\subsection{Distributions Satisfying \Cref{condition:deviation} and \Cref{condition:process}}
\label{section:distribution}
  
 In this subsection, we show that \Cref{condition:deviation}  and \Cref{condition:process}  hold  with high probability when  $ \{x_i,\epsilon_i \}_{i=1}^n $ are sampled from independent sub-Gaussian distributions. We begin by noting that 
  \Cref{condition:process} is a well known functional CLT result. See e.g., Theorem  A.4 of \cite{francq2019garch} and references therein.
 Throughout the section, we write 
 $$z\sim SG( \sigma_z^2) $$
 if $z$ is a sub-Gaussian random variable such that 
 $ \|z\|_{\psi_2}  \le \sigma_z.  $

\begin{theorem}
\label{theorem:restricted eigenvalues 2}
Suppose   \Cref{assume: model assumption}  holds. Denote     $  \E(x_ix_i^\top ) =\Sigma   $ and  $\mathcal C_S  := \{ v : \mathbb R^p : \| v_{S^c }\|_1 \le 3\| v_{S}\|_1 \}  $.
Then there exists constants $c$ and $C$ such that for all $\eta\le 1$, 
\begin{align} 
\mathbb P \left( \sup_{v \in \mathcal C _S , \|v\|_2=1   } \left | v^\top ( \widehat \Sigma -\Sigma )  v \right| \ge C\eta   \Lambda_{\max} (\Sigma)   \right)
\le 2\exp( -c n\eta ^2 + 2s \log(p) ) .
\end{align}
\end{theorem}
\begin{proof}
This is a well known restricted eigenvalue property for sub-Gaussian design. The proof can be found  in \cite{zhou2009restricted} or \cite{loh2011high}.   
\end{proof}  
\
\\
The following corollary, being a direct consequence of \Cref{theorem:restricted eigenvalues 2}, justifies  \Cref{condition:deviation}{\bf a}. 
\begin{corollary} \label{corollary:restricted eigenvalues 2} Denote $\mathcal C_S  := \{ v : \mathbb R^p : \| v_{S^c }\|_1 \le 3\| v_{S}\|_1 \}  $. Under   \Cref{assume: model assumption}, with probability at least $1- n ^{-5}$, it holds that
$$  \left | v^\top ( \widehat \Sigma_\I  -\Sigma )  v \right| \le C \sqrt { \frac{\s \log(pn) }{|\I| } }     \|v\|_2^2   
 $$
   for all $v \in \mathcal C _S $ and all $\I \subset  (0, n]$ such that $| \I |\ge \zeta n $. 
   
\end{corollary}
\begin{proof} 
For any $\I \subset  (0, n]$ such that $| \I |\ge \zeta n $, by \Cref{theorem:restricted eigenvalues 2}, it holds that 
\begin{align*} 
\mathbb P  \left( \sup_{v \in \mathcal C _S , \|v\|_2=1   } \left | v^\top ( \widehat \Sigma _\I -\Sigma )  v \right| \ge C\eta   \Lambda_{\max} (\Sigma)   \right)
\le 2\exp( -c \zeta n\eta ^2 + 2s \log(p) ) .
\end{align*}
Let  $\eta= C_1 \sqrt { \frac{\s \log(pn) }{|\I | } } $ for sufficiently large constant $C_1$. With probability at least 
$(pn)^{-7} $, 
$$\sup_{v \in \mathcal C _S , \|v\|_2=1   } \left | v^\top ( \widehat \Sigma _\I -\Sigma )  v \right| \ge C_2 \sqrt { \frac{\s \log(pn) }{|\I |  } } .$$
Since there are at most $n^2$ many different  choice of $ \I$, the desire result follows from a union bound argument. 
\end{proof} 
\
\\
The following  lemma  justifies  \Cref{condition:deviation}{\bf b} and {\bf c}.  Note that \Cref{condition:deviation}{\bf d}  follows from the same argument as  that of \Cref{condition:deviation}{\bf b}.

\begin{lemma} \label{lemma:consistency}
Suppose $n\ge C\s\log(p)$ for sufficiently large $C$ and that \Cref{assume: model assumption} holds. Then  with probability at least $1-n^{-5} $, it holds that
\begin{align}
    \left |  \frac{1}{|\I | } \sum_{i \in \mathcal I }  \epsilon_i  x_i^\top \beta \right| \le C\sqrt { \frac{\log(pn)}{|\I |   } } \|\beta\|_1 
\end{align}
for all $ \beta \in \mathbb R^p$ and all $\mathcal I \subset (0, n]$ such that $|\I|\ge \zeta n$.    
\\
\\
 Let $\{ u_i\}_{i=1}^n \subset  \mathbb R^p$  be a collection of  deterministic vectors. Then  it holds that with probability at least $1-n^{-5} $
\begin{align}\label{eq:independent condition 1c}
       \left |  \frac{1}{|\I |   } \sum_{i \in \I }    u^\top_i  x _i   x _i^\top \beta  - \frac{1}{ |\I |  } \sum_{i \in \mathcal I }   u^\top  _i  \Sigma \beta  \right| \le C \left(  \max _{1\le i \le n }\|u_i\|_2  \right) \sqrt {  \frac{\log(pn)}{|\I |   } } \|\beta\|_1    
 \end{align}
 for all $ \beta \in \mathbb R^p$ and all $\mathcal I \subset (0, n]$ such that $|\I|\ge \zeta n$.    
\end{lemma} 
 \begin{proof}
 The first bound is a well-known inequality. The proof of the first bound is also similar and simpler  than the second one. For conciseness, the proof of the first bound is omitted. 
 \\
 \\
 For \Cref{eq:independent condition 1c}, it suffices to show that 
 \begin{align*} 
      P \bigg\{  \left |  \frac{1}{ |\I |  } \sum_{i \in \I }    u^\top_i  x_i   x_i^\top \beta  - \frac{1}{|\I |  } \sum_{i \in \mathcal I }   u^\top  _i  \Sigma \beta  \right| \ge C \left(  \max _{1\le i \le n }\|u_i\|_2  \right) \sqrt {  \frac{\log(pn)}{ |\I |   } } \|\beta\|_1    \bigg\} \le n^{-7}
 \end{align*} 
 for any $\I  $ such that $|\I|\ge \zeta n $. Note that 
 \begin{align*}
 &\left|  \frac{1}{ |\I |  } \sum_{i\in \I }  u^\top  _i x_i   x_i^\top \beta   - \frac{1}{ |\I |  }  \sum_{i\in \I }   u^\top_i  \Sigma  \beta \right| 
 \\
 =
 & \left|  \left(  \frac{1}{|\I |  } \sum_{i\in \I }   u^\top _i   x_i   x_i^\top     - \frac{1}{|\I |  }    \sum_{i\in \I }   u^\top_i  \Sigma  \right) \beta  \right| 
 \\
 \le 
 & \max_{1\le j \le p }  \left |  \frac{1}{|\I |  } \sum_{i\in \I }   u^\top  _i   X_i   X_{i,j}      - \frac{1}{|\I |  }    \sum_{i\in \I }   u^\top_i  \Sigma  (, j)  \right| \|\beta\|_1.
 \end{align*}
 Note that 
 $ \E(u^\top  _i  x_i   x_{i,j} ) =u^\top _i \Sigma  (, j)    $ and in addition, 
 $$u^\top _ i   x_i   \sim SG ( C_x^2 \| u_i \|_2^2)    \quad \text{ and } \quad x_{i,j} \sim SG ( C_x^2  )   .$$
 So $u^\top _i   x_i   x_ {i,j} $  is a sub-exponential random variable such that 
 $$ u^\top _i   x_i   x_ {i,j} \sim SE( C_x^4 \| u_i \|_2^2 ).$$
 As a result, for $\gamma<1$ and every $j$,
$$ P \left( \left |  \frac{1}{|\I |  } \sum_{i \in \I }  u^\top _ i    x_i   x_{i,j}      - u^\top \Sigma  (, j)  \right|  \ge \gamma \sqrt {  \max_{1\le i\le n }    C_x^4 \| u_i \|_2^2   }\right) \le \exp(-c \gamma^2 |\I |  ). $$
By union bound,
$$ P \left(  \sup_{1\le j \le p }\left |  \frac{1}{|\I |  } \sum_{i \in \I}  u^\top _ i     x_i   x_{i,j}      - \frac{1}{|\I |  } \sum_{i \in\I}  u^\top _i  \Sigma  (, j)  \right|  \ge \gamma 
C_x^2  \left(  \max _{1\le i \le n }\|u_i\|_2  \right)  \right) \le p\exp(-c \gamma^2 n). $$
The desired result follows from  taking $\gamma =C_\gamma \sqrt { \frac{\log(pn) }{ n}  }$ for sufficiently large $C_\gamma$ and an union bound over all possible interval $\I $ such that $|\I| \ge \zeta n. $
 \end{proof}

\section{  $\beta$-mixing Time Series  }
\subsection{The Null Distribution under $\beta$-mixing }\label{sec:The Null Distribution under beta-mixing}
\textbf{Proof of \Cref{theorem:main_beta} (under $H_0$)}:
 \begin{proof}
 We follow the same strategy used  in the proof of \Cref{theorem:main null}.
 Consider the following conditions:
\begin{condition} \label{condition:deviation 2}
\
\\
{\bf c.} There exists an absolute constant $C$ such that given a fixed $ u \in \mathbb R^p  $,
we have that $$
  \left |  \frac{1}{|\I | } \sum_{i \in \I }   \big\{   u^\top   x_i   x_i^\top \beta  -      u^\top  \Sigma \beta   \big\} \right| \le C     \|u \|_2    \sqrt {  \frac{\log(pn)}{ | \I |  } } \|\beta\|_1   \quad  \text{for all } \beta \in \mathbb R^p  .$$ 
 \
 \\
 \\
{\bf e.}   Suppose that $$
  \bigg |  \frac{1}{|\I | } \sum_{i \in \I }   \big\{   u^\top   x_i   x_i^\top \beta    -       u^\top   \Sigma  \beta \big\}    \bigg | \le C   \sqrt {    \frac{\s \log(pn)}{ | \I |  } } \|\beta\|_1  $$
  for all $u \in \mathbb R^p$ such that $ \|u\|_2\le 1$ and $ u\in \mathcal C(S)   $ and all $\beta \in \mathbb R^p$.
\end{condition}

\ 
\\ 
To justify the limiting distribution under  the null hypothesis $H_0$, it suffices to justify \Cref{condition:deviation} {\bf a}, {\bf b}, {\bf d}, \Cref{condition:process} and \Cref{condition:deviation 2} under the new $\beta$-mixing assumption.
\
\\
\\
Observe that \Cref{condition:deviation 2}{\bf c} is a  less general  version of \Cref{condition:deviation}{\bf c}. However, under the null hypothesis  
$$H_0:  \beta_1^* = \cdots =\beta^*_n,$$
all we need is \Cref{condition:deviation 2}{\bf c}, as  the general version in  \Cref{condition:deviation}{\bf c} is used for power analysis and alternative distribution. 
\\
\\
 We also note that for the  independent  setting,   in the proof of    \Cref{lemma:bias term R3},  we  verify  \Cref{eq:using independence}   using the fact that 
$ \widehat \beta^*_\ot -\beta _\ot  $ is independent of the data
$\{ x_i, \epsilon_i\}_{i=t+1}^n.  $
   In the presence of temporal dependence, we  need to  
 justify \Cref{eq:using independence} using \Cref{condition:deviation 2}{\bf e}.  
 \\
 \\
 In view of the proof of \Cref{theorem:main null},  the desired result follows immediately from  \Cref{lemma:interval lasso consistency beta mixing}, \Cref{lemma: beta deviation 2},  \Cref{lemma: beta deviation 3}   and \Cref{lemma:beta to gaussian process}  in \Cref{subsection: additional beta}. 
 \\
 \\
For illustration purposes, we  show  \Cref{eq:using independence} as follows.
 Suppose all the good events hold in \Cref{lemma:beta interval lasso}. Then 
\begin{align*} \nonumber&(  \widehat \beta_\ot  - \beta^*_\ot ) ^\top ( \Sigma- \widehat \Sigma_\tn  )  (\beta_\tn^* -\widehat \beta_\tn ) 
\\
\le& C_2 \sqrt { \frac{\s \log(pn)}{n } } \| \widehat \beta_\ot  - \beta^*_\ot  \|_2    \| \beta_\tn^* -\widehat \beta_\tn\| _1 \nonumber 
\\
\le & C _2'  \sqrt { \frac{\s \log(pn)}{n } }  \sqrt { \frac{C_3\s \log(pn) }{ n }}    s\sqrt { \frac{C_3'  \log(pn) }{ n }}  
\\
=& C_4 \s ^2 \bigg( \frac{  \log(pn)}{n } \bigg)^{3/2 } \le C_4' \frac{\s \log(pn)}{n },
  \end{align*} 
where the first inequality follows from \Cref{condition:deviation 2}{\bf e},   the second inequality follows from \Cref{lemma:beta interval lasso}, and the last inequality follows from the assumption  that 
$ \s\sqrt { \frac{\log(pn)}{n}} =o(1)$.
\end{proof}

\subsection{Power Analysis under $\beta$-mixing}  
\textbf{Proof of \Cref{theorem:main_beta} (under $H_a  $)}:
  \begin{proof}
  The proof is analogous to the proof of \Cref{corollary:iid power}. 
    Suppose $H_a$ holds.
    \\
    \\
   {\bf Step 1.} Denote 
$$ \kappa = \max_{1\le k \le K }\kappa_k . $$    
    It follows from  \Cref{eq:equivalence Hypothesis} that there exists $\eta$, being a change point of $\{ \beta^*_i\}_{i=1}^n $ such that 
    $$ ( \beta^*_{(0, \eta] }    - \beta^*_{ (\eta,n] }  ) ^\top  \Sigma  ( \beta^*_{(0, \eta] }    - \beta^*_{ (\eta,n] }   )   \ge c_1 \kappa.  $$
Therefore, to justify the test $\psi$ has power going to 1 under $H_a$,  it suffices to show that for sufficiently large $n$,   
 \begin{align} \label{eq:beta power signal upper bound 1}    
 & \frac{ \sqrt n }{2\sigma_\epsilon \sigma_\xi }\bigg\{      ( \widehat \beta_{(0,  \eta ] } -   \widehat \beta_{ ( \eta , n ]}  )^\top     \widehat \Sigma _{(0, \eta] } ( \widehat \beta_{(0, \eta] } -   \widehat \beta_{ (\eta, n ]}   )   + 
\mathcal R_1(\eta)  +2 \mathcal R_2(\eta)  -2  \mathcal  R_3(\eta)   \bigg\}    >  \mathcal G_\alpha    ,
\\\label{eq:beta power signal upper bound 2} 
&  \frac{ \sqrt n }{2\sigma_\epsilon \sigma_\xi }   \bigg\{    ( \widehat \beta_{(  \eta , n ] } -   \widehat \beta_{ (0, \eta  ]}  )^\top     \widehat \Sigma _{( \eta , n ] } ( \widehat \beta_{(  \eta , n ] } -   \widehat \beta_{ (0, \eta  ]}   )   + 
\mathcal R_1 ' ( \eta )  +2 \mathcal R_2' ( \eta )  - 2 \mathcal   R_3' (\eta )    \bigg\}   >  \mathcal G_\alpha    .
 \end{align}   
The justification of  \Cref{eq:beta power signal upper bound 1} is provided as the justification of  \Cref{eq:beta power signal upper bound 2}  is the same by symmetry. 
\\
\\
{\bf Step 2.} In view of \Cref{corollary:iid power}, we need to justify that under $H_a $, \Cref{condition:deviation} {\bf a}, {\bf b}, {\bf c}, {\bf d},  and \Cref{condition:deviation 2}
continue to hold    in  the intervals $(0, \eta]$ and $(\eta, n]$. For illustration purpose, in \Cref{corollary: beta deviation 1 with change points}  we show that  \Cref{condition:deviation} {\bf a} and {\bf b} holds in the interval 
   $(0, \eta]$.  The justification of other conditions follow from similar arguments and is omitted for brevity. 
\
\\
\\
 {\bf Step 3.} Using  \Cref{condition:deviation} {\bf a}, {\bf b}, {\bf c}, {\bf d},  and \Cref{condition:deviation 2} in  the intervals $(0, \eta]$ and $(\eta, n]$, it follows from   the same argument as in   \Cref{theorem:alternative},   
\begin{align*}     
& \Bigg |   ( \widehat \beta_{(0, \eta] } -   \widehat \beta_{ (\eta,n] })^\top     \widehat \Sigma _{(0, \eta] } ( \widehat \beta_{(0, \eta] }  -   \widehat \beta_{ (\eta,n] }  )   + 
\mathcal R_1(\eta  )  +2 \mathcal R_2(\eta )  -2  \mathcal  R_3(\eta )    
 \\
 & -     ( \beta^*_{(0, \eta] }    - \beta^*_{ (\eta,n] }  ) ^\top  \Sigma  ( \beta^*_{(0, \eta] }    - \beta^*_{ (\eta,n] }   )      -  \big\{ \mathfrak T_{n, \xi} (\eta ) +\mathfrak S_n   (\eta )  \big\}    
\Bigg|  \le
  C_1 (1+\sigma_\xi ) \frac{\s \log(pn)}{ n}  
   . \end{align*}  
Here 
 \begin{align} 
 \label{eq:Gn at the change point beta mixing 0}   \mathfrak T_{n, \xi} ( \eta )  = &\frac{2}{\eta } \sum_{i=1}^{ \eta }      x_i  ^\top     (  \beta^*_ {(0, \eta] }    - \beta^*_ {( \eta , n]}   )      \epsilon_i   
 	 +  \frac{2}{\eta   } \sum_{i=1}^ { \eta }      x_i  ^\top     (  \beta^*_1 - \beta^*_{(0, \eta ]}    )      \xi_i    
 	   +  \frac{2}{\eta } \sum_{i=1}^ {\eta  }    \xi_i\epsilon_i   
 \\  \nonumber
   + &   ( \beta^*_{(0, \eta  ] }   -     \beta_{(\eta   , n]}   ^*  )  ^\top  (  \widehat \Sigma  _{(0, \eta   ] }   -\Sigma  )  (  \beta^*_{(0, \eta  ] }    -    \beta ^* _{(\eta  , n]}     )   +
  \frac{2}{  \eta   } \sum_{i=1}^{ \eta   }    (  \beta_i ^* -\beta_{(0, \eta] }   ^*    ) ^\top x_i   x_i^\top ( \beta_{(0,\eta  ] }  ^*  -  \beta_{(\eta  , n]}   ^*  )   , \
  \\ \label{eq:Gn at the change point beta mixing}
\mathfrak S_n  (\eta  )   =  &     \frac{2 }{n-\eta  } \sum_{i=\eta   +1} ^n x_i^\top   ( \beta_ {(0, \eta   ] }   ^{* } -\beta_{ (\eta , n ]}  ^*)  \epsilon_i 
  +  \frac{2 }{n-\eta   } \sum_{i=\eta  +1} ^n     ( \beta_{ (\eta   , n ]}  ^* - \beta_ {(0, \eta   ] }  ^{* }  )  x_ix_i^\top   (\beta_ i ^* - \beta^*_{ (\eta   , n ]}   )  .
 \end{align}      
   \
   \\
  Then under   $H_a$, 
\begin{align} \label{eq:beta power signal upper bound}   &    ( \widehat \beta_{(0, \eta] } -   \widehat \beta_{ (\eta,n] })^\top     \widehat \Sigma _{(0, \eta] } ( \widehat \beta_{(0, \eta] }  -   \widehat \beta_{ (\eta,n] }  )   + 
\mathcal R_1(\eta  )  +2 \mathcal R_2(\eta )  -2  \mathcal  R_3(\eta )     
 \\ \nonumber
\ge    &   \kappa^2      -  | \mathfrak T_{n, \xi}   ( \eta ) | -| \mathfrak  S_n( \eta )   |    
-  C_1 (1+\sigma_\xi ) \frac{\s \log(pn)}{ n} .
 \end{align} 
 So it suffices to show that 
\begin{align} \label{eq:beta power signal upper bound2}   \kappa^2      -  | \mathfrak T_{n, \xi}  (\eta ) | -| \mathfrak   S_n(\eta )   |    
-  C_1 (1+\sigma_\xi ) \frac{\s \log(pn)}{ n}   \ge  \frac{ 2 \mathcal G_{\alpha   } \sigma_\epsilon\sigma_\xi }{ \sqrt  {   n} }  . 
\end{align} 
\
\\
The rest of the proof follows from the  the  same  argument of \Cref{corollary:iid power} and the  union bound argument  (similar as that in \Cref{corollary: beta deviation 1 with change points}). 
\\
\\
For illustration purposes, in \Cref{lemma:Gn at the change point beta mixing},     we demonstrate how to handle  the term $| \mathfrak   S_n(\eta )   |$ in \Cref{eq:beta power signal upper bound2}. More precisely, we show that  there exists a constant $C_2$  such that with probability goes to 1, 
$$|\mathfrak  S_n (\eta ) | \le C_2 \sqrt {  \frac{ D_n }{n} }  \kappa    ,$$
where $D_n  $ is any slowly diverging sequence. 
\end{proof} 

\begin{lemma}
\label{lemma:Gn at the change point beta mixing} 
Suppose all the assumptions in \Cref{theorem:main_beta} hold.  Let $\mathfrak{S}_n (\eta)$ be defined as in 
\Cref{eq:Gn at the change point beta mixing}. Then 
 with probability goes to 1, 
$$|\mathfrak  S_n (\eta ) | \le C_2 \sqrt {  \frac{ D_n }{n} }  \kappa    ,$$
where  $ \kappa = \max_{1\le k \le K }\kappa_k  $ and   $D_n  $ is any slowly diverging sequence. 
\end{lemma}
\begin{proof}
{\bf Step 1.} 
We show that 
$$ \frac{1}{n - \eta } \sum_{i=\eta+1 }^{ n  }      x_i  ^\top     (  \beta^*_ {(0, \eta] }    - \beta^*_ {( \eta , n]}   )      \epsilon_i  =O\bigg(\kappa  \sqrt { \frac{D_n}{n}} \bigg)$$
with probability goes to 1.
Let 
$$z_i =       x_i  ^\top     (  \beta^*_{(0,  \eta ] }    - \beta^*_ {(\eta , n]} )      \epsilon_i      .$$
As a result,  $z_i$ is strictly stationary.
Since $z_i$ is measurable to $\sigma(x_i, \epsilon_i )$, the process $\{ z_i\}_{i=1}^\eta $ is $\beta$-mixing with the same mixing coefficient as $\{x_i ,\epsilon_i \}_{i=1}^\eta  $.
In addition, 
\begin{align*}  \| \frac{1}{\kappa }z_i\|_{\psi _{\gamma_1/2 } }  =& \frac{1}{\kappa } \|  x_i  ^\top     (  \beta^*_{(0,  \eta ] }    - \beta^*_ {(\eta , n]} )  \epsilon_i \|_{\psi _{\gamma_1/2 } } 
\\
\le &  \frac{c_ 1}{\kappa } \| x_i  ^\top     (  \beta^*_{(0,  \eta ] }    - \beta^*_ {(\eta , n]} )    \|_{ \psi_{ \gamma_1}}  \|  \epsilon_i    \|_{ \psi_{ \gamma_1}} 
\\
\le &   \frac{c_1 }{\kappa } K_X   \|  \beta^*_{(0,  \eta ] }    - \beta^*_ {(\eta , n]} \|_2   K _\epsilon 
=c_1' K_X  K_\epsilon,  
\end{align*} 
where $c_1$ is a constant only depending on $\gamma_1$ and the last equality follows from \Cref{eq:cumsum upper bound}. 
Therefore by \Cref{theorem:beta deviation},
\begin{align*}  \mathbb P \bigg\{ \frac{1}{ (n- \eta)   } \bigg|   \frac{1}{\kappa }\sum_{i =\eta+  1}^ n   z_i \bigg| \ge \delta \bigg\} 
\le  n \exp\big (-c_2 (\delta \eta )^\gamma\big) + \exp(- c_3\delta^2 \eta ), 
\end{align*} 
where 
$ \gamma  < 1  $ is defined in \Cref{assume:order_betamixing}.  Let $ \delta = C_\delta \sqrt { \frac{ D_n}{ n }} $ for sufficiently large $C_\delta$.
Since  $ n -  \eta \ge \zeta n$, this 
 leads to 
\begin{align*}  \mathbb P \bigg\{ \frac{1}{  n -  \eta  } \big|   \sum_{i =\eta+ 1}^ n  
  x_i  ^\top     (  \beta^*_{(0, \eta] }    - \beta^*_ {(\eta, n]} )      \epsilon_i 
 \big| \ge C_\delta \kappa  \sqrt { \frac{ D_n}{ n }}   \bigg\} =o(1),
\end{align*}  
as desired. 
\\
\\
{\bf Step 2.} In this step, we show that 
\begin{align*} 
\mathbb P  \bigg\{\Bigg|   \frac{1}{n- \eta } \sum_{i=\eta +1} ^n     ( \beta_ {(0, \eta  ] }  ^{* } -\beta_{ (\eta  , n ]}  ^*)  x_ix_i^\top   (\beta_ i ^* - \beta^*_{ (\eta  , n ]}   )  \Bigg|  
\ge    C_3 \kappa  \sqrt { \frac{ D_n}{ n }}   \bigg\} =o(1)
\end{align*}  
 Let $\eta= \eta_q $ for some $q \in \{ 1,\ldots, K\} $.  Denote 
$\mathcal J_k  =(\eta_{k-1},\eta_k] $. Then 
$$ (\eta, n ] = \bigcup _{ k=q+1}^{K+1} \mathcal J_k .  $$
    Observe that

\begin{align} \nonumber 
&\Bigg|   \frac{1}{n- \eta } \sum_{i=\eta +1} ^n     ( \beta_ {(0, \eta  ] }  ^{* } -\beta_{ (\eta  , n ]}  ^*)  x_ix_i^\top   (\beta_ i ^* - \beta^*_{ (\eta  , n ]}   )  \Bigg| 
\\   \nonumber 
= & \Bigg|   \frac{1}{n- \eta } \sum_{i= \eta  +1} ^n     ( \beta_ {(0, \eta ] }  ^{* } -\beta_{ (\eta , n ]}  ^*)  x_ix_i^\top   (\beta_ i ^* - \beta^*_{ (\eta , n ]}   )  
-   ( \beta_ {(0, \eta ] }  ^{* } -\beta_{ (\eta , n ]}  ^*)  \Sigma   (\beta_ i ^* - \beta^*_{ (\eta , n ]}   )  \Bigg| 
\\  \nonumber  
\le & \sum_{k=q+1 }^{K+1 }\Bigg|   \frac{1}{n- \eta }  \sum_{i \in \mathcal J_k } ^n     ( \beta_ {(0, \eta ] }  ^{* } -\beta_{ (\eta , n ]}  ^*)  x_ix_i^\top   (\beta_ i ^* - \beta^*_{ (\eta , n ]}   )  
-   ( \beta_ {(0, \eta ] }  ^{* } -\beta_{ (\eta , n ]}  ^*)  \Sigma   (\beta_ i ^* - \beta^*_{ (\eta , n ]}   )  \Bigg|  
\\ \label{eq:beta mixing power step 2}
= &  \sum_{k=q+1 }^{K+1 }\Bigg|  \frac{1}{n- \eta }  \sum_{i \in \mathcal J_k } ^n     ( \beta_ {(0, \eta ] }  ^{* } -\beta_{ (\eta , n ]}  ^*)  x_ix_i^\top   (\beta_  {\mathcal J_k } ^* - \beta^*_{ (\eta , n ]}   )  
-   ( \beta_ {(0, \eta ] }  ^{* } -\beta_{ (\eta , n ]}  ^*)  \Sigma   (\beta_  {\mathcal J_k } ^* - \beta^*_{ (\eta , n ]}   )  \Bigg|   
.\end{align}
where we use the fact that $\beta_i^*$ is unchanged in each of the interval $ \mathcal J_k$.
For each interval $\mathcal J_k$, the time series 
$$  w_i := ( \beta_ {(0, \eta ] }  ^{* } -\beta_{ (\eta , n ]}  ^*)  x_ix_i^\top   (\beta_  {\mathcal J_k } ^* - \beta^*_{ (\eta , n ]}   )  
 $$
is strictly stationary beta-mixing. In addition, by similar calculations as in  the previous step, it follows that 
\begin{align*}  
\| \frac{1}{\kappa   } w _i\|_{\psi _{\gamma_1/2 } }  =& \frac{1}{\kappa   } \|   ( \beta_ {(0, \eta ] }  ^{* } -\beta_{ (\eta , n ]}  ^*)  x_ix_i^\top   (\beta_  {\mathcal J_k } ^* - \beta^*_{ (\eta , n ]}   )    \|_{\psi _{\gamma_1/2 } } 
\\
\le &  \frac{c_ 1}{\kappa   } \| ( \beta_ {(0, \eta ] }  ^{* } -\beta_{ (\eta , n ]}  ^*)  X_i     \|_{ \psi_{ \gamma_1}}  \|    x_i^\top   ( \beta_  {\mathcal J_k } ^* - \beta^*_{ (\eta , n ]}  )       \|_{ \psi_{ \gamma_1}} 
\\
\le &   \frac{c_1 }{\kappa   } K_X ^2    \|  \beta^*_{(0,  \eta ] }    - \beta^*_ {(\eta , n]} \|_2   \|\beta_  {\mathcal J_k } ^* - \beta^*_{ (\eta , n ]} \|_2 
=c_2  K_X ^2 .
\end{align*}  
\
\\
Therefore by \Cref{theorem:beta deviation},
\begin{align*}  \mathbb P \bigg\{ \frac{1}{ \kappa   |\mathcal J_k|    } \big|   \sum_{i  \in \mathcal J_k }^ n  w _i -E(w_i ) \big| \ge \delta \bigg\} 
\le  n \exp\big (-c_2 (\delta \eta )^\gamma\big) + \exp(- c_3\delta^2 \eta ), 
\end{align*} 
where 
$ \gamma  < 1  $ is defined in \Cref{assume:order_betamixing}.  Let $ \delta = C_\delta \sqrt { \frac{ D_n}{ n }} $ for sufficiently large $C_\delta$.
Since  $$ n -  \eta \asymp n \asymp | \mathcal J_k| ,$$ this 
 leads to 
\begin{align*}  \mathbb P \bigg\{   \Bigg|  \frac{1}{n- \eta }  \sum_{i \in \mathcal J_k } ^n     ( \beta_ {(0, \eta ] }  ^{* } -\beta_{ (\eta , n ]}  ^*)  x_ix_i^\top   (\beta_  {\mathcal J_k } ^* - \beta^*_{ (\eta , n ]}   )  
-   ( \beta_ {(0, \eta ] }  ^{* } -\beta_{ (\eta , n ]}  ^*)  \Sigma   (\beta_  {\mathcal J_k } ^* - \beta^*_{ (\eta , n ]}   )  \Bigg|    \ge C_\delta \kappa  \sqrt { \frac{ D_n}{ n }}   \bigg\} =o(1).
\end{align*}  
 By a union bound argument and \Cref{eq:beta mixing power step 2},   it holds that 
\begin{align*} 
\mathbb P  \bigg\{\Bigg|   \frac{1}{n- \eta } \sum_{i=\eta +1} ^n     ( \beta_ {(0, \eta  ] }  ^{* } -\beta_{ (\eta  , n ]}  ^*)  x_ix_i^\top   (\beta_ i ^* - \beta^*_{ (\eta  , n ]}   )  \Bigg|  
\ge    C_\delta (K+1) \kappa  \sqrt { \frac{ D_n}{ n }}   \bigg\} =o(1)
\end{align*} 
as desired.

 \end{proof}

\subsection{Alternative Distribution under $\beta$-mixing}   
      Consider the alternative hypothesis
\begin{align*}
H_a: &   \text{ there exists at least one change point in } \{ \beta_i^*\}_{i=1}^n .  
\end{align*}
 Let  $r  \in [\zeta, 1-\zeta]  $  and  $t = \lfloor r n\rfloor  $.  Denote 
\begin{align*}   
\mathfrak S(r ) :=   & ( \widehat \beta_ {(0, t] }  -   \widehat \beta_{(t, n] } )^\top     \widehat \Sigma _{(0, t] } 
( \widehat \beta_{(0, t] }  -   \widehat \beta_{(t, n] }  )  + 
\mathcal R_1 (t) 
  +2 \mathcal R_2(t)  -2 \mathcal R_3(t)       
\\
 + &    ( \widehat \beta_ {(0, t] }  -   \widehat \beta_{(t, n] } )^\top     \widehat \Sigma _{(t,n ] } 
( \widehat \beta_{(0, t] }  -   \widehat \beta_{(t, n] }  )  + 
\mathcal R_1 '  (t) 
  +2 \mathcal R_2 ' (t)  -2 \mathcal R_3 '(t)        ,
\end{align*}  
 where  $  \mathcal R_1 (t) 
, \mathcal R_2(t)    , \mathcal R_3(t)      $ are  defined as in  \Cref{theorem:alternative} and $  \mathcal R_1' (t) 
, \mathcal R_2'(t)    , \mathcal R_3'(t)      $ are defined as in  \Cref{theorem:alternative counterpart}. 
Furthermore, denote 
\begin{align*}  \mu (r )  = &  \lim_{n \to \infty }    ( \beta^*_{(0, t] }  - \beta^*_ {(t,n] } ) ^\top  \Sigma  ( \beta^*_{(0, \eta] }   - \beta^*_ {(\eta,n] }  ), 
 \\
\Phi _L (r ) = & \lim_{n\to \infty } \bigg\{ 16  \sigma_{L,1}^2  +   \frac{4}{t  }   \sum_{i=1}^t   \sigma_\xi^2  ( \beta_i^* -   \beta^*_{ (0, t ] }   ) ^\top \Sigma   ( \beta_i^* -   \beta^*_{ (0, t ] }   )  
 + 4 \sigma_\epsilon^2    \sigma_\xi^2    + \sigma_{L,2}^2  \bigg\}  , \quad \text{and} 
\\
\Phi _R (r )  = &  \lim_{n\to \infty } \bigg\{ 16  \sigma_{R,1}^2  +   \frac{4}{ n- t  }   \sum_{i=t +1}^{n } \sigma_\xi^2  ( \beta_i^* -   \beta^*_{ (t , n]  }   ) ^\top \Sigma   ( \beta_i^* -  \beta^*_{ (t , n]  }   ) 
  + 4 \sigma_\epsilon^2    \sigma_\xi^2    +  \sigma^2_{R,2}\bigg\} ,
\end{align*}   
where 
\begin{align}\label{eq:long run variance term 1}
\sigma_{L,1} ^2 = &\lim_{n \to\infty }Var \bigg\{ \frac{1}{ \sqrt {t  }   } \sum_{i=1}^ t    x_i  ^\top     (  \beta^*_{ (0,  t ] }  - \beta^*_ { ( t , n] }  )      \epsilon_i  \bigg\}  ,
\\\label{eq:long run variance term 2}
 \sigma_{R,1} ^2 =& \lim_{n \to\infty }  Var \bigg\{ \frac{1}{ \sqrt { n - t  }   } \sum_{i=t+1}^ n    x_i  ^\top     (  \beta^*_{ (0,  t ] }  - \beta^*_ { ( t , n] }  )      \epsilon_i  \bigg\} ,
  \\\label{eq:long run variance term 3}
  \sigma_{L,2} ^2 =& \lim_{n \to\infty }  Var \bigg\{ \frac{1}{ \sqrt {   t  }   } \sum_{i= 1}^ t    ( \beta^* _{ (0, t ] }  -     \beta _ { (t , n] } ^*  )  ^\top  (  \widehat \Sigma  _{ (0, t ] }   -\Sigma  )  (  \beta^*_{ (0, t ] }   -    \beta ^* _ { (t , n] }   )  \bigg\} ,
 \\\label{eq:long run variance term 4}
  \sigma_{R,2} ^2 =& \lim_{n \to\infty }  Var \bigg\{ \frac{1}{ \sqrt { n - t  }   } \sum_{i=t+1}^ n   ( \beta^* _{ (0, t ] }  -     \beta _ { (t , n] } ^*  )  ^\top  (  \widehat \Sigma  _{ (t , n ] }   -\Sigma  )  (  \beta^*_{ (0, t ] }   -    \beta ^* _ { (t , n] }   )  \bigg\} .
\end{align}
Note that by \Cref{lemma:bound for specific long-run variance}, all of $\sigma_{L,1}^2 ,\sigma_{R,1}^2,\sigma_{L,2}^2,\sigma_{R,2}^2$ are finite.

 \begin{proof}[Proof of  \Cref{theorem:alternative 2_beta}] Let  $r  \in [\zeta, 1-\zeta]  $  and  $t = \lfloor r n\rfloor  $.    Without loss of generality, assume that $\sigma_\epsilon=1  $. 
Following the same argument as \Cref{corollary:iid power},  \Cref{theorem:alternative}  and   \Cref{theorem:alternative counterpart} continue to hold under \Cref{assume:beta}. 
 Let    $\mathfrak T_{n, \xi} (t)$ and   $\mathfrak S_n    (t) $  be  defined  as in \Cref{theorem:alternative},  and 
 $\mathfrak T'_{n,\xi} (t)$ and   $ \mathfrak S'_n  (t) $ be  defined  as  in \Cref{theorem:alternative counterpart}. 
 Denote 
 $$ \mathfrak T_{n, \xi} (t  )  +\mathfrak S_n    (t  ) + \mathfrak T'_{n,\xi} ( t )  +  \mathfrak S'_n  ( t  ) =  \mathcal  M  (r)       , $$
 where 
\begin{align} \nonumber  &\mathcal M  (r)     
\\ = &\frac{4}{ t  } \sum_{i=1}^ t    x_i  ^\top     (  \beta^*_{ (0,  t ] }  - \beta^*_ { ( t , n] }  )      \epsilon_i  +\frac{4}{n- t } \sum_{i=t  +1}^ n     x_i  ^\top     (  \beta^*_ { (t , n] }    - \beta^*_{ (0, t ] }    )      \epsilon_i  
 \label{eq:beta mixing alternative term 1}
\\\label{eq:beta mixing alternative term 2}
  +&   \frac{2}{  t } \sum_{i=1}^  t     x_i  ^\top     (  \beta^*_i  - \beta^*_{ (0,  t ] }    )      \xi_i       +  \frac{2}{n-  t } \sum_{i=t  +1}^ n     x_i  ^\top     (  \beta^*_i  - \beta^*_ { (t , n] }     )      \xi_i       
  \\\label{eq:beta mixing alternative term 3}
   +&   \frac{2}{  t } \sum_{i=1}^ t    \xi_i\epsilon_i    - \frac{2}{n-   t   } \sum_{i= t  +1}^ n    \xi_i\epsilon_i      
\\\label{eq:beta mixing alternative term 4}
  + &   ( \beta^* _{ (0, t ] }  -     \beta _ { (t , n] } ^*  )  ^\top  (  \widehat \Sigma  _{ (0, t ] }   -\Sigma  )  (  \beta^*_{ (0, t ] }   -    \beta ^* _ { (t , n] }   )  +  ( \beta^*_ { (t , n] }    -     \beta _{ (0, t ] }  ^*  )  ^\top  (  \widehat \Sigma  _ { (t , n] }     -\Sigma  )  (  \beta^*_ { (t , n] }    -    \beta ^* _{ (0, t ] }    )   .
\end{align}    
  \
 \\
 { \bf Step 1.} Recall that as defined in \Cref{eq:qf_random}, with $t =\lfloor nr  \rfloor $,
\begin{align*} 
	\mathcal{S}_n(r)=&\frac{1}{2}(\widehat{\Delta}_t^\top \widehat{\Sigma}_{(0,t]}\widehat{\Delta}_t+\widehat{\Delta}_t^\top\widehat{\Sigma}_{(t,n]} \widehat{\Delta}_t)\nonumber\\
	+&\frac{1}{t}\sum_{i=1}^t(2\widehat{\Delta}_t^\top x_i+\xi_i)(y_i-x_i^\top\widehat{\beta}_{(0,t]})-\frac{1}{n-t}\sum_{i=t+1}^n(2\widehat{\Delta}_t^\top x_i+\xi_i) (y_i-x_i^\top\widehat{\beta}_{(t,n]}) .\end{align*} 
  Note that by assumption, 
\begin{align}\label{eq:lower bound of left}  \Phi_L (r )  \ge 4 \sigma_\xi^2\quad \text{and} \quad \Phi_R     (r )  \ge4 \sigma_\xi^2.
\end{align}
With  $t =\lfloor nr  \rfloor $,  denote  $	\mu (r)  =   ({\beta}_{(0,t]}^*-{\beta}_{(t,n]}^* )^\top   \Sigma ({\beta}_{(0,t]}^*-{\beta}_{(t,n]}^* ) $.
 We have that with probability goes to 1, 
 \begin{align*}
  & \bigg(  \sqrt {  \frac{ \Phi_L (r ) }{ n r  }   +  \frac{  \Phi_R (r ) } {n(1-r)   }  } \bigg)^{-1}      \bigg| \mathcal S_n  (r   ) -2 \mu(r)   -   \mathcal M   (r)       \bigg| 
  \\
 =&  \bigg(  \sqrt {  \frac{ \Phi_L(r ) }{  r }  +  \frac{  \Phi_R (r ) } {n-r   }  } \bigg)^{-1}      \bigg| \mathcal S_n   ( r ) -2 \mu (r)   -   \mathfrak T_{n, \xi} (t  )  -\mathfrak S_n    (t  ) -  \mathfrak T'_{n,\xi} (t )  -  \mathfrak S'_n  (t  ) \bigg| 
 \\
  \le 
  &  C (1+\sigma_\xi ) \frac{\s \log(pn)}{     n }   \bigg(  \sqrt {  \frac{ \Phi_L(r ) }{ n r   }    +   \frac{  \Phi_R(r )  } {n (1-r)   }  } \bigg)^{-1}  
  \\
   \le & C' (1+\sigma_\xi )   \frac{ \s \log(pn)  }{  n }   \bigg(  \sqrt {  \frac{  \sigma_\xi ^2  }{ nr    }   +   \frac{  \sigma_\xi ^2    } { n(1-r)      }  } \bigg)^{-1}  
  \\
  \le & C' (1+\sigma_\xi )   \frac{ \s \log(pn)  }{\sqrt n }   \bigg(  \sqrt {  \frac{  \sigma_\xi ^2   }{r   }   +   \frac{\sigma_\xi ^2    } {1- r  }  } \bigg)^{-1}  
  \\
  \le & C''   \frac{ (1+\sigma_\xi )  \s \log(pn) }{  \sigma_\xi  \sqrt { n }  },
 \end{align*}
 where the first inequality follows from  \Cref{theorem:alternative}  and   \Cref{theorem:alternative counterpart}, the second inequality follows from \Cref{eq:lower bound of left}.
By assumption, $ \frac{\s \log(pn)}{   \sqrt { n  }   } = o(1)$,
    $\sigma_\xi  = A_n \frac{ \s \log(pn)}{\sqrt n} $ for some diverging sequence $A_n$.  So $  \frac{ (1+\sigma_\xi )  \s \log(pn) }{ \sigma_\xi  \sqrt { n   }  } =o(1)$.  
  \\
  \\
Consequently it suffices to   show that 
$$ \bigg(  \sqrt {  \frac{ \Phi_L (r ) }{ n r  }   +  \frac{  \Phi_R (r ) } {n(1-r)   }  } \bigg)^{-1}        \mathcal M   (r)          $$  converges  to $N(0,1)$.  
\\
\\
{\bf Step 2.} Since $ \{ X_i\}_{i=1}^n,\{ \epsilon_i\}_{i=1}^n $ and $\{ \xi_i\}_{i=1}^n$ are independent, so \eqref{eq:beta mixing alternative term 1}, \eqref{eq:beta mixing alternative term 2}, \eqref{eq:beta mixing alternative term 3} and \eqref{eq:beta mixing alternative term 4} are pairwise uncorrelated. Therefore, it suffices to show that each of \eqref{eq:beta mixing alternative term 1}$\sim$\eqref{eq:beta mixing alternative term 4} are converging to normal random variables, as this would imply that \eqref{eq:beta mixing alternative term 1}$\sim$\eqref{eq:beta mixing alternative term 4} are asymptotically independent. For brevity,   the justifications of \eqref{eq:beta mixing alternative term 1} and \eqref{eq:beta mixing alternative term 2} are provided only, as the justification of \eqref{eq:beta mixing alternative term 3} and \eqref{eq:beta mixing alternative term 4} are either similar or simpler. 
\
\\
\\
{\bf Step 3.} For \eqref{eq:beta mixing alternative term 1},  note that each $ x_i  ^\top     (  \beta^*_{ (0,  t ] }  - \beta^*_ { ( t , n] }  )      \epsilon_i$ is sub-Weibull($2\gamma_2$) $\beta$-mixing time series with mixing coefficient $\exp(-cn^{\gamma_1})$.
So  by \Cref{theorem:beta mixing CLT}, 
$$\frac{1}{ \sqrt {t }  \sigma_{L, 1}  } \sum_{i=1}^ t    x_i  ^\top     (  \beta^*_{ (0,  t ] }  - \beta^*_ { ( t , n] }  )      \epsilon_i  \to N(0,1)\quad \text{and} \quad \frac{1}{\sqrt { (n- t) }\sigma_{R, 1}     } \sum_{i=t  +1}^ n     x_i  ^\top     (  \beta^*_ { (t , n] }    - \beta^*_{ (0, t ] }    )      \epsilon_i   \to N(0,1),  $$
where $ \sigma_{L, 1}^2$ are defined in \Cref{eq:long run variance term 1} and $\sigma_{L,2}^2$ are defined in \Cref{eq:long run variance term 2}.   
\\
\\
By \Cref{lemma:partial sum independent}, $\frac{1}{ \sqrt {t }    } \sum_{i=1}^ t    x_i  ^\top     (  \beta^*_{ (0,  t ] }  - \beta^*_ { ( t , n] }  )      \epsilon_i $  and $ \frac{1}{\sqrt { (n- t) }    } \sum_{i=t  +1}^ n     x_i  ^\top     (  \beta^*_ { (t , n] }    - \beta^*_{ (0, t ] }    )      \epsilon_i    $ are asymptotically independent, it follows that 
$$  \frac{1}{ \sqrt {t }    } \sum_{i=1}^ t    x_i  ^\top     (  \beta^*_{ (0,  t ] }  - \beta^*_ { ( t , n] }  )      \epsilon_i   +  \frac{1}{\sqrt { (n- t) }     } \sum_{i=t  +1}^ n     x_i  ^\top     (  \beta^*_ { (t , n] }    - \beta^*_{ (0, t ] }    )      \epsilon_i   \to N(0, \sigma_{L,1}^2 + \sigma^2_{L,2} ).  $$
\
\\
\\
{\bf Step 4.} For \eqref{eq:beta mixing alternative term 2}, let $\{ \eta_q ^* \}_{q=1}^Q   = \{\eta_{k} ^*\}_{k=1}^K \cap (0,r] $, where $Q =0$ indicates that $(0,r]$ contains no change points.  Denote $\eta_k =\lfloor n\eta_k^* \rfloor $ and  
$$ \mathcal J_1= (0, \eta_1], \  \mathcal J_2 = (\eta_1, \eta_2] \  \ldots  \  \mathcal J_Q= (\eta_{Q_1}, \eta_Q], \ \mathcal J_{Q+1} = (\eta_{Q}, t].$$
  Note that 
$
          x_i  ^\top     (  \beta^*_i  - \beta^*_{ (0, t] }    )      \xi_i    
$ and $ x_j  ^\top     (  \beta^*_j  - \beta^*_{ (0, t] }    )      \xi_j $  are uncorrelated because $\{ \xi_i\}_{i=1}^n $ are i.i.d. 
By \Cref{theorem:beta mixing CLT}, for  each $q\in \{ 1,\ldots, Q+1\} $
$$\frac{ 1}{\sqrt {|\mathcal J_q|  }\nu_q } \sum_{ i \in \mathcal J_q } x_i  ^\top     (  \beta^*_i  - \beta^*_{ (0, t] }    )      \xi_i  \to N(0,1) .$$
Here  $$ \nu^2_q   = \lim_{n\to \infty }  Var\bigg\{ \frac{ 1}{\sqrt {|\mathcal J_q|  }  } \sum_{ i \in \mathcal J_q } x_i  ^\top     (  \beta^*_i  - \beta^*_{ (0, t] }    )      \xi_i  \bigg\} = \lim_{n\to \infty }  \frac{4}{ |\mathcal J_q|   }   \sum_{i \in \mathcal J_q } \sigma_\xi^2  ( \beta_i^* -   \beta^*_{ (t , n]  }   ) ^\top \Sigma   ( \beta_i^* -  \beta^*_{ (t , n]  }   ) , $$
where the second  equality follows from the fact that $ \{  x_i  ^\top     (  \beta^*_i  - \beta^*_{ (0, t] }    )      \xi_i \}_{i=1}^n$ are pairwise uncorrelated.
\\
\\
  In addition,
by \Cref{lemma:partial sum independent},   the collection 
$$\bigg \{   \frac{ 1}{\sqrt {|\mathcal J_q|  }\nu_q } \sum_{ i \in \mathcal J_q } x_i  ^\top     (  \beta^*_i  - \beta^*_{ (0, t] }    )      \xi_i \bigg\}_{q=1}^{Q+1}$$
 are asymptotically independent.  Since $\bigcup_{q=1}^{Q+1} \mathcal J_q =(1, t] $,
 it follows that  
 \begin{align}\label{eq:beta mixing non-stationary term 1}\frac{ 1}{\sqrt {  t  }\nu_L } \sum_{i=1   }^t  x_i  ^\top     (  \beta^*_i  - \beta^*_{ (0, t] }    )      \xi_i  \to N(0,1), 
 \end{align}
 where 
 $   \nu^2_L= \frac{1}{t  }   \sum_{i=1}^t   \sigma_\xi^2  ( \beta_i^* -   \beta^*_{ (0, t ] }   ) ^\top \Sigma   ( \beta_i^* -   \beta^*_{ (0, t ] }   ) .  $
 Similarly
 \begin{align}\label{eq:beta mixing non-stationary term 2}\frac{ 1}{\sqrt {  n-t  }\nu_R } \sum_{i=t+1   }^n  x_i  ^\top     (  \beta^*_i  - \beta^*_{ (0, t] }    )      \xi_i  \to N(0,1), \end{align}
 where 
 $   \nu^2_R= \frac{ 1}{n-t  }   \sum_{i=t+1}^n   \sigma_\xi^2  ( \beta_i^* -   \beta^*_{ (0, t ] }   ) ^\top \Sigma   ( \beta_i^* -   \beta^*_{ (0, t ] }   ) .  $ Since  by \Cref{lemma:partial sum independent},   \eqref{eq:beta mixing non-stationary term 1}  and \eqref{eq:beta mixing non-stationary term 2} are asymptotically independent, it follows that 
 \begin{align*} \frac{ 1}{\sqrt {  t  }  } \sum_{i=1   }^t  x_i  ^\top     (  \beta^*_i  - \beta^*_{ (0, t] }    )     \xi_i +  \frac{ 1}{\sqrt {  n-t  }  } \sum_{i=1   }^t  x_i  ^\top     (  \beta^*_i  - \beta^*_{ (0, t] }    )      \xi_i   \to N(0,\nu_L^2 + \nu_R^2), 
 \end{align*} as desired.
\
\\
\\
{\bf Step 5.} Let  $r  \in [\zeta, 1-\zeta]  $  and  $t = \lfloor r n\rfloor  $.  To show that   $ \sqrt n  \mathcal {S}_n(r)$ converges to Gaussian process, by {\bf Step 1}  it suffices to show that  $\sqrt n    \mathcal M(r)$ converges to a Gaussian process. For any $N \in \mathbb Z^{+ }$, let $\{ r_\ell\}_{\ell=1}^N  \subset [\zeta, 1-\zeta ]$ be given.  Applying the Cramer-World rule and the same argument as in {\bf Step 1 }-{ \bf 4},   the joint distribution of the vector 
$$ \{ \sqrt { n} \mathcal M (r_1) , \ldots, \sqrt { n} \mathcal M (r_N) \} $$
converges to a joint normal law.  Therefore it suffices to show the asymptotic tightness of the process $\sqrt n \mathcal M (r)$. By Theorem 12.3 in \cite{billingsley2013convergence}, it suffices to show that for any $ r, r' \in  \subset [\zeta, 1-\zeta ]  $, 
$$\E \big[ ( \sqrt n \mathcal M(r) -\sqrt n \mathcal M(r')   )^4 \big]  \le C  |r-r'|^2  $$
for some absolute constant $C$ independent of $n$ and $p$. 
Denote 
\begin{align*}  
 \mathcal M_1(r)  = &\frac{4}{ t  } \sum_{i=1}^ t    x_i  ^\top     (  \beta^*_{ (0,  t ] }  - \beta^*_ { ( t , n] }  )      \epsilon_i  +\frac{4}{n- t } \sum_{i=t  +1}^ n     x_i  ^\top     (  \beta^*_ { (t , n] }    - \beta^*_{ (0, t ] }    )      \epsilon_i  
\\ 
   \mathcal M_2 (r) = &   \frac{2}{  t } \sum_{i=1}^  t     x_i  ^\top     (  \beta^*_i  - \beta^*_{ (0,  t ] }    )      \xi_i       +  \frac{2}{n-  t } \sum_{i=t  +1}^ n     x_i  ^\top     (  \beta^*_i  - \beta^*_ { (t , n] }     )      \xi_i       
  \\ \mathcal M_3(r)  = &
      \frac{2}{  t } \sum_{i=1}^ t    \xi_i\epsilon_i    - \frac{2}{n-   t   } \sum_{i= t  +1}^ n    \xi_i\epsilon_i      
\\ \mathcal M_4(r) =  & ( \beta^* _{ (0, t ] }  -     \beta _ { (t , n] } ^*  )  ^\top  (  \widehat \Sigma  _{ (0, t ] }   -\Sigma  )  (  \beta^*_{ (0, t ] }   -    \beta ^* _ { (t , n] }   )  +  ( \beta^*_ { (t , n] }    -     \beta _{ (0, t ] }  ^*  )  ^\top  (  \widehat \Sigma  _ { (t , n] }     -\Sigma  )  (  \beta^*_ { (t , n] }    -    \beta ^* _{ (0, t ] }    )   .
\end{align*}    
So
$$   \mathcal M(r) =  \sum_{j=1}^4 \sqrt n \mathcal M_j (r)   $$
and therefore 
$$\E \big[ ( \sqrt n \mathcal M(r) -\sqrt n \mathcal M(r')   )^4 \big]  \le 8\sum_{j=1}^p \E \big[ ( \sqrt n \mathcal M_j (r) -\sqrt n \mathcal M_j (r')   )^4 \big] . $$
The moment bound related  $\sqrt n \mathcal M_1$   is shown here,   as the other moment bounds can be shown in the same way.  
\
\\
\\
{\bf Step 6.} Denote  $t = \lfloor r n\rfloor  $ and $t' = \lfloor r' n\rfloor  $, in  this step, it is shown that 
\begin{align}
\label{eq:tightness term 1}  
   \E \bigg[  \bigg ( \frac{1}{ t  } \sum_{i=1}^ t    x_i  ^\top     (  \beta^*_{ (0,  t ] }  - \beta^*_ { ( t , n] }  )      \epsilon_i  -  \frac{1}{ t '  } \sum_{i=1}^ {t'}     x_i  ^\top     (  \beta^*_{ (0,  t' ] }  - \beta^*_ { ( t' , n] }  )      \epsilon_i    \bigg )^4 \bigg]  \le C_1 \frac{ |t-t' |^2}{ n^4 }. 
   \end{align} 
 Note that by symmetry, 
$$   
   \E \bigg[  \bigg ( \frac{1}{n- t } \sum_{i=t  +1}^ n     x_i  ^\top     (  \beta^*_ { (t , n] }    - \beta^*_{ (0, t ] }    )      \epsilon_i -\frac{1}{n- t' } \sum_{i=t'  +1}^ n     x_i  ^\top     (  \beta^*_ { (t' , n] }    - \beta^*_{ (0, t' ] }    )      \epsilon_i 
        \bigg )^4 \bigg]  \le C_2 \frac{ |t-t' |^2}{ n^4 } . $$
The above two inequalities imply that $\E \big[ ( \sqrt n \mathcal M_1 (r) -\sqrt n \mathcal M_1 (r')   )^4 \big] \le C_3 |r-r'|  $.
\\
\\
{\bf Step 7.} Without loss of generality assume that $t<t'$, where
$t = \lfloor r n\rfloor  $ and $t' = \lfloor r' n\rfloor  $.  To this end, observe that
\begin{align*}
 \frac{1}{ t  } \sum_{i=1}^ t    x_i  ^\top       \beta^*_{ (0,  t ] }     \epsilon_i  -  \frac{1}{ t '  } \sum_{i=1}^ {t'}     x_i  ^\top    \beta^*_{ (0,  t' ] }    \epsilon_i   
=   \bigg( \frac{1}{t}- \frac{1}{t'} \bigg) \sum_{i=1}^ t    x_i  ^\top       \beta^*_{ (0,  t ] }     \epsilon_i   + \frac{1}{t'} \sum_{i=t+1}^t  x_i  ^\top       \beta^*_{ (0,  t '  ] }\epsilon_i .
\end{align*}
Since the vector time series $\{(x_i,\epsilon_i) \}_{i=1}^n $ is beta-mixing with exponential decay rate  and both $ \{ x_i\}_{i=1}^n$ and $ \{ \epsilon_i\}_{i=1}^n  $ are sub-Weibull$(\gamma_2)$,  the time series 
$\{ x_i^\top \beta^*_{ (0,  t ] }     \epsilon_i \}_{i=1}^n$ is also beta-mixing with the same exponential decay rate and marginal  sub-Weibull$(2\gamma_2)$. So all the moments of $ x_i^\top \beta^*_{ (0,  t ] }     \epsilon_i$ exist  and are finite.   Therefore,
by Theorem B.5 in \cite{kirch2006resampling}
$$\E\bigg[ \bigg( \sum_{i=1}^{t }x_i  ^\top       \beta^*_{ (0,  t ] }     \epsilon_i \bigg)^4 \bigg]\le C_5 t^{2}\quad \text{and} \quad \E\bigg[ \bigg(\sum_{i=t+1}^t  x_i  ^\top       \beta^*_{ (0,  t '  ] }\epsilon_i       \bigg)^4 \bigg]  \le C_6 (t'-t)^2
  $$
for some absolute constant $C_5$ and $C_6$. Since $t , t' \in [\zeta n, (1-\zeta) n ] $, it follows that 
\begin{align*}
\E \bigg[  \bigg(\frac{1}{ t  } \sum_{i=1}^ t    x_i  ^\top       \beta^*_{ (0,  t ] }     \epsilon_i  -  \frac{1}{ t '  } \sum_{i=1}^ {t'}     x_i  ^\top    \beta^*_{ (0,  t' ] }    \epsilon_i    \bigg)^4  \bigg] \le C_5 t^2 \bigg( \frac{1}{t}- \frac{1}{t'} \bigg) ^4 + C_6 \frac{(t-t')^2}{ t'^4} \le C_7 \frac{(t-t')^2}{n^4}
\end{align*}
for some $C_7$ depending on $ C_5$, $C_6$ and $\zeta$. 
By the same argument,  there exists $C_8$ such that
$$\E \bigg[  \bigg ( \frac{1}{n- t } \sum_{i=t  +1}^ n     x_i  ^\top       \beta^*_{ (0, t ] }          \epsilon_i -\frac{1}{n- t' } \sum_{i=t'  +1}^ n     x_i  ^\top       \beta^*_{ (0, t' ] }          \epsilon_i 
        \bigg )^4 \bigg]  \le C_8 \frac{ |t-t' |^2}{ n^4 }.$$
      The above two moment bounds directly lead to \eqref{eq:tightness term 1}.
 
  \end{proof} 
   
\subsubsection{Additional Technical Results}   
\label{subsection:basic properties}
   Throughout \Cref{subsection:basic properties}, assume that  $ \{ z_i\}_{i=1}^n$ is a sequence of time series  satisfying the following additional conditions. 
\begin{assumption} \label{assume:beta 1D}
	\
	\\ 
	{\bf a.} 
	Suppose that the time series $\{ z_i\}_{i=1}^n  $
	is strictly stationary and geometrically $\beta$-mixing; i.e., there exist constant
	$c $ and $\gamma_1$ such that the $\beta$-coefficients of $\{ z_i\}_{i=1}^n $ satisfy 
	$$ \beta(n) \le \exp(-c n^{\gamma_1})\quad \text{for all} \quad 
	n \in \mathbb N. $$
	\
	\\
	{\bf b.} Each  s $ z_i $   follows a sub-Weibull$(\gamma_2)$ distribution with  $\|z_i\|_{\psi_{\gamma_2}} \le K_\epsilon$ for $ i \in \{ 1,\ldots, n\}. $
	\
	\\
	\\
	{\bf c.} It holds that  $ \E(z_i)=0 $ for all $i\in \{ 1,\ldots, n\}.$
\end{assumption} 
\begin{lemma} \label{lemma:bound for long-run variance}
Suppose $\{ z_i\}_{i=1}^n $ satisfies  \Cref{assume:beta 1D}. Then 
$$Var\bigg(  \frac{ 1}{\sqrt n }\sum_{i=1}^n z_i  \bigg )  \le \sum_{l=-\infty}^{  \infty }  | Cov(z_1, z_{t+l})|   < \infty  $$
\end{lemma}   
   \begin{proof}
   This is a well known property for  $\beta$-mixing time series. 
 Since $z_i$ is sub-Weibull, all the moments of $z_i$ exist  and are finite. Since the $\beta$-coefficient of $\{z_i\}_{i=1}^n$ decay exponentially, it is  faster than any polynomial decay. The desired result follows from Corollary A.2 of \cite{francq2019garch}.
   
   \end{proof}
   
\begin{theorem}[$\beta$-mixing CLT] \label{theorem:beta mixing CLT}
Suppose $\{ z_i\}_{i=1}^n $ satisfies  \Cref{assume:beta 1D}. Then
$$ \frac{1}{\sqrt { n  }  \sigma_{z}}\sum_{i=1}^n z_i \to N(0,1), $$
where 
 $ \sigma_z^2 =   Var\bigg(  \frac{ 1}{\sqrt n }\sum_{i=1}^n z_i  \bigg ) $.
\end{theorem}   
    \begin{proof}
   This is a well known property for  $\beta$-mixing time series. 
   The desired result follows  from Theorem  A.4 of \cite{francq2019garch}.
   
   \end{proof}

   \begin{lemma} \label{lemma:partial sum independent}
   Let $a\in (0,1) $. Then 
   $$  \lim_{n \to \infty }Cov \bigg(\frac{1}{\sqrt n } \sum_{i=1}^{ \lfloor a n \rfloor} z_i, \ 
   \frac{1}{\sqrt n } \sum_{j =  \lfloor a n \rfloor +1}^n  z_j    \bigg ) = 0.$$ 
   Therefore $\frac{1}{\sqrt n } \sum_{i=1}^{ \lfloor a n \rfloor} z_i$ and $\frac{1}{\sqrt n } \sum_{j =  \lfloor a n \rfloor +1}^n  z_j    $ are asymptotically independent. 
   \end{lemma}
   \begin{proof}
   Denote $  b_{ ij}= Cov(z_i, z_j)$ and 
   $$\sigma_{LV}^2  = \sum_{k=-\infty }^\infty |  Cov (z_1, z_k)|.$$  Then 
   \begin{align*}
    Cov \bigg(\frac{1}{\sqrt n } \sum_{i=1}^{ \lfloor a n \rfloor} z_i, \ 
   \frac{1}{\sqrt n } \sum_{j =  \lfloor a n \rfloor +1}^n  z_j    \bigg )  
    =    \frac{1}{n } \sum_{i=1}^{ \lfloor a n \rfloor -M } \sum_{j =  \lfloor a n \rfloor +1}^n Cov \big (   z_i,  
     z_j    \big  )   +\frac{1}{n } \sum_{i=  \lfloor a n \rfloor -M+1  }^{ \lfloor a n \rfloor  } \sum_{j =  \lfloor a n \rfloor +1}^n Cov \big (   z_i,  
     z_j    \big  )    
   \end{align*}
   Note that for $i \in \{ 1, \ldots, \lfloor a n \rfloor -M \}  $  
   $$ \sum_{j =  \lfloor a n \rfloor +1}^n |Cov \big (   z_i,  
     z_j    \big  ) |  \le \sum_{k=M}^\infty |Cov(z_1,z_k)|   .$$
   So for sufficiently large $M$ and $i \in \{ 1, \ldots, \lfloor a n \rfloor -M \}  $, 
   $$ \sum_{j =  \lfloor a n \rfloor +1}^n |Cov \big (   z_i,  
     z_j    \big  ) |  \le \epsilon. $$
     Therefore 
     $$  \frac{1}{n } \sum_{i=1}^{ \lfloor a n \rfloor -M } \sum_{j =  \lfloor a n \rfloor +1}^n Cov \big (   z_i,  
     z_j    \big  )  \le \frac{1}{n } \sum_{i=1}^{ \lfloor a n \rfloor -M }  \epsilon \le a \epsilon. $$
   In addition,
   \begin{align*}
    \frac{1}{n } \sum_{i=  \lfloor a n \rfloor -M+1  }^{ \lfloor a n \rfloor  } \sum_{j =  \lfloor a n \rfloor +1}^n Cov \big (   z_i,  
     z_j    \big  )     \le \frac{1}{n } \sum_{i=  \lfloor a n \rfloor -M+1  }^{ \lfloor a n \rfloor  } \sum_{j = -\infty }^{\infty } | Cov \big (   z_i,  
     z_j    \big  ) | 
     \le   \frac{1}{n } \sum_{i=  \lfloor a n \rfloor -M+1  }^{ \lfloor a n \rfloor  } \sigma_{LV}^2  \le \frac{M}{n}   \sigma_{LV}^2 .
   \end{align*} 
  So 
  \begin{align*}
    \frac{1}{n } \sum_{i=1}^{ \lfloor a n \rfloor   } \sum_{j =  \lfloor a n \rfloor +1}^n  | Cov \big (   z_i,  
     z_j    \big  )   |  \le a\epsilon + \frac{M}{n}   \sigma_{LV}^2 .
   \end{align*}
   Therefore \begin{align*}
   \lim_{n\to \infty } \frac{1}{n } \sum_{i=1}^{ \lfloor a n \rfloor   } \sum_{j =  \lfloor a n \rfloor +1}^n  | Cov \big (   z_i,  
     z_j    \big  )   |  \le a\epsilon   .
   \end{align*} 
   Since $\epsilon$ is arbitrary, the desired result follows. 
   \end{proof}

   \begin{lemma}  \label{lemma:bound for specific long-run variance}The long-run variances $ \sigma_{L,1}^2 ,\sigma_{R,1}^2,\sigma_{L,2}^2,\sigma_{R,2}^2$ defined in 
\eqref{eq:long run variance term 1}$\sim$\eqref{eq:long run variance term 4} are all finite.
   \end{lemma}
   \begin{proof}
   The argument for each long-run variance  is the same. For brevity, only $\sigma_{L,1}^2$ is shown in detailed. Observe that   $ \{ x_i  ^\top     (  \beta^*_{ (0,  t ] }  - \beta^*_ { ( t , n] }  )      \epsilon_i \}_{i=1}^n $ are sub-Weibull($2\gamma_2$) $\beta$-mixing time series with mixing coefficient $\exp(-cn^{\gamma_1})$. The desired result follows immediately from \Cref{lemma:bound for long-run variance}.
   \end{proof}

\subsection{Technical results for \Cref{theorem:main_beta} }
\label{subsection: additional beta}
\begin{lemma} 
 \label{corollary: beta deviation 1 with change points}  Suppose  $\I \subset (0,n]$ is any  generic interval such that 
	$$	\frac{ \max\{ \log^{1/\gamma}(p),  ( \s\log p )^{ \frac{2}{\gamma}- 1}  \} }{|\I| } \to 0.$$
Under the same assumptions as in \Cref{theorem:main_beta}, then it holds that for sufficiently large constant $C$,
\begin{align}
\label{eq:res eigen beta} 
&\mathbb P \bigg\{ \big|v^\top \big( \widehat \Sigma_{\I } -\Sigma \big) v  \big |\ge C  \sqrt { \frac{\s \log(pn) }{ |\mathcal I|  }} \|v\|_2^2  \   \forall  \  v \in \mathcal C_S \bigg\}\le 2\exp\big\{-c_1 \log(pn) \big\}   ,\quad \text{and}
\\\label{eq:dev beta} 
&\mathbb P \bigg\{  \bigg | \frac{1}{| \mathcal I | } \sum_{i\in \I }  \epsilon_i x_i^\top \beta  \bigg|
\ge C  \sqrt { \frac{\log(pn )}{ |\mathcal I|   }}\| \beta\|_1  \ \forall  \  \beta \in \mathbb R^p   \bigg\} \le 
2 np\exp\big\{-c_2 n^\gamma  \big\}  ,
\end{align}
where $ c_1 ,c_2 \ge 3$ are positive constants. 
\end{lemma}
\begin{proof}
The deviation bound in \Cref{eq:res eigen beta} is a straight forward  adaption of Proposition  8 in Lasso Guarantees for    $\beta$-mixing heavy-tailed time series  by \cite{Wong2020}, and the deviation bound in \Cref{eq:dev beta} is Proposition  7 \cite{Wong2020}.  
\end{proof}

\begin{lemma} \label{lemma:interval lasso consistency beta mixing}
Suppose  $\I \subset (0,n]$ is any  generic interval such that 
	$$	\frac{ \max\{ \log^{1/\gamma}(p),  ( \s\log p )^{ \frac{2}{\gamma}- 1}  \} }{|\I| } \to 0.$$
     Suppose in addition that the interval $\I$ does not contain any change points.    Let $\widehat \beta_\I$ be the Lasso estimator defined in \eqref{eq:interval lasso}. 
   There exists $\lambda = C_\lambda \sqrt { \log p }$ with some sufficiently large constant $C_\lambda $, such that  with probability $1-n^{-3}$,
	\begin{align*}
	 \| \widehat \beta_\I -\beta^*_\I \| _2^2 \le \frac{C \s\log p}{n }, ~ \| \widehat \beta_\I  -\beta^*_\I \| _1  \le     C \s\sqrt {  \frac{\log p} {n } }, ~ \| (\widehat \beta_\I -\beta^*_\I)_{S_\I^c} \| _1 \le  3 \| (\widehat \beta _\I  -\beta^*_\I )_{S_\I } \| _1   .
	\end{align*}
where $C$ is an absolute constant and $S^c_\I=\{1,2,\cdots,p\} \setminus S_\I$ with $S_\I$ being the support set of $\beta_\I^*.$
\end{lemma}  
\begin{proof}
Using the deviation bounds in  \Cref{corollary: beta deviation 1 with change points}, the proof of \Cref{lemma:interval lasso consistency beta mixing} follow from the same argument as in the proof  \Cref{lemma:interval lasso}.

\end{proof}

\begin{lemma}\label{lemma: beta deviation 2} Let $ \zeta \in (0, 1/2)$ and $ u\in \mathbb R^p   $ be any   deterministic vector.
Under the same assumptions as in \Cref{theorem:main_beta}, it holds that 
 \begin{align*}
 &  \mathbb P\bigg\{ \bigg|  \frac{1}{|\I | } \sum_{i \in \I }   \big\{   u^\top   x_i   x_i^\top \beta  -  u^\top  \Sigma \beta \big\}  \bigg|   \ge C     \|u \|_2   \sqrt {  \frac{\log(pn)}{ | \I |  } } \|\beta\|_1   \ \forall \   \beta \in \mathbb R^p,   \ | \I| \ge \zeta n   \bigg\}
  \\
   \le
   & C'   \exp\bigg (-c  (n \log(pn) )^{\gamma/2} \bigg) +   \exp \bigg(- c ' \log(pn)  \bigg)  ,
  \end{align*}
  where $ C, C',c,c'$ are all absolute constants and that $c, c'\ge 3 $.
\end{lemma}

\begin{proof}
Without loss of generality, assume   $\| u\|_2=1 .$   For any given $ j\in \{ 1,\ldots, p\}$,
let 
$$z_i = u^\top x_ix_{i,j}      -   u^\top  \Sigma (, j)   = u^\top x_ix_ {i,j}     -  \E ( u^\top x_i x_ {i, j}   )  .$$
As a result,  $z_i$ is strictly stationary.
Since $z_i$ is measurable to $\sigma(x_i)$, the process $\{ z_i\}_{i=1}^t$ is $\beta$-mixing with the same mixing coefficient as $\{x_i \}_{i=1}^n $.
In addition, 
$$ \| z_i\|_{\psi _{\gamma_2/2 } }  =\| u^\top x_i   x_ {i,j}\|_{\psi _{\gamma_2/2 } } 
\le c_1 \|u^\top x_i  \|_{ \psi_{ \gamma_2}}  \|  x_{i, j}   \|_{ \psi_{ \gamma_2}} \le c_1 K_X^2,  $$
where $c_1$ is a constant only depending on $\gamma_2 $.   Therefore by \Cref{theorem:beta deviation},
\begin{align}\label{eq:deviation 2 beta one vector}  \mathbb P \bigg\{ \frac{1}{|\I| } \big|   \sum_{i\in \I }z_i \big| \ge \delta \bigg\} 
\le C_2 n \exp\big (-c_2 (\delta n)^\gamma\big) + \exp(- c_3\delta^2n), 
\end{align} 
where 
$ \gamma  < 1  $ is defined in \Cref{assume:order_betamixing}. Consequently, taking an union bound   overall coordinates, it holds that 
\begin{align}\nonumber  & \mathbb P \bigg\{  \max_{1\le j \le  p }\bigg|  \frac{1}{|\I | } \sum_{i \in \I }   \big \{   u^\top    x_i   x_ {i,j}      -      u^\top  \Sigma (, j)  \big \}   \bigg | \ge C_1 \delta    \ \forall \|u\|_2=1     \bigg\}  
\\
\le& C_2 n p\exp\big (-c_2 (\delta n)^\gamma\big) +p \exp(- c_3\delta^2n) \nonumber 
\end{align}  
Taking a union bound over all $\I $ such that $|\I|\ge \zeta n$, it holds that 
\begin{align*} & \mathbb P \bigg\{  \max_{1\le j \le  p }\bigg|  \frac{1}{|\I | } \sum_{i \in \I }   \big \{   u^\top    x_i    x_ { i,j}     -      u^\top  \Sigma (, j)  \big \}   \bigg | \ge \delta  \text{ for all }   \I \subset (0, n] \text{ such that } |\I| \ge \zeta n    \bigg\}  
\\
\le& C_2 n^3p \exp\big (-c_2 (\delta n)^\gamma\big) +n^2 p  \exp(- c_3\delta^2n).
\end{align*}
\
\\
Taking $\delta = C_1 \sqrt { \frac{\log(pn) }{ | \I| }} $, it follows that 
\begin{align*} & \mathbb P \bigg\{  \max_{1\le j \le  p }\bigg|  \frac{1}{|\I | } \sum_{i \in \I }   \big \{   u^\top    x_i    x_ { i,j}      -      u^\top  \Sigma (, j)  \big \}   \bigg | \ge C_1 \sqrt { \frac{\log(pn) }{ | \I| }}   \text{ for all }   \I \subset (0, n] \text{ such that } |\I| \ge \zeta n    \bigg\}  
\\
\le& C_2   \exp\bigg (-c_2'  (n \log(pn) )^{\gamma/2} \bigg) +   \exp \bigg(- c_3' \log(pn)  \bigg),
\end{align*} 
where $ \frac{\log^{ \frac{2}{\gamma } -1 } (np)}{ n} =o(1) $ is used in the last inequality. 
The desired result follows from the observation that for all $\beta\in \mathbb R^p$, 
\begin{align*}
\bigg|  \frac{1}{|\I | } \sum_{i \in \I }     u^\top   x_i   x_i^\top \beta  - \frac{1}{|\I | } \sum_{i \in \I }      u^\top  \Sigma \beta  \bigg|  
\le \max_{1\le j \le  p }\bigg|  \frac{1}{|\I | } \sum_{i \in \I }   \big \{   u^\top   x_i   x_i  (j)   -      u^\top  \Sigma (, j)  \big \}   \bigg | \|\beta\|_1. 
\end{align*}
\end{proof}

\begin{lemma}\label{lemma: beta deviation 3}
Let $ \zeta \in (0, 1/2)$ and $ u\in \mathbb R^p   $ be any   deterministic vector.
Under the same assumptions as in \Cref{theorem:main_beta}, it holds that 
 \begin{align*}
 &  \mathbb P\bigg\{ \bigg|  \frac{1}{|\I | } \sum_{i \in \I }   \big\{   u^\top   x_i   x_i^\top \beta  -  u^\top  \Sigma \beta \big\}  \bigg|   \ge C     \sqrt {  s \frac{\log(pn)}{ | \I |  } } \|\beta\|_1    \ \forall \   \beta \in \mathbb R^p,  \ \forall  | \I| \ge \zeta n, \ \forall   \|u \|_2 =1, u\in \mathcal C_S   \bigg\}
  \\
   \le
   &C   \exp\big (-c  ( \s \log(pn)  n)^{\gamma/2} \big) +  C '   \exp(- c  '  \s \log(pn))  .  
  \end{align*} 
  where $ C, C',c,c'$ are all absolute constants and that $c, c'\ge 3 $.

\end{lemma}
 \begin{proof} 
  { \bf Step 1.}
By \Cref{eq:deviation 2 beta one vector}, it holds that for any $u \in \mathbb R ^p $  such that $ \|u\|_2=1$, 
 \begin{align*}   \mathbb P \bigg\{  \bigg|  \frac{1}{|\I | } \sum_{i \in \I }   \big\{   u^\top    x_i    x_ { i,j}         -  u^\top  \Sigma (, j)  \big\}   \bigg| \ge \delta  \bigg\} 
\le C_2 n \exp\big (-c_2 (\delta n)^\gamma\big) + \exp(- c_3\delta^2n), 
\end{align*} 
where 
$ \gamma = \bigg( \frac{1}{\gamma_1} + \frac{2}{\gamma_2}\bigg) ^{-1}  $  as defined in \Cref{assume:order_betamixing}.
Denote 
$$ \mathcal K (s) = \{ u \in \mathbb R^p:  \|u\|_2= 1, \|u\|_0 \le s  \}.$$
\
\\
By Lemma F1 in the appendix of  BasuMichailidis(2015),
$$ \mathcal C_S \cap \{ u \in \mathbb R^p : \|u\|_2=1 \} \subset 5cl \{ conv\{ \mathcal K(s) \},$$
where  $cl \{ conv\{ \mathcal K(s) \} $ is the closure of  convex hull of $\mathcal K(s)$. 
So
\begin{align*}  & \mathbb P \bigg\{   \bigg|  \frac{1}{|\I | } \sum_{i \in \I }   \big\{   u^\top    x_i    x_ { i,j}        -  u^\top  \Sigma (, j)  \big\}   \bigg| \ge \delta \  \forall  \ \| u\|_2=1, u \in \mathcal C_S  \bigg\}
\\ 
\le &  \mathbb P \bigg\{   \sup_{ u \in  5cl\{ cov\{ \mathcal K( s) \} \} } \bigg|  \frac{1}{|\I | } \sum_{i \in \I }   \big\{   u^\top    x_i    x_ { i,j}       -  u^\top  \Sigma (, j)  \big\}   \bigg| \ge \delta     \bigg\} 
\end{align*}  
\
\\
By \Cref{lemma: approximation of convex hull}, with $ V=  \frac{1}{|\I| } \sum_{i\in \I }x_i   x_i^\top (j)    -     \Sigma (, j)  $, 
it holds that 
\begin{align*}   
&  \mathbb P \bigg\{   \sup_{ u \in  5cl\{ cov\{ \mathcal K( s) \} \} } \bigg|  \frac{1}{|\I | } \sum_{i \in \I }   \big\{   u^\top    x_i    x_ { i,j}        -  u^\top  \Sigma (, j)  \big\}   \bigg| \ge \delta    \bigg\} 
\\   \le &
  \mathbb P \bigg\{   \sup_{ u \in  5 \mathcal K(s)   } \bigg|  \frac{1}{|\I | } \sum_{i \in \I }   \big\{   u^\top    x_i    x_ { i,j}         -  u^\top  \Sigma (, j)  \big\}   \bigg| \ge \delta    \bigg\}  .
\end{align*}   
\\
By \Cref{lemma:sparse discretization}, it holds that
\begin{align*} \mathbb P \bigg\{   \sup_{ u \in   \mathcal K( s) }\bigg|  \frac{1}{|\I | } \sum_{i \in \I }   \big\{   u^\top    x_i    x_ { i,j}         -  u^\top  \Sigma (, j)  \big\}   \bigg| \ge \delta  \bigg\} 
\le C_2 n p^s 5^s  \exp\big (-c_2 (\delta n)^\gamma\big) + p^s 5^s  \exp(- c_3\delta^2n)  . 
\end{align*} 
As a result 
\begin{align*}  \mathbb P \bigg\{   \sup_{ u \in  5  \mathcal K( s) }\bigg|  \frac{1}{|\I | } \sum_{i \in \I }   \big\{   u^\top    x_i    x_ { i,j}         -  u^\top  \Sigma (, j)  \big\}   \bigg| \ge 5 \delta  \bigg\} 
\le C_2 n p^s 5^s  \exp\big (-c_2 (\delta n)^\gamma\big) + p^s 5^s  \exp(- c_3\delta^2n)  . 
\end{align*} 
Combining the calculations in this step, it holds that 
\begin{align} \nonumber 
&  \mathbb P \bigg\{   \bigg|  \frac{1}{|\I | } \sum_{i \in \I }   \big\{   u^\top    x_i  x_ { i,j}        -  u^\top  \Sigma (, j)  \big\}   \bigg| \ge   \delta  \quad -  \forall  \ \| u\|_2=1, u \in \mathcal C_S  \bigg\} 
\\ \label{eq:condition 2 beta step 1} 
 \le & C_2 n p^s 5^s  \exp\big (-c_2 (\delta n)^\gamma\big) + p^s 5^s  \exp(- c_3\delta^2n)  . 
\end{align}  
 \
 \\
{ \bf Step 2.} Since 
 \begin{align*}
\bigg|  \frac{1}{|\I | } \sum_{i \in \I }     u^\top   x_i   x_i^\top \beta  - \frac{1}{|\I | } \sum_{i \in \I }      u^\top  \Sigma \beta  \bigg|  
\le \max_{1\le j \le  p }\bigg|  \frac{1}{|\I | } \sum_{i \in \I }   \big \{   u^\top    x_i    x_ { i,j}       -      u^\top  \Sigma (, j)  \big \}   \bigg | \|\beta\|_1,
\end{align*} 
 it suffices to bound 
  \begin{align*}
 &  \mathbb P\bigg\{  \max_{1\le j \le p } \bigg|  \frac{1}{|\I | } \sum_{i \in \I }   \big\{   u^\top    x_i    x_ { i,j}         -  u^\top  \Sigma (, j)  \big\}  \bigg|   \ge C     \sqrt {  \frac{\log(pn)}{ | \I |  } }      \ \forall  | \I| \ge \zeta n, \ \forall   \|u \|_2 =1, u\in \mathcal C_S   \bigg\}. 
  \end{align*}  
  Since there are at most $n^2$ many intervals $\I  $ such that $|\I   |\ge \zeta n$,   by \Cref{eq:condition 2 beta step 1}, 
   \begin{align*}
 &  \mathbb P\bigg\{  \max_{1\le j \le p } \bigg|  \frac{1}{|\I | } \sum_{i \in \I }   \big\{   u^\top   x_i   x_i^\top (j)    -  u^\top  \Sigma (, j)  \big\}  \bigg|   \ge  \delta      \ \forall  | \I| \ge \zeta n, \ \forall   \|u \|_2 =1, u\in \mathcal C_S   \bigg\}
 \\
 \le &C_2 n^3 p^{\s+1} 5^\s  \exp\big (-c_2 (\delta n)^\gamma\big) + n ^2  p^{\s+1} 5^\s  \exp(- c_3\delta^2n)  .  
  \end{align*}   
  \
  \\
Taking 
$\delta =C_3 \sqrt { \frac{\s \log(pn) }{ n }}   $ for sufficiently large constant $C_3$ and using the assumption that 
$ \frac{( \s\log( np) )^{ \frac{2}{\gamma}- 1}   }{n} =o(1) $, it holds that 
 \begin{align*}
&  \mathbb P \bigg\{   \sup_{ u \in \mathcal K( s) }\bigg|  \frac{1}{|\I | } \sum_{i \in \I }   \big\{   u^\top   x_i   x_i^\top (j)    -  u^\top  \Sigma (, j)  \big\}   \bigg| \ge 5C_3 \sqrt { \frac{\s \log(pn) }{ n }}   \  \forall  \ \| u\|_2=1, u \in \mathcal C_S ,  | \I | \ge \zeta n  \bigg\} 
\\
 \le & C_2'   \exp\big (-c_2' ( \s \log(pn)  n)^{\gamma/2} \big) +  C_3'   \exp(- c_3 '  \s \log(pn))  . 
\end{align*} 
 \end{proof}

\begin{lemma}\label{lemma:beta to gaussian process}
Suppose $ \{ \xi_i\}_{i=1}^n$ is a collection of i.i.d. Gaussian random variables $N(0,\sigma^2_\xi)$   and $\{ \epsilon_i\}_{i=1}^n $  satisfies \Cref{assume:beta} with 
 $Var(\epsilon_i)=\sigma_\epsilon^2$. 
For any $r\in (0,1) $, let $t =\lfloor r n\rfloor   $ and 
$$\mathcal G_n (r ) =\frac{ \sqrt { n  }}{\sigma_\xi \sigma_\epsilon} \bigg\{  \frac{1}{t}\sum_{i=1}^t \epsilon_i\xi_i - \frac{1}{ n-t} \sum_{i=t+1}^{n } \epsilon_i\xi_i  \bigg\}   .$$ 
Then  
$$ \mathcal G_n(r ) \overset{\mathcal D } {\to}  \mathcal G (r ), $$
where the convergence is in Skorokhod topology  and $\mathcal G(r ) $ is a  Gaussian Process on $r\in (0,1)$,  with covariance function 
$$\sigma(r,s) =  \frac{1  }{ r(1-s)} \quad \text{when}  \quad   0< s\le r <1   .$$

\end{lemma}
\begin{proof} This is a direct consequence of \Cref{theorem:beta mixing CLT} and the observation that   
$$ Cov( \epsilon_i\xi_i,\epsilon_j\xi_j ) = \E (\epsilon_i\xi_i \epsilon_j\xi_j)  = \E (\epsilon_i\epsilon_j)\E(\xi_i)\E(\xi_j ) = 0. $$

\end{proof}

 \begin{lemma} \label{lemma:beta interval lasso} 
Suppose all the assumptions in \Cref{theorem:main_beta} holds and that 
$$\beta^*_1 =\ldots=\beta^*_n =\beta^* .$$ 
 Let $ \zeta \in (0,1/2)$ be  a   given constant. 
For any generic interval $\I$,   let $\widehat \beta_ \I $ be defined as in  \Cref{eq:interval lasso}.
   If   $\lambda = C_\lambda \sqrt { \log(pn) }$ for some sufficiently large constant $C _ \lambda $, then 
it holds that 
\begin{align*}
&  \mathbb P\bigg\{   \| \widehat \beta_\I -\beta^*  \| _2^2 \le \frac{C\s \log(pn)}{n }   \ \forall \I \subset (0,n] , |\I| \ge \zeta n \bigg\} \ge 1- (pn)^{-3} ;
 \\
 &   \mathbb P\bigg\{   \| \widehat \beta_\I  -\beta^*  \| _1  \le     C s\sqrt {  \frac{\log(pn)} {n } }  \ \forall \I \subset (0,n] , |\I| \ge \zeta n \bigg\}\ge 1-  (pn)^{-3} ;
 \\
 &  \mathbb P\bigg\{    \| (\widehat \beta_\I -\beta^* )_{S^c} \| _1 \le  3 \| (\widehat \beta_\I  -\beta^*)_{S } \| _1  \ \forall \I \subset (0,n] , |\I| \ge \zeta n \bigg\}\ge 1- (pn)^{-3} . 
\end{align*} 
\begin{proof}Note that the number of possible interval $\I$ is bounded  by $n^2$. 
As a result, the proof is a simple  consequence of \Cref{lemma:interval lasso consistency beta mixing}   together with a   union bound over all possible intervals $\I$. 
\end{proof}
 \end{lemma}

\subsection{Additional technical results}

\begin{theorem} \label{theorem:beta deviation}
Let $\{ x_j\}_{j=1}^n  $ be a strictly stationary sequence of zero mean random variables that are subweibull$(\gamma_2)$ with subweibull constant $K$.  Suppose  the $\beta$-mixing coefficients of $\{ x_j\}_{j=1}^n  $ satisfy 
$ \beta(m) \le \exp(-c m^{\gamma_1})$. Let  $\gamma$ be given by 
$$ \frac{1}{\gamma}  =\frac{1}{\gamma_1} + \frac{2}{\gamma_2}$$
and suppose that $ \gamma<1$. Then for $n\ge 4$  and any  $\delta > 1/n $, 
$$ \mathbb P \bigg\{  \bigg| \frac{1}{n} \sum_{i=1}^n x_i  \bigg|  \ge \delta  \bigg\} \le n \exp\bigg\{ 
-\frac{ ( \delta n )^\gamma }{   K^ \gamma C_1 } 
\bigg\}  
+ \exp\bigg\{ -\frac{\delta^2 n}{ K ^2C_2}\bigg\}, $$
where the constants $C_1, C_2 $ depend only on $\gamma_1,\gamma_2$ and $c$. 
\end{theorem}
\begin{proof}
This is Lemma 13 of \cite{Wong2020}.
\end{proof}

 \begin{lemma} \label{lemma:sparse discretization}
Let $V \in \mathbb R^p  $ be any random vector.  Suppose for all $u \in \mathbb R^p$ such that $\|u\|_2=1$, it holds that
$$ \mathbb P \bigg\{  | u^\top V| \ge  \delta  \bigg\} \le  g(\delta) . $$
 Then 
 $$ \mathbb P \bigg\{  \sup_{ \|u \|_2 =1 , \|u\|_0\le \s } | u^\top V| \ge  \delta  \bigg\} \le  p^\s 5^\s g(\delta/2).   $$
 \end{lemma}
\begin{proof}This is a standard covering lemma. 
Choose $ S\subset \{ 1,\ldots, p\}$ with $ |S| =\s $. Define 
$$\Omega_{S} : =\{ u \in \mathbb  R^p: \|u\|_2=1 , \ \text{support}(u) \subset S  \}.$$
Then we have 
$$ \{u \in \mathbb R^p : \|u\|_2=1, \|u\|_0\le \s \} = \bigcup_{|S| = \s } \Omega_S .$$
{\bf Step 1.} Let $S$ be given such that $|S|\le \s$. 
Let $\{ w_i\}_{i=1}^M $ be $1/2$-net of $\Omega_S$. That is for any $u\in \Omega_S $, there exists 
$w_j $ such that $\|w_j-u \|_2 \le 1/2$.  By standard covering result,  $\{ w_i\}_{i=1}^M $ can be chosen so that 
$M\le 5^\s. $
\\
\\
Denote 
$$ \alpha =\sup_{ u \in \Omega_S   } | u^\top V| \ . $$
Since 
$ \text{support}(w_{j}-u ) \subset S$  and $\|w_j-u \|_2 \le 1/2$,  
$$| (w_j -u) ^\top V| \le \alpha /2  .$$
So we have for any $u \in \Omega_S $
$$ |u^\top V| \le \max _{1\le i \le  M } | w_i ^\top V| + | (w_j -u) ^\top V|  \le \max _{1\le i \le  M } | w_i ^\top V|  + \frac{\alpha}{2}.$$
This implies that 
$$\alpha \le \max _{1\le i \le  M } | w_i ^\top V|  + \frac{\alpha}{2}  $$
or simply 
$$ \alpha \le 2  \max _{1\le i \le  M } | w_i ^\top V|  .$$
As a result 
$$ \mathbb P \bigg\{ \sup_{u \in \Omega_S }  | u^\top V| \ge  \delta  \bigg\} \le 
\mathbb P \bigg\{ \max_{1 \le i \le M    }  | w_i ^\top V| \ge   \delta /2   \bigg\} \le 5^\s  g(\delta/2).  $$
\
\\
{\bf Step 2.} Since there are at most ${ p \choose \s} \le p^\s $ number of possible choices of $S$  such that $|S|\le \s$, it holds that 
$$ \mathbb P \bigg\{  \sup_{ \|u \|_2 =1 , \|u\|_0\le \s } | u^\top V| \ge  \delta  \bigg\} \le  p^\s 5^\s g(\delta/2 ) .$$
\end{proof}

\begin{lemma} \label{lemma: approximation of convex hull}Denote 
$$ \mathcal K (\s) = \{ u \in \mathbb R^p:  \|u\|_2= 1, \|u\|_0 \le \s  \}.$$ 
For any vector  $V \in \mathbb R^p$, it holds that 
$$ \sup_{u  \in cl \{ conv (\mathcal K(\s))\}} |u^\top V | \le   \sup_{u \in \mathcal K( \s)} |u^\top V| . $$ 
\end{lemma}
\begin{proof}
By continuity of the linear form, it suffices to show that 
$$ \sup_{u  \in   conv (\mathcal K(\s)) } |u^\top V | \le  \sup_{u \in \mathcal K( \s)} |u^\top V| . $$ 
Let $v \in conv\{ \mathcal K(\s) \}$. Then $v = \sum_{l=1}^L \alpha_l v_l  $ where  $L\ge 1$, $v_ l \in \mathcal K(\s)$ and 
$0\le \alpha_ l \le 1  $ such that $\sum_{ l =1}^n\alpha_ l  =1 .  $
Then 
\begin{align*}
|v^\top V |  \le \sum_{l=1}^L \alpha_l   | v_l^\top V|  \le \sup_{u \in \mathcal K( \s)} |u^\top V|,
\end{align*}
which directly gives the desired result. 
\end{proof}

\section{Lower bounds for detection boundary  }
\textbf{Proof of \Cref{lemma:lower bound for detection}}: 
\begin{proof} To establish the lower bound, it suffices to consider the special case for $K=1$,   $x_i  \overset{i.i.d}{\sim}  N_p(0, I_p) $ and $\epsilon_i \overset{i.i.d}{\sim}   N(0,1) $. This is because the i.i.d. Gaussian conditions are a special case of the sub-Weibull beta-mixing conditions. The generic settings   will only result in better constants.  
\\
\\
{\bf Step 1.} Denote 
\[
		\mathcal{C } = \left\{ \beta  \in \mathbb{R}^p: \, \beta (i)   \in \{  \kappa/\sqrt{\s},  0  , -\kappa/\sqrt {\s}   \}, \ \| \beta  \|_0 =\s \right\}.
	\] 
Let $\widetilde P_{ \beta } $ denote the joint distribution of  $\{ y_i, x_i\}_{i=1}^n $ generated according to \Cref{assume: model assumption}, where 
$$\beta^*_t = \begin{cases} 
0  & \text{when } 1\le t \le n/2;
\\
\beta & \text{when }  (n/2+ 1)\le t \le n. 
\end{cases} $$
So under $\widetilde P_{ \beta } $, there is only one change point $n/2$ in the time series and 
$$ \|\beta^*_{n/2} -\beta^*_{n/2+1}\|_2^2 = 
 \|  \beta\|_2^2  = \kappa^2. 
$$
where $\Sigma =I_p$ is used.  Let 
$$\kappa ^2= b \frac{s\log(p)}{n} $$ for sufficiently small constant $b$. So by assumption, $\kappa^2<1/8$.
\\
\\Let $ \widetilde P_0 $ denote the distribution of data generated according to \Cref{assume: model assumption} with $ \beta_ i ^* = 0 $ for all $i$. 
 Let    $$\widetilde P_1 = \frac{1}{|\mathcal C |  } \sum_{\beta \in  \mathcal C} \widetilde  P_\beta ,$$
where $|\mathcal C| $ denotes the cardinality of $\mathcal C$. Denote $\E_\beta  $ to be the expectation under the distribution of $\widetilde P_ \beta  $.
Note that  
\begin{align*}
 \inf _{\psi }\sup_{P    \in \mathcal P_1 (b)} \E _ { 0}  (  \psi) + \E _ {P } (1-    \psi)  
 \ge &  \inf _{\psi }\sup_{ \beta    }\E _ { 0}  (  \psi) + \E _{\beta }    (1-    \psi)  
 \\
 \ge &  \inf _{\psi } \E _ 0 (  \psi) +  \frac{1}{|\mathcal C |  } \sum_{\beta \in  \mathcal C}    \E _ \beta   (1-    \psi)   
 \\
 \geq &1 - \frac{1}{2} \| \widetilde P_0 -\widetilde P_1\|_1 ,
\end{align*}
where the last inequality is due to LeCam's Lemma. 
Let $P_0 ^{n/2 }$ denote the joint distribution of $\{ y_i, x_i\}_{i=1}^{n/2 } $ with $\beta^*_i=0$ for all $1\le i \le n/2$.
In addition, let  $P_\beta^{n/2 }  $ denote the joint distribution of $\{ y_i, x_i\}_{i=1}^{n/2 } $ with $\beta^*_i=\beta $ for all $1\le i \le n/2$ and   let 
$$   P_1 ^{n/2 } = \frac{1}{|\mathcal C |  } \sum_{\beta \in  \mathcal C}   P_\beta^{n/2 }   .$$
Since under  $ \widetilde P_0$ or $ \widetilde P_1 $, the distribution of $\{ y_i   ,  X_ i   \}_{i =1}^{n/2} $ are the same,
straightforward calculations show that 
$$  \| \widetilde P_0 -\widetilde P_1\|_1 =  \|   P_0^{n/2 }  - P_1^{n/2 } \|_1.$$
\
\\
In addition, by Le Cam's lemma,   
 $$
  \inf _{\psi }\sup_{P   \in \mathcal P_1 (b)} \E _ { 0 }  (  \psi) +\E _ {P} (1-    \psi)   \ge 1 -  \|  P_0^{n/2 }  -  P_1^{n/2 } \|_{TV} 
 $$   
where  
\begin{align*}
 \|  P_0^{n/2 }  -  P_1^{n/2 } \|_{TV}     \le  \frac{1}{2}\sqrt { \E_{P^{n/2}  _0 }  \bigg (  \frac{d P_{ 1} ^{n/2} }{ d P_0^{n/2}} \bigg) ^2    - 1}.
\end{align*}  
The relationship between different norms of distributions can be found in \cite{tsybakov2009lower} 
\\
\\
So it suffices to show 
\begin{align} \E_{P^{n/2}  _0 }  \bigg (  \frac{d P_{ 1} ^{n/2} }{ d P_0^{n/2}} \bigg) ^2   =1+o(1). \label{eq:lower bound target}
\end{align}
\
\\
{\bf Step 2.}
 Note that 
 $ \frac{dP_\beta }{ dP_0 } =\exp\bigg\{ y x^\top\beta  - \frac{1}{2} (x^\top \beta)^2  \bigg\} ,$ where $x\sim N(0,I_p) $. 
	 So 
	 \begin{align*}
	 \E_{P_0} \left( \frac{dP_\beta }{ dP_0 }\frac{dP_{\beta'}  }{ dP_0 }  \right) = \E _X (\exp\{ \beta^\top xx^\top \beta'  \}  ). 
	 \end{align*}
	 Straightforward calculations show that 
 \begin{align*}
\E_{P^{n/2}  _0 }   \bigg (  \frac{d P_{ 1} ^{n/2} }{ d P_0^{n/2}}  \bigg) ^2  =& \E_{\beta,\beta'\sim \mathcal C}  \Bigg\{   \Bigg[ \E_X\Bigg(
\exp\bigg\{ \beta^\top xx^\top \beta ' \bigg\}  \Bigg)\Bigg] ^{n/2}  \Bigg\}
\\
=& \E_{\beta,\beta'\sim \mathcal C}   \Bigg\{   \Bigg( 1- \bigg [    2  \beta^\top\beta'   + \kappa^4 -   (   \beta^\top\beta'  )^2  \bigg]  \Bigg) ^{-n/4}  \Bigg\} ,
\\
\le& \E_{\beta,\beta'\sim \mathcal C}   \Bigg\{   \Bigg( 1- \bigg [    2  \beta^\top\beta'   + \kappa^4   \bigg]  \Bigg) ^{-n/4}  \Bigg\} ,
	 \end{align*} 
	 where $\beta \sim \mathcal C$ indicates that $\beta$ is selected from $\mathcal C$ uniformly at random. 
Since $\log(1-x)^{-1}\le x+x^2 $, and 
	 $$ 2   \beta^\top\beta'      + \kappa^4 \le 2 \kappa^2 +\kappa^4 \le 2 \frac{1}{8} + \frac{1}{64} < 1/2,$$
	  it holds that 
	 \begin{align*}
	& \E  _{\beta,\beta'\sim \mathcal C}  \Bigg\{   \Bigg( 1- \bigg [    2      \beta^\top\beta'      + \kappa^4   \bigg]  \Bigg) ^{-n/4}  \Bigg\} 
\\	
	= & \E _{\beta,\beta'\sim \mathcal C}  \Bigg\{    \exp \bigg(  \frac{n}{4}  \log \frac{1}{ 1-  2    \beta^\top\beta'        - \kappa^4  } \bigg )    \Bigg\} 
	 \\
	 \le &  \E _{\beta,\beta'\sim \mathcal C}  \Bigg\{    \exp  \frac{n}{4} \bigg(    2    \beta^\top\beta'        +  \kappa^4  + 
\bigg[ 2   \beta^\top\beta'        +  \kappa^4   \bigg  ]^2 	 
	  \bigg )    \Bigg\}  
	  \\
	  \le &  \E _{\beta,\beta'\sim \mathcal C}   \Bigg\{   \exp \frac{n}{4}  \bigg(   4   \beta^\top\beta'        + 2 \kappa^4  
	  \bigg )    \Bigg\}   
	  \\
	    = & \exp \bigg( \frac{n \kappa^4}{2} \bigg) \E _{\beta,\beta'\sim \mathcal C}   \Bigg\{    \exp   \bigg(    n  \beta^\top\beta'        
	  \bigg )    \Bigg\}   ,
	 \end{align*} 
	 where $ 4 \beta^\top\beta'      + 2 \kappa^4   <1 $ is used in the last inequality.
	 Let $J$ denote the support of $\beta$,  $J'$ denote the support of $\beta'$ and 
	 $$U =: |\{j \in \mathcal J \cap \mathcal J' : \beta_j =\beta_j'  \} | . $$
	 As a result 
	 $$\beta^\top \beta'  = 2 \frac{\kappa^2}{s } U -  \frac{\kappa^2}{s } |\mathcal J \cap \mathcal J' | . $$
	 So 
	$$  \E _{\beta,\beta'\sim \mathcal C}   \Bigg\{    \exp   \bigg(     n   \beta^\top\beta'        
	  \bigg )    \Bigg\}  = \E_{\mathcal J \cap \mathcal J' } \Bigg\{   \exp   \bigg[  n  \bigg(   2 \frac{\kappa^2}{s } U -  \frac{\kappa^2}{s } |\mathcal J \cap \mathcal J' |  \bigg) \bigg]      \Bigg \} .$$
	  Conditioning on $|\mathcal J \cap \mathcal J' |  $, $U\sim Bin (|\mathcal J \cap \mathcal J' | , 1/2) $. 
	  So 
	  \begin{align*}
	  & \E_{\mathcal J \cap \mathcal J' } \Bigg\{   \exp \bigg[    n  \bigg(   2 \frac{\kappa^2}{s } U -  \frac{\kappa^2}{s } |\mathcal J \cap \mathcal J' |  \bigg)   \bigg ]   \Bigg \}  
	  \\
	   = & \E _{\mathcal J \cap \mathcal J' }  E_{U |  {\mathcal J \cap \mathcal J' }  } \Bigg\{   \exp \bigg [ n  \bigg(   2 \frac{\kappa^2}{s } U -  \frac{\kappa^2}{s } |\mathcal J \cap \mathcal J' |  \bigg)  \bigg ]    \Bigg \}   
	   \\
	   = &  \E _{\mathcal J \cap \mathcal J' }    \Bigg\{   \bigg( 1- \frac{1}{2} + \frac{1}{2} e^{2n \frac{\kappa^2}{s }  }\bigg) ^{|\mathcal J \cap \mathcal J' |  } \exp   \bigg(    - n \frac{\kappa^2}{s } |\mathcal J \cap \mathcal J' |  \bigg)     \Bigg \}    
	   \\
	   = & \E _{\mathcal J \cap \mathcal J' }    \Bigg\{  \cosh \bigg(  \frac{ n \kappa^2}{s } \bigg )^{|\mathcal J \cap \mathcal J' |  }      \Bigg \}     
	   \\
	   \le &  \E _{\mathcal J \cap \mathcal J' } \Bigg\{  \exp  \bigg(  \frac{n  \kappa^2}{s } |\mathcal J \cap \mathcal J' | \bigg )      \Bigg \}      ,
	  \end{align*}
	  where $\cosh(x) \le \exp(x)$ for all $x$ is used in the last inequality. 
	  Putting this together, it holds that 
	  \begin{align}\label{eq:lower bound target 2}
	  \E_{P^{n/2}  _0 }   \bigg (  \frac{d P_{ 1} ^{n/2} }{ d P_0^{n/2}}  \bigg) ^2 \le \exp \bigg( \frac{n \kappa^4}{2} \bigg) \E _{\mathcal J \cap \mathcal J' } \Bigg\{  \exp  \bigg(  \frac{ n  \kappa^2}{s } |\mathcal J \cap \mathcal J' | \bigg )      \Bigg \}  
	  \end{align}

\	  
\\
{\bf Step 3.}	  Let $H(p,s,s)   $ denote the hypergeometric distribution counting the number of black balls in $s$ draws  from an urn containing $s$ black balls out of $p$ balls. Then by \Cref{eq:lower bound target 2},
\begin{align*} 
	  \E_{P^{n/2}  _0 }   \bigg (  \frac{d P_{ 1} ^{n/2} }{ d P_0^{n/2}}  \bigg) ^2 \le \exp \bigg( \frac{n \kappa^4}{2} \bigg) \E  \Bigg\{  \exp  \bigg(  \frac{ n  \kappa^2}{s }  H(p,s,s)  \bigg )      \Bigg \}  
	  \end{align*}
	 By  Lemma 3 in the supplement of \cite{arias2011global}, $H(p,s,s)$ is stochastically smaller than $Bin(s, s/(p-s) )$. 
	 So if $B\sim  Bin(s, s/(p-s))$, then 
\begin{align*} \E_{P^{n/2}  _0 }   \bigg (  \frac{d P_{ 1} ^{n/2} }{ d P_0^{n/2}}  \bigg) ^2
\le &  \exp \bigg( \frac{n \kappa^4}{2} \bigg) \E  \Bigg\{  \exp  \bigg(  \frac{ n  \kappa^2}{s } B \bigg )      \Bigg \}   
	  \\
	  = & \exp \bigg( \frac{n\kappa^ 4}{2}\bigg ) \Bigg\{ 1- \frac{s }{p-s } +   \frac{s }{p-s }  \exp( \frac {n\kappa^2 }{s})\Bigg\}  ^s
	  \\
	  \le &  \exp \bigg( \frac{n\kappa^ 4}{2}\bigg ) \Bigg\{ 1 +   \frac{s }{p-s }  \exp( \frac {n\kappa^2 }{s})\Bigg\}  ^s
	  \\
	  \le &\exp \bigg( \frac{n\kappa^ 4}{2} +  \frac{s ^2  }{p-s }  \exp( \frac {n\kappa^2 }{s}) \bigg ).
\end{align*} 
Since $s\log(p) =o(\sqrt n)  $, $s\le p ^{\alpha}  $ for some $\alpha <1/2$, and  $\kappa ^2= b \frac{s\log(p)}{n} $  for sufficiently small constant $b$, $\lim_{n,p\to \infty }\frac{n\kappa^ 4}{2}  = 0 $ and $ \lim_{n,p\to \infty } \frac{s ^2  }{p-s }  \exp( \frac {n\kappa^2 }{s}) =0$. This directly implies that  \Cref{eq:lower bound target} holds. 
	 \end{proof}

\section{Proofs in  \Cref{sec:adaptive}}

  \subsection*{Estimating the Variance}
\begin{lemma} \label{lemma:estimation of variance}
Suppose  $\I \subset (0,n]$ is any  generic interval such that  $|\I|\ge N $ and 
	$$	\frac{ \max\{ \log^{1/\gamma}(p),  ( \s\log p )^{ \frac{2}{\gamma}- 1}  \} }{N } \to 0.$$
     Suppose in addition that the interval $\I$ does not contain any change points.    Let $\widehat \beta_\I$ be the Lasso estimator defined in \eqref{eq:interval lasso} with   $\lambda = C_\lambda \sqrt { \log p }$ for  some sufficiently large constant $C_\lambda $. Then 
   $$ \p\bigg(  \bigg|  \frac{1}{|\I| } \sum_{i\in \I } (y_i-x_i^\top \widehat \beta _\I )^2  -  \sigma_\epsilon^  2 \bigg | \ge \delta + \frac{C_1 \s\log(p)  }{ N  }   \bigg)  
   \le 2  N \exp \bigg\{ - c_2  (\delta N ) ^\gamma  \bigg\}+ \exp \bigg\{ - c_3  \delta ^2 N    \bigg\}  +N ^{-5}. $$
 
\end{lemma}
\begin{proof}   Since 
$\I$ contains no change points, 
$\beta_i^* = \beta^*_ \I $ for all $i\in \I$. 
By by 
   \Cref{theorem:beta deviation}, $$ \p\bigg(  \bigg|  \frac{1}{|\I| } \sum_{i\in \I } \epsilon_i^2  -  \sigma_\epsilon^  2 \bigg | \ge \delta  \bigg)  
   \le 2 N \exp \bigg\{ - c_2 (\delta N ) ^\gamma  \bigg\}  + \exp \bigg\{ - c_3  \delta ^2 N    \bigg\}  .  $$ 
  Therefore, it suffices to  control 
$$   \bigg| \frac{1}{|\I| } \sum_{i\in \I } \epsilon_i^2  -  \frac{1}{|\I| } \sum_{i\in \I } (y_i-x_i^\top \widehat \beta _\I )^2      \bigg|  = \bigg| \frac{1}{|\I| } \sum_{i\in \I } (y_i-x_i^\top   \beta  ^*   _\I  )^2   -  \frac{1}{|\I| } \sum_{i\in \I } (y_i-x_i^\top \widehat \beta _\I )^2      \bigg|      . $$ 
To this end, note that 
  Note that 
\begin{align}  \nonumber 
   \bigg|   \frac{1}{ |\I|  }  \sum_{  i      \in \I } (y_ i      - X_i   ^\top \widehat \beta_\I   )^2  -  \frac{1}{ |\I|  }  \sum_{ i    \in \I } (y_ i      - X_i  ^\top \beta^* _\I  )^2  \bigg|   
 = &\bigg|  \frac{1}{ |\I|  }  \sum_{ i \in \I } \big\{  X_i  ^\top  ( \widehat \beta_\I  -   \beta^* _\I  )   \big\} ^2 -  2  \frac{1}{ |\I|  }  \sum_{i\in \I }  \epsilon_i  X_i ^\top ( \widehat \beta_\I   -  \beta^*_\I   ) \bigg| 
 \\ \label{eq:regression change point beta mixing deviation bound term 1}
 \le &    \frac{1}{ |\I|  }  \sum_{ i  \in \I } \big\{  X_ i ^\top  ( \widehat \beta_\I  -    \beta^*_\I  )   \big\}   ^2 
 \\   \label{eq:regression change point beta mixing deviation bound term 3}
 +& 2 \bigg|  \frac{1}{ |\I|  }     \sum_{i  \in \I }  \epsilon_ i X_i^\top ( \widehat \beta_\I   -  \beta^*_\I ) \bigg|  .
\end{align}
\
\\
Suppose all the good events in \Cref{lemma:interval lasso consistency beta mixing} holds. 
\\
\\
{\bf Step 1.} For \Cref{eq:regression change point beta mixing deviation bound term 1}, by \Cref{lemma:interval lasso consistency beta mixing},   $\widehat \beta_\I  -    \beta^*_\I  $ satisfies the cone condition that 
$$  | (\widehat \beta_\I  -    \beta^*_\I  )_{S^c} |_1 \le 3  | (\widehat \beta_\I  -    \beta^*_\I )_S   |_1   .$$
It follows from \Cref{corollary: beta deviation 1 with change points} that with probability at least $ 1-N^{-5}$,
\begin{align*} 
 \bigg| \frac{1}{|\I| } \sum_{ i \in \I } \big\{  x_i^\top  ( \widehat \beta_\I  -    \beta^*_\I  )   \big\} ^2    - ( \widehat \beta_\I  -    \beta^*_\I    )^\top \Sigma  ( \widehat \beta_\I  -    \beta^*_\I  )  \bigg|  \le C_4  \sqrt { \frac{\s \log(p ) }{|\I| }} \| \widehat \beta_\I  -    \beta^*_\I     \|_2 ^2  .
\end{align*}
The above display   gives
\begin{align*}    \bigg| \frac{1}{|\I| } \sum_{i  \in \I } \big\{   x_i^\top  ( \widehat \beta_\I  -    \beta^*_\I  )    \big\}^2 \bigg| \le &  \Lambda_{\max}(\Sigma)   \| \widehat \beta_\I  -    \beta^*_\I     \|_2 ^2   + C_4  \sqrt { \frac{\s \log(p ) }{|\I| }} \| \widehat \beta_\I  -    \beta^*_\I     \|_2 ^2 
\\
\le & \Lambda_{\max}(\Sigma)  \| \widehat \beta_\I  -    \beta^*_\I     \|_2 ^2 + C_ 4  \sqrt { \frac{\s \log(p ) }{C_\zeta\s \log(p )  }} \| \widehat \beta_\I  -    \beta^*_\I     \|_2 ^2   
\\
\le &  \frac{ C_5\s\log(p )}{|\I| }  , \end{align*} 
where the second inequality follows from the assumption that  $|\I| \ge    ( \s \log(p ) )^{2/\gamma -1 } \ge  \s \log(p ) $ when $ \gamma\ge 1$ and the last inequality follows from  \Cref{lemma:interval lasso consistency beta mixing}.
This gives 
$$    \bigg|  \frac{1}{|\I| } \sum_{ i \in \I } \big\{  x_i ^\top  ( \widehat \beta_\I  -    \beta^*_\I  )    \big\}^2 \bigg| \le  \frac{2 C_5\s\log(p) }{|\I| }. $$
\
\\
{\bf Step 2.}
For 
\Cref{eq:regression change point beta mixing deviation bound term 3}, note that
\begin{align*} 
  \bigg|  \frac{1}{ |\I|  }   \sum_{i  \in \I }  \epsilon_ i x _i^\top ( \widehat \beta_\I   -  \beta^*_\I ) \bigg| 
   \le    C_3 \sqrt { \frac{\log(p )}{|\I| }} \| \widehat \beta_\I   -  \beta^*_\I \|_1 
  \le  C_4 \sqrt { \frac{\log(p )}{|\I| }} \s \sqrt { \frac{\log(p )}{|\I| }}  = \frac{C_6 \s \log(p )}{|\I| } 
\end{align*} 
where both the first and the second inequality follow  from \Cref{lemma:interval lasso consistency beta mixing}.

\end{proof}

\begin{proof}[Proof of \Cref{lemma:consistent estimate of variance under dependence}] Let $ m \in \{ 1 ,\ldots,   \lfloor \log(n) \rfloor\} $ such that $\mathcal J_m$ contains no change point. Since 
$|\mathcal J_m| \ge n /  \log(n) $, by \Cref{lemma:estimation of variance}, it follows that 
 $$\p\bigg(  \big |   \sigma_\epsilon^2   -  \widehat  \sigma_ m ^  2 \big  | \ge C_1 \bigg\{  \frac{\log ^{1+ 1/\gamma }(n) }{n  }  + \frac{ \s\log(p)\log(n)   }{ n  }   + \frac{ \log(n) }{\sqrt n } \bigg \}   \bigg)
   \le 3 n  ^{-3}.$$
   Since there at most $\log(n)  $ such  intervals,  by union bound 
  \begin{align*} &\p\bigg(  \big |   \sigma_\epsilon^2   -  \widehat  \sigma_ m ^  2 \big  | \ge C_1 \bigg\{  \frac{\log ^{1+ 1/\gamma }(n) }{\sqrt n  }  + \frac{ \s\log(p)\log(n)   }{ n  }   \bigg \}  \ \text{for all $m$ such that $\mathcal J_m$ contains no change point }   \bigg)
  \\
   &\le 3 n  ^{-3}\log(n) \le 3n^{-2}.
   \end{align*}
   Since there at most $K$ intervals in $\{\mathcal J_m \}_{m=1}^{  \lfloor \log(n) \rfloor}  $ containing  change points,  and $ 2K \le 2 \log(n) < n $ for sufficiently large $n$,  by \Cref{lemma:pigeon-hole} it follows that 
 $$\p\bigg(  \big |   \sigma_\epsilon^2   -  \widehat  \sigma_ \epsilon ^  2 \big  | \ge C_1 \bigg\{  \frac{\log ^{1+ /\gamma }(n) }{ \sqrt n  }  + \frac{ \s\log(p)\log(n)   }{ n  }   \bigg \}    \bigg)
   \le 3 n  ^{-2}.$$  
 
\end{proof}

\begin{lemma}\label{lemma:pigeon-hole} Let $N,  K\in \mathbb Z^+$ and $N  >   2K$. 
Suppose $\{ a_{ m } \}_{m=1}^N \subset \mathbb R $  is such that 
 there exist  at least $N- K$ elements of $\{ a_{ m } \}_{m=1}^N $ such that 
 $$ |a_{ m }  -b^*  | \le \delta. $$ 
 Let $ \widehat  b  $ denote the median of $\{ a_{ m } \}_{m=1}^N$. Then 
 $$ |b-b^*| \le \delta. $$
\end{lemma}
\begin{proof} Let $\{ a_{(m)} \}_{m=1}^N$ be the sorted set of $\{ a_{ m } \}_{m=1}^N$.
There are at least $N- K$ elements of $\{ a_{ m } \}_{m=1}^N $ such that 
 $$  a_{ m }   \ge  b^* -\delta . $$ 
 So there are at most $K$ elements of $\{ a_{ m } \}_{m=1}^N $ such that 
 $$ a_{m} < b^* -\delta. $$
 Since $\{ a_{(m)} \}_{m=1}^N $ is sorted, it follows that 
 $$ a_{( K +1)} \ge  b^* -\delta .$$
Therefore for all $m \ge K +1 $,
 $$ a_{(m)}\ge  b^* -\delta .$$
 Since $ N > 2K  \ge K+1  $, it follows that  the median $\widehat b$ of $\{ a_{ m } \}_{m=1}^N$ satisfies 
 $$  \widehat b  \ge b^*-\delta.$$
 The other side of the inequality follows from a symmetric argument. 
\end{proof}

\subsection*{Estimating the sparsity parameter}

\begin{lemma} \label{lemma:consistent sparsity estimate no change points}
  Suppose  $\I \subset (0,n]$ is any  generic interval such that  $|\I|\ge  n/\log(n)  $ and 
	\begin{align} \label{eq:consistent sparsity estimate no change points} 
	\frac{\s\log (p)  \log(n)  }{\sqrt n } \to 0.
	\end{align}
     Denote  
     $$\widehat S_{\I  }   =  \bigg\{j:   |\widehat  \beta _{\I }  (j) | \ge  n^{-1/4} \bigg \}  . $$
    $\bullet$ Suppose   that the interval $\I$ does not contain any change points.   Then 
     $$\p( |\widehat S_{\I } | = |S_{\I }|  ) \ge 1-n^{-2}. $$
     $\bullet$ Suppose   that the interval $\I$   contain $K$  change points.  
     Then 
     $$\p(  |\widehat S_{\I } |  \le  (K+1)\s ) \ge 1-n^{-2}. $$
\end{lemma}
\begin{proof} 
Suppose all the good events in \Cref{lemma:interval lasso consistency beta mixing} holds.  
\\
\\
{\bf Step 1.} Suppose   that the interval $\I$ does not contain any change points.   Then 
$$\|\widehat \beta_\I - \beta^*_\I \|_\infty \le\|\widehat \beta_\I - \beta^*_\I \| _2 \le C_1\sqrt { \frac{\s \log(p) }{ |\I| }} \le C_1 \sqrt { \frac{\s \log(p)\log(n)  }{ n }}  . $$
Thus by \Cref{assume:recovery signal size}, if  $j \not \in S_\I$, then 
\begin{align*}| \widehat \beta_\I (j)   | =  | \widehat \beta_\I (j)  -\beta^* _\I (j)  | \le 
     C_1 \sqrt { \frac{\s \log(p)\log(n)  }{ n }} = o(n^{-1/4}),
\end{align*}
where the last inequality follows from  the assumption   
$\frac{ \s\log(p) \log(n) }{\sqrt n }\to 0 $.
\
\\
In addition, if  $j  \in S_\I$, then 
\begin{align*}| \widehat \beta_\I (j)   | \ge   | \beta^* _\I (j)  | -  
     C_1 \sqrt { \frac{\s \log(p)\log(n)  }{ n }} \ge  n^{-\theta } -  C_1 \sqrt { \frac{\s \log(p)\log(n)  }{ n }} \ge \frac{ n^{-\theta} }{2} > n^{-1/4} ,
\end{align*}
where the third inequality follows from the assumptions that $ \theta <1/4$ and that $\frac{ \s\log(p) \log(n) }{\sqrt n }\to 0 $. Therefore under the good event in \Cref{lemma:interval lasso consistency beta mixing},
$ \widehat S_\I = S_\I$.
\\
\\
{\bf Step 2.} Suppose   that the interval $\I$   contains $K$ change points.   Let $S_\I =\bigcup_{ i\in \I} S_i  $. Then $|S_\I| \le (K+1)\s.  $ Thus by \Cref{assume:recovery signal size}, if  $j \not \in S_\I$, then 
\begin{align*}| \widehat \beta_\I (j)   | =  | \widehat \beta_\I (j)  -\beta^* _\I (j)  | \le 
     C_1 \sqrt { \frac{\s \log(p)\log(n)  }{ n }} = o(n^{-1/4}),
\end{align*}
where the last inequality follows from  the assumption    that
$\frac{ \s\log(p) \log(n) }{\sqrt n }\to 0 $. 
This shows that $|\widehat S_\I| \le | S_\I | \le (K+1)\s.  $

\end{proof}

\begin{proof}[Proof of \Cref{lemma:consistent estimate of sparisty under dependence}]
Observe that under $H_0$,   \Cref{lemma:consistent sparsity estimate no change points} implies that $| \widehat S_{\mathcal J_m}| =  |S_{\mathcal J_m} | =|S_1| $. Thus $\widehat \s =|S_1|$.   
\\
\\
Under $H_a$, observe that  each interval  in $\{\mathcal J_m\}_{m=1}^{\lfloor \log(n)\rfloor} $ contains at most $K$ change  points, where $K \le \log(n) $.  So $|\s_m| \le \log(n) \s $ for all $1\le m \le \lfloor \log(n)\rfloor$. Therefore $\widehat \s \le  \log(n) \s $. On the other hand,   there are $\lfloor \log(n)\rfloor - K $ many intervals  in $\{\mathcal J_m\}_{m=1}^{\lfloor \log(n)\rfloor} $  containing    change points. So  there are 
$\lfloor \log(n)\rfloor - K > \frac{1}{2}\lfloor \log(n)\rfloor  $ many intervals in  $\{\mathcal J_m\}_{m=1}^{\lfloor \log(n)\rfloor} $ contains no change points,  and thus by the first part   of \Cref{lemma:consistent sparsity estimate no change points},  there are at least 
$  \frac{1}{2}\lfloor \log(n)\rfloor  $ many $ m \in\{1,\ldots,  \lfloor \log(n)\rfloor\}$ satisfying $| \widehat S_{\mathcal J_m}|=|   S_{\mathcal J_m}| \ge \s $. Thus the median $\widehat s  $ of $\{| \widehat |S_{\mathcal J _m} \}_{m=1}^{\lfloor \log(n)\rfloor}  $ satisfies
 $ \widehat \s\ge \s.  $
\end{proof}

\subsection*{The adaptive procedure}

\begin{proof}[Proof of \Cref{theorem:practical main_beta}]
Under $H_0$, \Cref{lemma:consistent estimate of variance under dependence} gives   $ \widehat \sigma_\epsilon \overset{\p}{\to } \sigma_\epsilon$ and  \Cref{lemma:consistent estimate of sparisty under dependence} implies that with high probability $\widehat \s =|S_1| =\s$.  So   the same argument   for  $H_0$ in 
\Cref{theorem:main_beta}  continues to hold.  Under $H_a$, \Cref{lemma:consistent estimate of variance under dependence} implies that $ \widehat \sigma_\epsilon \overset{\p}{\to } \sigma_\epsilon$ and  \Cref{lemma:consistent estimate of sparisty under dependence} implies that with high probability $\widehat \s \le  2\s$. Therefore under $H_a$, with high  probability
 $  \widehat   \sigma_\xi \le 2 A_n\frac{\widehat \s \log(p)}{ \sqrt n }   $  and  so   the same argument   for  $H_a$ in 
\Cref{theorem:main_beta}  continues to hold.
\end{proof}

\section{Proofs in \Cref{sec:localization}}

 \begin{proof}[Proof of \Cref{theorem:localization bound}  ]     Without loss of generality, assume that $\sigma_\epsilon=1  $. 
Let $\widetilde S_n(t)$ be defined as in \Cref{eq:localization qf_random},
\begin{align*} \nonumber   \widetilde { \mathcal M}  (t)      = &\frac{4}{ t  } \sum_{i=1}^ t    x_i  ^\top     (  \beta^*_{ (0,  t ] }  - \beta^*_ { ( t , n] }  )      \epsilon_i  +\frac{4}{n- t } \sum_{i=t  +1}^ n     x_i  ^\top     (  \beta^*_ { (t , n] }    - \beta^*_{ (0, t ] }    )      \epsilon_i  
\\
  + &   ( \beta^* _{ (0, t ] }  -     \beta _ { (t , n] } ^*  )  ^\top  (  \widehat \Sigma  _{ (0, t ] }   -\Sigma  )  (  \beta^*_{ (0, t ] }   -    \beta ^* _ { (t , n] }   )  +  ( \beta^*_ { (t , n] }    -     \beta _{ (0, t ] }  ^*  )  ^\top  (  \widehat \Sigma  _ { (t , n] }     -\Sigma  )  (  \beta^*_ { (t , n] }    -    \beta ^* _{ (0, t ] }    ) , \text{ and}
  \\
	\mu(t) =    &   ({\beta}_{(0,t]}^*-{\beta}_{(t,n]}^* )^\top   \Sigma ({\beta}_{(0,t]}^*-{\beta}_{(t,n]}^* ) .
\end{align*}    
  By the same argument as in  { \bf Step 1 } as in the proof  of \Cref{theorem:alternative 2_beta}
uniformly with 
     $t \in \{ \lfloor \zeta n \rfloor  , \ldots  \lfloor (1- \zeta) n \rfloor \}  $,   
 we have that with probability at least $1-n^{-3}$, 
 \begin{align*}
     \bigg|   \widetilde  {\mathcal S} _n  ( t    ) -2 \mu(t )   -    \widetilde { \mathcal M }   (t )       \bigg| 
  \le C_1\frac{\s \log(p)}{ n } .
 \end{align*}
So with 
     $t \in \{ \lfloor \zeta n \rfloor  , \ldots  \lfloor (1- \zeta) n \rfloor \}  $,   
 \begin{align*}
    \bigg\{  \frac{(n-t  )t  }{n} \bigg\}  \bigg|   \widetilde  {\mathcal S} _n  ( t    ) -2 \mu(t )   -    \widetilde { \mathcal M }   (t )       \bigg| 
  \le C_1  \s \log(p) .
 \end{align*}
   By definition of $\he $,
   $$    \bigg\{  \frac{(n-\eta )\eta }{n} \bigg\} \ws _n(\eta) \le \bigg\{  \frac{(n-\he  )\he }{n} \bigg\} \ws_n(\he)   . $$
   The above inequalities imply that 
   \begin{align}\label{eq:localization basic inequality}   \bigg\{  \frac{(n-\eta )\eta }{n} \bigg\} \bigg( 2 \mu(\eta) + \wm(\eta) \bigg)  \le  \bigg\{  \frac{(n-\he )\he }{n} \bigg\} \bigg( 2 \mu(\he ) + \wm (\he ) \bigg)     
 +    C_1  \s\log(p)   .
   \end{align}
   Without loss of generality, assume that $\he \ge \eta $. For sufficiently large $C_2$, if 
    $ \he \le \eta + C_2 \frac{\s\log(p)}{\kappa^2 } $ 
   then the desired result immediately follows. So it suffices to assume that 
    \begin{align}\label{eq:sample condition localziation}     \he -\eta  \ge  C_2 \frac{\s\log(p)}{\kappa^2 } \ge C_2' \s\log(p), 
    \end{align}
    where the last inequality follows from $\kappa=O(1)$. 
\
\\
{\bf Step 1.}  Note that \eqref{eq:localization basic inequality} implies that 
\begin{align}\label{eq:localization step 1}
& 2    \frac{(n-\eta )\eta }{n}\mu(\eta)         -  2   \frac{(n-\he )\he }{n}       \mu(\he )  \le    \frac{(n-\he )\he }{n}        \wm (\he )    - \frac{(n-\eta )\eta }{n}   \wm (\eta  )
 +    C_1  \s\log(p) .
\end{align}
  Denote 
 $ \theta  =    \beta_{\eta}^*-\beta_{\eta+1}^*   . $ 
Note that $\theta ^\top  \Sigma \theta  \ge \Lambda_{\min}(\Sigma) \kappa^2$.
    By definition of $\mu(t)$, it follows that
    $$ \frac{(n-\eta )\eta }{n} \mu(\eta) - \frac{(n-\he )\he }{n} \mu(\he)  -     = \frac{\eta( \he -\eta )}{t} \theta ^\top  \Sigma \theta  \ge c( \he-\eta )\kappa^2 $$
    for some sufficiently small constant $c$. 
Therefore
\eqref{eq:localization step 1} implies that 
\begin{align}\label{eq:localization step 11}
 c( \he-\eta )\kappa^2   \le    \frac{(n-\he )\he }{n}        \wm (\he )    - \frac{(n-\eta )\eta }{n}   \wm (\eta  )
 +    C_1  \s\log(p) 
\end{align}
\
\\
{\bf Step 2.}
Note that 
\begin{align*} &  \bigg| \frac{(n-\he )\he }{n}        \wm (\he )    - \frac{(n-\eta )\eta }{n}   \wm (\eta  ) \bigg| 
\\
\le &\bigg|     \frac{(n-\he )\he }{n}     \wm _1 (\he ) -\frac{(n-\eta  )\eta  }{n}     \wm _1 (\eta )  \bigg| +    \bigg|  \frac{(n-\he )\he }{n}     \wm _2 (\he ) -\frac{(n-\eta  )\eta  }{n}     \wm _2 (\eta ) \bigg| 
  \end{align*} 
where 
\begin{align*} \wm _1 (t ) & = \frac{4}{ t  } \sum_{i=1}^ t    x_i  ^\top     (  \beta^*_{ (0,  t ] }  - \beta^*_ { ( t , n] }  )      \epsilon_i  +\frac{4}{n- t } \sum_{i=t  +1}^ n     x_i  ^\top     (  \beta^*_ { (t , n] }    - \beta^*_{ (0, t ] }    )      \epsilon_i,    \quad \text{and }
\\
\wm _2(t ) & = ( \beta^* _{ (0, t ] }  -     \beta _ { (t , n] } ^*  )  ^\top  (  \widehat \Sigma  _{ (0, t ] }   -\Sigma  )  (  \beta^*_{ (0, t ] }   -    \beta ^* _ { (t , n] }   )  +  ( \beta^*_ { (t , n] }    -     \beta _{ (0, t ] }  ^*  )  ^\top  (  \widehat \Sigma  _ { (t , n] }     -\Sigma  )  (  \beta^*_ { (t , n] }    -    \beta ^* _{ (0, t ] }    )  
  \end{align*} 
{\bf Step 3.}  We first  show that with high probability,
$$ \bigg|     \frac{(n-\he )\he }{n}     \wm _1 (\he ) -\frac{(n-\eta  )\eta  }{n}     \wm _1 (\eta )    \bigg| \le C_1\kappa \sqrt { \he -\eta } .$$
Observe that 
\begin{align}\nonumber 
&     \frac{(n-\he )\he }{n}     \wm _1 (\he ) -\frac{(n-\eta  )\eta  }{n}     \wm _1 (\eta )     
\\ \label{eq:localization deviation term 1}
= &    \bigg( \frac{(n-\he )\he }{n} \frac{1}{\he } \frac{ \eta }{\he }  - \frac{(n-\eta   )\eta   }{n}\frac{1}{\eta }  \bigg) \sum_{i=1}^\eta x_i^\top   \theta \epsilon_i 
\\ \label{eq:localization deviation term 2}
+ & \bigg( \frac{(n-\he )\he }{n} \frac{1}{t} \frac{\he }{t }  - \frac{(n-\eta   )\eta   }{n}\frac{1}{n-\eta }   \bigg) \sum_{i=\eta+1 }^{\widehat \eta}  x_i^\top   \theta \epsilon_i  
\\\label{eq:localization deviation term 3}
 + &\bigg( \frac{(n-\he )\he }{n} \frac{1}{n - \he } \frac{ \eta }{\he }   - \frac{(n-\eta   )\eta   }{n}\frac{1}{n-\eta }   \bigg) \sum_{i=\he+1 }^{ n }  x_i^\top   \theta \epsilon_i   .
\end{align}
For \eqref{eq:localization deviation term 1}, note that by \Cref{theorem:beta deviation}, 
$$ \bigg| \sum_{i=1}^\eta x_i^\top   \theta \epsilon_i  \bigg|  \le C_3 \kappa \sqrt \eta. $$
So
$$\eqref{eq:localization deviation term 1} \le C_3 \frac{\sqrt \eta }{\he }(\he -\eta )  \kappa \le C_3'  \sqrt{ (\he -\eta)}\kappa, $$
where $ \widehat \eta \ge \zeta n  \ge \zeta (\he -\eta)$ is used in the last inequality. 
\\
\\
To bound \eqref{eq:localization deviation term 1}, note that  for any $t \ge  C_2' \s\log(p)$, by \Cref{theorem:beta deviation}    with probability $1-n^{-3}$, 
$$  \bigg | \sum_{i=\eta+1 }^{  \eta +t }  x_i^\top   \theta \epsilon_i   \bigg|  \le C_4 \kappa \sqrt {t }.  $$
So by union bound, with probability $1-n^{-2}$, 
$$ \bigg | \sum_{i=\eta+1 }^{  \eta +t }  x_i^\top   \theta \epsilon_i   \bigg|  \le C_4 \kappa \sqrt {t } \text{ for all } t \ge  C_2' \s\log(p) . $$
So  with probability $1-n^{-2}$, 
$$\eqref{eq:localization deviation term 2} \le C_4' \kappa \sqrt { \he - \eta   }, $$
By a similar argument, with probability $1-n^{-2}$, 
$$\eqref{eq:localization deviation term 3} \le C_5 ' \kappa \sqrt { \he - \eta   }. $$
Thus
$$ \bigg| \frac{(n-\he )\he }{n}     \wm _1 (\he ) -\frac{(n-\eta  )\eta  }{n}     \wm _1 (\eta )  \bigg| \le C_6 \kappa \sqrt { \he - \eta   }. $$ 
 \
 \\
{ \bf Step 4.}
 By a similar argument as {\bf Step 3}, it follows that 
 $$\bigg|  \frac{(n-\he )\he }{n}     \wm _2 (\he ) -\frac{(n-\eta  )\eta  }{n}     \wm _2 (\eta ) \bigg| \le  C_7 \kappa \sqrt { \he - \eta   }. $$ 
 \
 \\
 { \bf Step 5.} { \bf Step 3} and { \bf Step 4} imply that  with probability at least $1-n^{-2}$,
\begin{align}\label{eq:localization step 5} 
 \bigg|   \frac{(n-\he )\he }{n}        \wm (\he )    - \frac{(n-\eta )\eta }{n}   \wm (\eta  ) \bigg| \le C_8 \kappa \sqrt { \he - \eta   }.
\end{align} 
This inequality together with \Cref{eq:localization step 11} implies that 
$$ c( \he-\eta )\kappa^2   \le   C_8 \kappa \sqrt { \he - \eta   } 
 +    C_1  \s\log(p)  .$$
 This immediate implies that 
 $$|\widehat \eta -\eta| \le C\frac{\s\log(p)}{\kappa^2 } . $$
\end{proof}

  \section{Additional Technical Results}

\begin{lemma} \label{lemma:beta bounded 1}
Suppose   \Cref{assume: model assumption} holds. Let $\I \subset [1,n] $. Denote 
 $ \kappa = \max_{k\in \{1,\ldots, K\} } \kappa_k,$ 
where $ \{ \kappa_k\}_{k=1}^K$ are defined in  \Cref{assume: model assumption}. Then 
   $$\|\beta^*_\I - \beta_ i\|_2  \le C\kappa  \le CC_\kappa,$$
  for some absolute constant $C$ independent of $n$.
\end{lemma} 
 \begin{proof}
 It suffices to consider $\I =[1,n]$  and  $\beta_ i= \beta_1 $ as the general case is similar. 
Denote $$ \Delta_k = \eta_{k+1}-\eta_k. $$ 
  Observe that  
 \begin{align*}
  \| \beta^*_{[1,n]}  - \beta_ 1^*  \|_2  =& 
  \left\|  \frac{1}{n} \sum_{i=1}^n \beta_i ^*  - \beta_1^*  
  \right \|_2 
 = \left\|  \frac{1}{n} \sum_{k=0}^{K} \Delta_k \beta_{\eta_k+ 1} ^*   -    \frac{1}{n}\sum_{k=0}^{K} \Delta_k  \beta_1 ^*
  \right \|_2  
  \\
 \le 
 &    \frac{1}{n} \sum_{k=0}^{K} \left\| \Delta_k  (\beta_{\eta_k+ 1} ^* -\beta_1 ^* ) \right\|_2    \le 
 \frac{1}{n}\sum_{k=0}^{K} \Delta_k  (K+1)  \kappa   \le  (K+1)\kappa . 
  \end{align*}
  By \Cref{assume: model assumption}, both $\kappa $ and $K $ bounded above. 
 \end{proof}
 
 \begin{lemma}\label{eq:cumsum upper bound}
 Let $t\in \I =(s,e] \subset [1,n]$. Denote 
 $ \kappa_{\max}= \max_{k\in \{1,\ldots, K\} } \kappa_k,$ 
where $ \{ \kappa_k\}_{k=1}^K$ are defined in  \Cref{assume: model assumption}.  Then   
 $$ \sup_{0< s < t<e\le n  }\| \beta_{ (s, t]}  ^* - \beta_{(t,e] }  ^* \|_2  \le C\kappa  \le CC_\kappa. $$ 
for some absolute constant $C$ independent of $n$. 
 \end{lemma}
 \begin{proof}
 It suffices to consider $(s,e]=(0,n]$, as the general case is similar.  
 Denote $$ \Delta_k = \eta_{k+1}-\eta_k. $$  Suppose that $\eta_{q} <t \le \eta_{q+1} $.
 Observe that 
 \begin{align*}
&  \| \beta^*_{(1,t]}  - \beta ^* _{(t,n]}  \|_2  
  \\
  =& 
  \left\|  \frac{1}{t } \sum_{i=1}^t \beta_i ^*  - \frac{1}{ n- t } \sum_{i=t+1}^n \beta_i ^*
  \right \|_2 
 \\
 = & \left\|  \frac{1}{t}  \left( \sum_{k=0}^{q-1} \Delta_k \beta_{\eta_k+ 1}  ^*  + (t-\eta_{q} )\beta^*_{\eta_q + 1} \right) -    \frac{1}{n-t } \left( \sum_{k=q+1}^{K}  \Delta_k \beta_{\eta_k+ 1} ^*   + ( \eta_{q+1}-t  )\beta^*_{\eta_q + 1} \right) 
  \right \|_2  
  \\
=
 & \left\|  \frac{1}{t}  \left( \sum_{k=0}^{q-1} \Delta_k (\beta ^* _{\eta_k+ 1}-\beta^*_{\eta_q+1} ) \right)  +  \beta^*_{\eta_q + 1} -    \frac{1}{n-t } \left( \sum_{k=q+1}^{K}  \Delta_k (  \beta ^* _{\eta_k+ 1}  -\beta ^* _{\eta_q+1} ) \right)   -  \beta^*_{\eta_q + 1} 
  \right \|_2 
  \\
  =
  &\left\|  \frac{1}{t}  \left( \sum_{k=0}^{q-1} \Delta_k (\beta_{\eta_k+ 1} ^* -\beta^*_{\eta_q+1} ) \right)    -    \frac{1}{n-t } \left( \sum_{k=q+1}^{K}  \Delta_k (  \beta_{\eta_k+ 1} ^*   -\beta_{\eta_q+1} ^*  ) \right)   
  \right \|_2 
  \\
  \le 
  &\frac{1}{t} \sum_{k=0}^{q-1} \Delta_k    K \kappa      +\frac{1}{n-t} \sum_{k=q+1}^{K} \Delta_k    K  \kappa 
  \le 2K\kappa  . 
  \end{align*}
  
 \end{proof}

\begin{lemma}\label{eq:equivalence Hypothesis} Let $0< \zeta<1/2$ is any constant sufficiently  close to $0$.
Denote 
$$ \kappa   = \max_{k\in \{1,\ldots, K\} } \kappa_k,$$
where $ \{ \kappa_k\}_{k=1}^K$ are defined in  \Cref{assume: model assumption}.
Suppose $\{ \beta^*_i\}_{i=1}^n$ is a collection of vectors satisfying \Cref{assume: model assumption} with $K>0$. Then it holds that 
$$ \max_{\zeta n \le t\le (1-\zeta) n} (\beta^*_\ot- \beta_\tn^*) ^\top  \Sigma  (\beta^*_\ot- \beta_\tn^* ) \ge c\kappa^2  $$
for some sufficiently small constant $c$.
In addition, suppose that $ \zeta $ is sufficiently small such that 
$$\zeta \le \eta_1^*<\ldots<  \eta_K^* \le 1-\zeta. $$
Then there exists $ q \in \{ 1,\ldots,K\}  $  such that
\begin{align}
\label{eq:variance cusum maximized at the change point}   
(\beta^*_{ (0,\eta_q] }- \beta_ { ( \eta_q,n ] }^*) ^\top  \Sigma (\beta^*_{ (0,\eta_q] }- \beta_ { ( \eta_q,n ] }^*)     \ge c'\kappa^2  ,  
\end{align}
where $ \eta_q= n\eta_q^*$ and $c'$ is some sufficiently small constant.
 \end{lemma}
\begin{proof}
Let 
$$ \alpha _ i = \Sigma^{1/2 } \beta_i^*  \quad \text{for all} \quad  1\le i \le n .$$
Then the  vector CUSUM statistics of $\{ \alpha_i\}_{i=1}^n$ is 
$$ \widetilde \alpha (t)  : = \sqrt { \frac{t(n-t)}{ n } } \bigg(  \frac{1}{t} \sum_{i=1}^t  \alpha_i -\frac{1}{n-t} \sum_{i=t+1}^n  \alpha_i   \bigg)  .$$
The time series $\{ \alpha_i\}_{i=1}^n$ is a collection of  piecewise constant vectors such that 
$$ \alpha_i = \alpha_{i'} \quad \text{for all} \quad \eta_{k}+1\le i\le i' \le \eta_{k+1} $$
and that the jump sizes at the change points satisfy 
$$ \| \alpha_{\eta_k+1} - \alpha_{\eta_k}\|_2 ^2 =(\beta ^* _{\eta_k+1} - \beta ^* _{\eta_k}) ^\top  \Sigma (\beta ^* _{\eta_k+1} - \beta ^* _{\eta_k}) \ge c_x \| \beta ^* _{\eta_k+1} - \beta ^* _{\eta_k} \|^2_2  =  c_x\kappa^2_k, $$
where $\kappa_k$ is defined in \Cref{assume: model assumption}. 
By Lemma S.12 of  \cite{wang2021optimal}, for any $k\in\{ 1,\ldots, K\}$, 
$$  \max_{\zeta n \le t\le (1-\zeta) n} \| \widetilde \alpha (t)   \| _2^2 \ge \frac{\Delta^2 }{48 n  }c_x \kappa^2_k   \ge    c_1  \kappa^2_k  n,  $$
where 
$$ \Delta : =\min_{1\le k \le K} (\eta_{k+1}-\eta_k) \ge c_2n $$
is used in the last inequality. 
Therefore 
$$  \max_{\zeta n \le t\le (1-\zeta) n} \| \widetilde \alpha (t)   \| _2^2    \ge    c_1  \kappa^2  n,  $$
 The desired result follows by observing that 
\begin{align*}  \| \widetilde \alpha (t)   \| _2^2  = \frac{t(n-t)}{ n } (\beta^*_\ot- \beta_\tn^*) ^\top  \Sigma  (\beta^*_\ot- \beta_\tn^* ) 
\end{align*}
and so 
\begin{align*}  \max_{\zeta n \le t\le (1-\zeta) n}   (\beta^*_\ot- \beta_\tn^*) ^\top  \Sigma  (\beta^*_\ot- \beta_\tn^* )  
=  & 
\max_{\zeta n \le t\le (1-\zeta) n}   \frac{ n}{t(n-t)  } \| \widetilde \alpha (t)   \| _2^2
\\ 
  \ge 
  & \max_{\zeta n \le t\le (1-\zeta) n}     \frac{ n}{t(n-t)  }     c_1  \kappa  ^2  n \ge c_1 \kappa ^2. 
\end{align*} 
For \Cref{eq:variance cusum maximized at the change point}, note that  by Proposition S.1 of  \cite{wang2021optimal},
$   \| \widetilde \alpha (t)   \| _2^2 $ as a function of $t$ is maximized at the change points. Say $\eta_q$ is the maximizer. Then 
$$  \| \widetilde \alpha (\eta_q )   \| _2^2  =  \max_{\zeta n \le t\le (1-\zeta) n} \| \widetilde \alpha (t)   \| _2^2   \ge    c_1  \kappa^2  n .  $$
Therefore 
\begin{align*} (\beta^*_{ (0,\eta_q] }- \beta_ { ( \eta_q,n ] }^*) ^\top  \Sigma (\beta^*_{ (0,\eta_q] }- \beta_ { ( \eta_q,n ] }^*)    
=  & 
 \frac{ n}{ \eta_q (n-\eta_q )  }   \| \widetilde \alpha (\eta_q )   \| _2^2 
  \ge   \frac{ n}{ \eta_q (n-\eta_q )  }  c_1  \kappa^2  n  \ge c'\kappa ^2,
\end{align*}  
where the last inequality follows that by \Cref{assume: model assumption} {\bf c},
$ \eta_q \asymp n $ and $ n-\eta_q \asymp n $.
\end{proof}  
 
\
\\
  
\section{Additional Numerical Results}\label{subsec:additional_num}

\begin{figure}
	\centering
	\vspace{-0.8cm}
	\begin{subfigure}[b]{0.4\textwidth}
		\centering
		\includegraphics[width=\textwidth, angle=270]{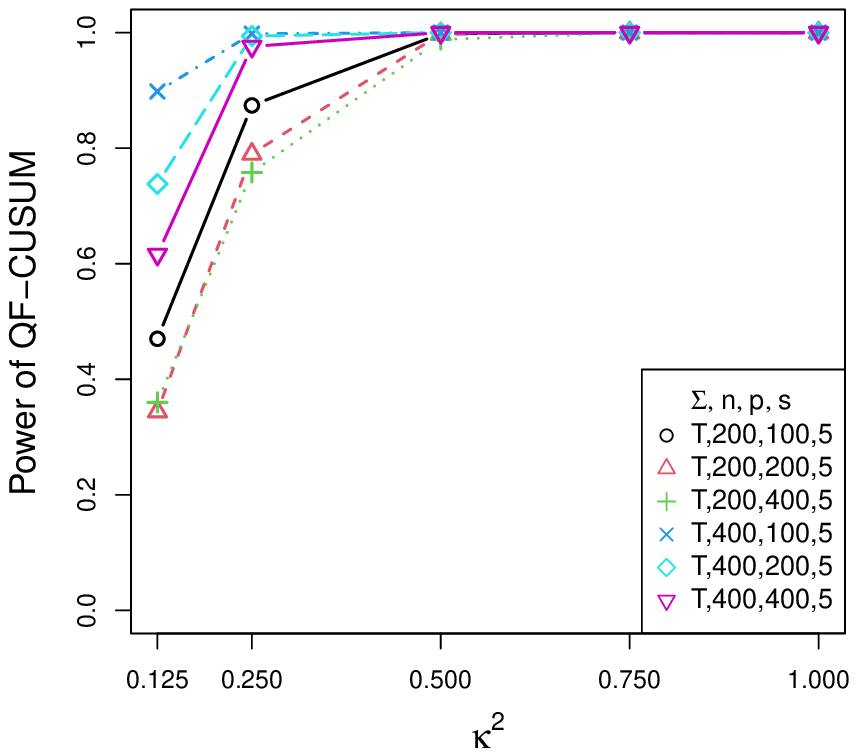}
	\end{subfigure}
    \hspace{-0.1cm} 
    \vspace{-0.8cm}
	\begin{subfigure}[b]{0.4\textwidth}
		\centering
		\includegraphics[width=\textwidth, angle=270]{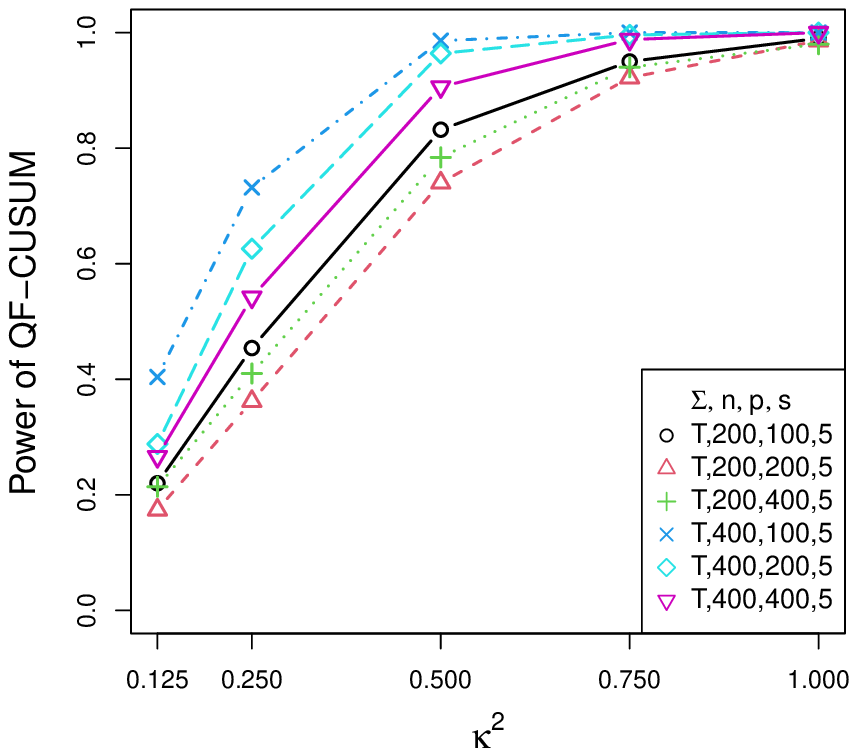}
	\end{subfigure}
 	\vspace{-0.8cm}
	\begin{subfigure}[b]{0.4\textwidth}
		\centering
		\includegraphics[width=\textwidth, angle=270]{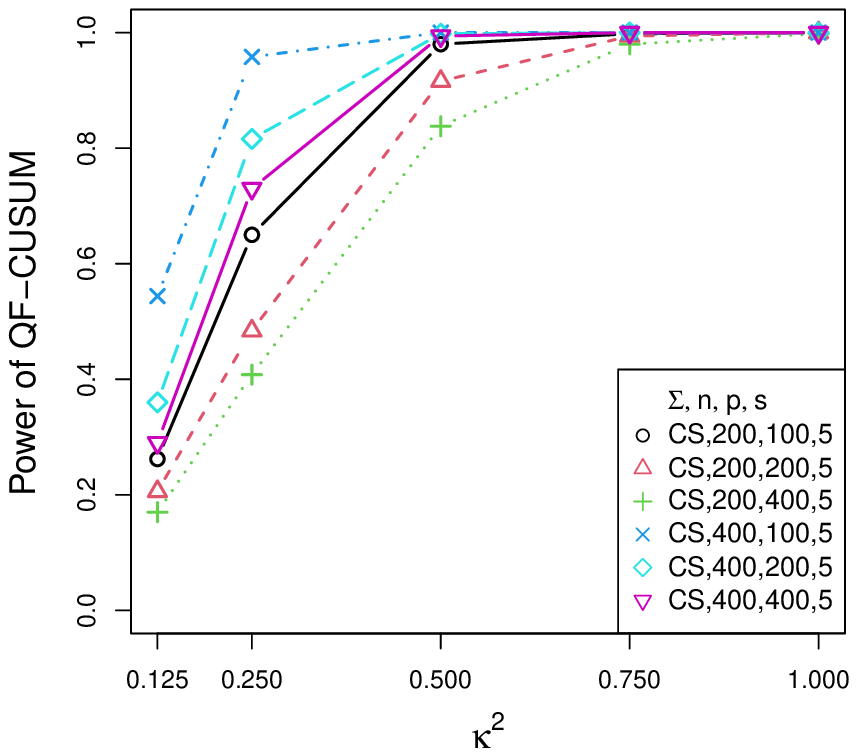}
	\end{subfigure}
    \hspace{-0.1cm}
	\begin{subfigure}[b]{0.4\textwidth}
		\centering
		\includegraphics[width=\textwidth, angle=270]{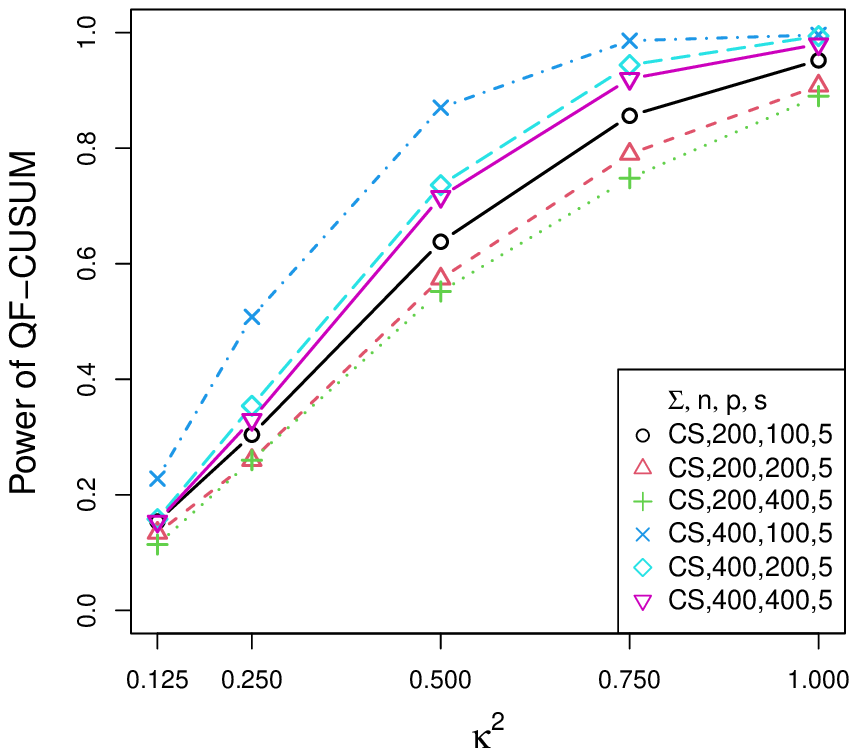}
	\end{subfigure}
    \vspace{-0.8cm}
	\begin{subfigure}[b]{0.4\textwidth}
		\centering
		\includegraphics[width=\textwidth, angle=270]{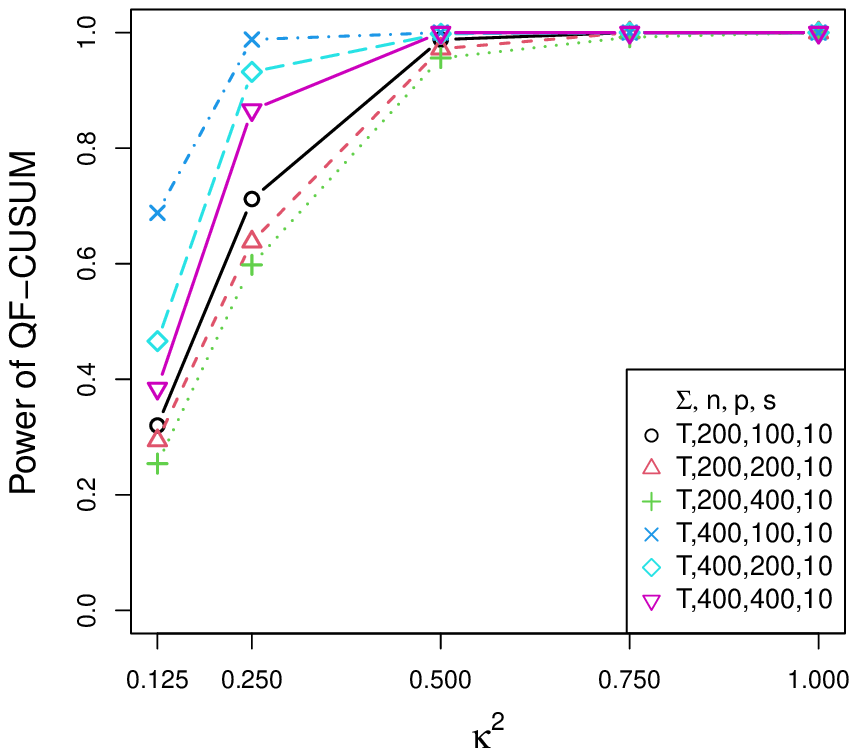}
	\end{subfigure}
    \hspace{-0.1cm}
	\begin{subfigure}[b]{0.4\textwidth}
		\centering
		\includegraphics[width=\textwidth, angle=270]{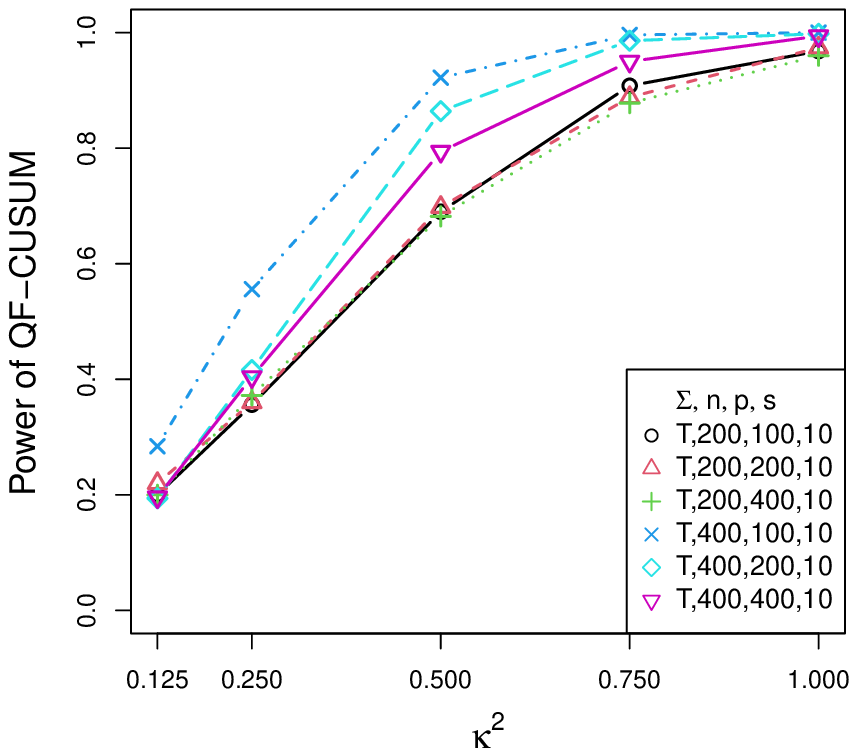}
	\end{subfigure}
    \vspace{-0.2cm}
	\begin{subfigure}[b]{0.4\textwidth}
		\centering
		\includegraphics[width=\textwidth, angle=270]{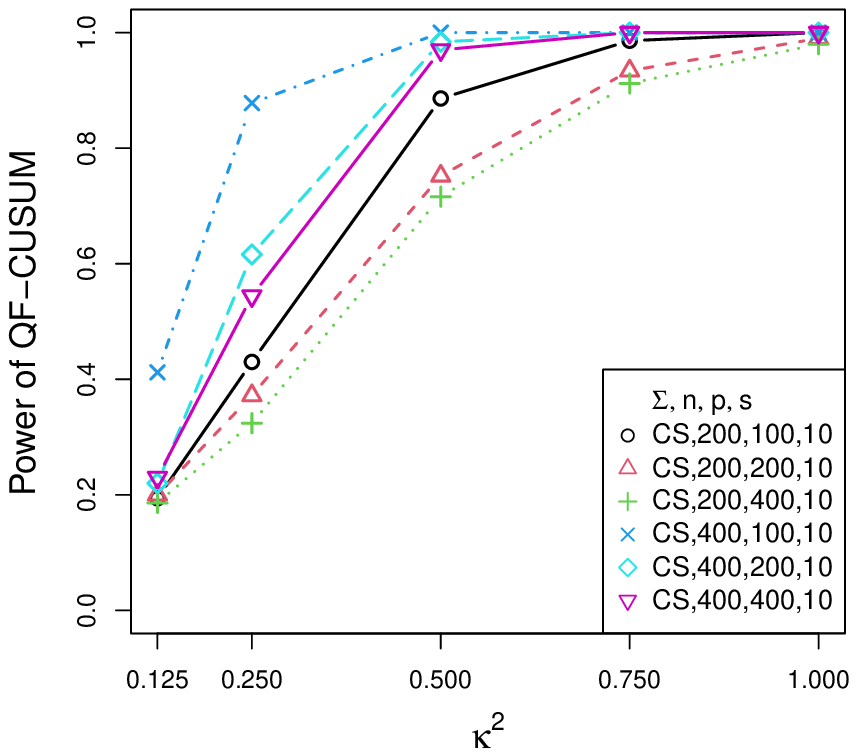}
	\end{subfigure}
    \hspace{-0.1cm}
	\begin{subfigure}[b]{0.4\textwidth}
		\centering
		\includegraphics[width=\textwidth, angle=270]{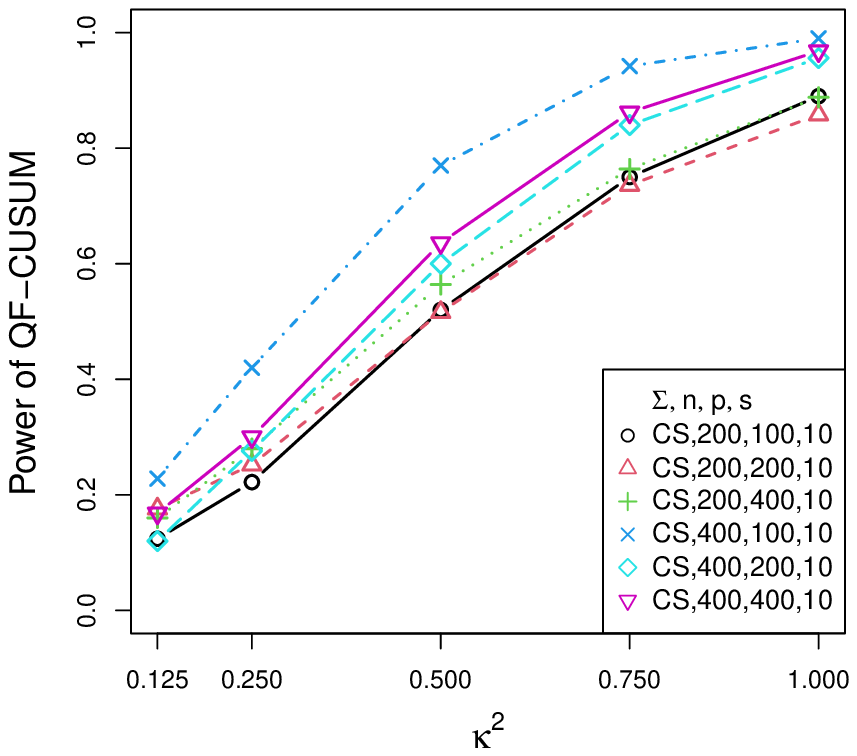}
	\end{subfigure}
    \vspace{-0.2cm}
	\caption{Power performance of QF-CUSUM under the single change-point case (left column) and multiple change-point case (right column) with AR temporal dependence.}
	\label{fig:power_AR}
\end{figure}
    
\begin{figure}
	\centering
	\vspace{-0.8cm}
	\begin{subfigure}[b]{0.4\textwidth}
		\centering
		\includegraphics[width=\textwidth, angle=270]{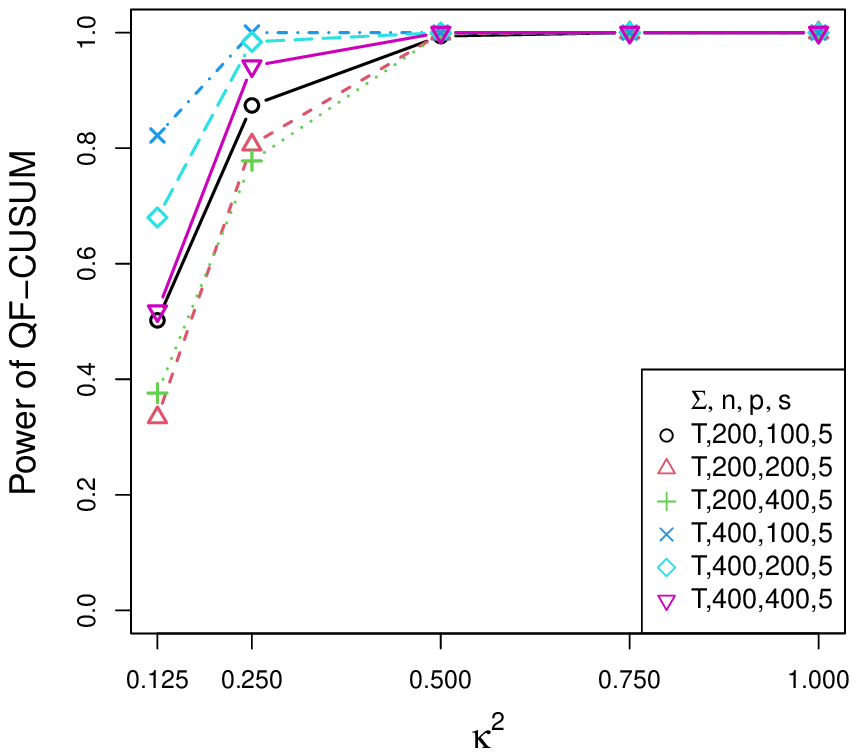}
	\end{subfigure}
    \hspace{-0.1cm} 
    \vspace{-0.8cm}
	\begin{subfigure}[b]{0.4\textwidth}
		\centering
		\includegraphics[width=\textwidth, angle=270]{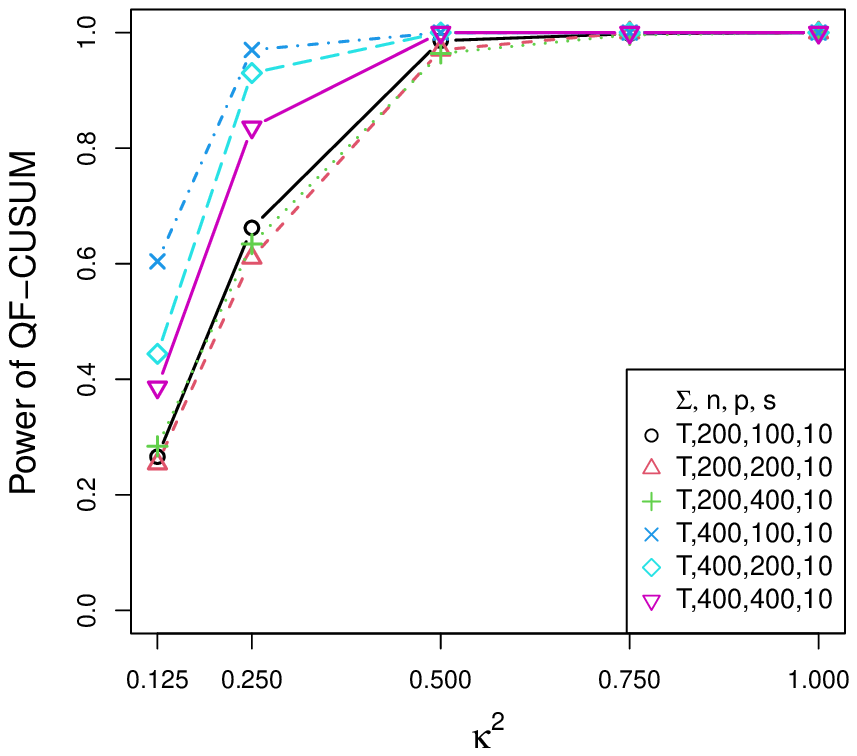}
	\end{subfigure}
    \vspace{-0.8cm}
	\begin{subfigure}[b]{0.4\textwidth}
		\centering
		\includegraphics[width=\textwidth, angle=270]{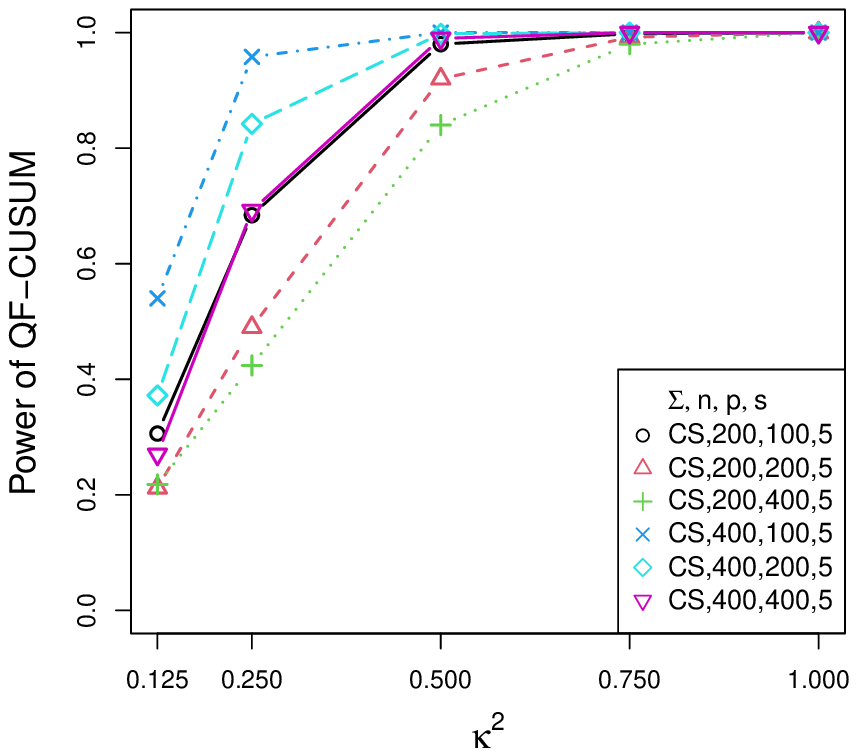}
	\end{subfigure}
    \hspace{-0.1cm}
	\begin{subfigure}[b]{0.4\textwidth}
		\centering
		\includegraphics[width=\textwidth, angle=270]{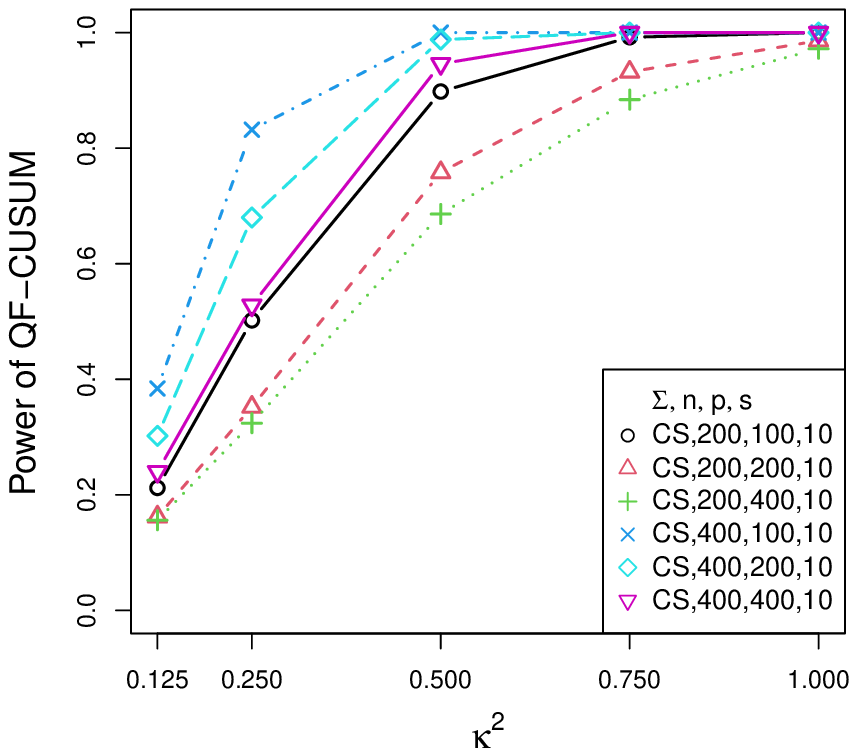}
	\end{subfigure}
	\vspace{-0.8cm}
	\begin{subfigure}[b]{0.4\textwidth}
		\centering
		\includegraphics[width=\textwidth, angle=270]{./EPS/MAonecp_symbetaUnequal_T_sp5}
	\end{subfigure}
	\hspace{-0.1cm}
	\begin{subfigure}[b]{0.4\textwidth}
		\centering
		\includegraphics[width=\textwidth, angle=270]{./EPS/MAonecp_symbetaUnequal_T_sp10}
	\end{subfigure}
	\vspace{-0.2cm}
	\begin{subfigure}[b]{0.4\textwidth}
		\centering
		\includegraphics[width=\textwidth, angle=270]{./EPS/MAonecp_symbetaUnequal_CS_sp5}
	\end{subfigure}
    \hspace{-0.1cm}
	\begin{subfigure}[b]{0.4\textwidth}
		\centering
		\includegraphics[width=\textwidth, angle=270]{./EPS/MAonecp_symbetaUnequal_CS_sp10}
	\end{subfigure}
    \vspace{-0.2cm}
	\caption{Power performance of QF-CUSUM under the single change-point case (left column) and multiple change-point case (right column) with MA temporal dependence.}
	\label{fig:power_MA}
\end{figure}

\end{document}